\documentclass[english,10pt,a4paper,leqno]{article}
\usepackage{amsmath,amstext, amsthm, amssymb}
\usepackage[hypertex]{hyperref}

\usepackage{srcltx}
\usepackage[totalheight=24 true cm, totalwidth=17 true cm]{geometry}
\usepackage[english]{babel}
\usepackage[all]{xy}
\usepackage{color}
\usepackage{tocvsec2}
\newtheorem{thm}{Theorem}[section]
\newtheorem{cor}[thm]{Corollary}
\newtheorem{lemma}[thm]{Lemma}
\newtheorem{prop}[thm]{Proposition}

\newtheorem{defn}[thm]{Definition}
\newtheorem{defi}[thm]{Definition} %this is redundant

\theoremstyle{remark}
\theoremstyle{definition}

\newtheorem{rmk}[thm]{Remark}
\newtheorem{rem}[thm]{Remark} %this is redundant

\newtheorem{exa}[thm]{Example}
\numberwithin{equation}{section}

\def\beq{\begin{equation}}
\def\eeq{\end{equation}}

\def\crash#1{}

\def\N{{\mathbb N}}
\def\Z{{\mathbb Z}}
\def\Q{{\mathbb Q}}
\def\R{{\mathbb R}}
\def\C{{\mathbb C}}

\def\I{{\mathbb I}}

\def\F{{\mathbb F}}

\def\l{\left}
\def\r{\right}
\def\p{\prime}

\def\cE{{\mathcal E}}

\def\cL{{\mathcal L}}

\def\wtilde{\widetilde}

\def\ul{\underline}

\def\veps{\varepsilon}
\def\a{\alpha}
\def\be{\beta}
\def\de{\delta}
\def\ga{\gamma}
\def\sg{\sigma}

\def\la{\lambda}

\def\deg{\mathop{\rm deg}}

\def\s{\sigma}

\def\Gl{\operatorname{Gl}}
\def\Sl{\operatorname{Sl}}

\def\Aut{\operatorname{Aut}}
\def\ida{\mathfrak{a}}
\def\idb{\mathfrak{b}}
\def\Alg{\operatorname{Alg}}
\def\spec{{\operatorname{Spec}}}
\def\q{\mathfrak{q}}

\def\quot{\operatorname{Quot}}
\def\sdim{\sigma\text{-}\dim}
\def\strdeg{\sigma\text{-}\operatorname{trdeg}}
\def\Ga{\mathbf{G}_a}
\def\Gm{\mathbf{G}_m}
\def\trdeg{\operatorname{trdeg}}
\def\id{\operatorname{id}}

\def\<{\langle}
\def\>{\rangle}

\def\ds{{\delta\sigma}}

\def\sgal{\sigma\text{-}\operatorname{Gal}}
\def\diag{\operatorname{diag}}
\author{Lucia Di Vizio, Charlotte Hardouin and Michael Wibmer}

\title{Difference algebraic relations among solutions of linear differential equations
\footnotetext{Lucia Di Vnizio,
Laboratoire de Math\'ematiques UMR8100, UVSQ,
45 avenue des \'Etats-Unis
78035 Versailles cedex, France.
{e-mail: {\tt divizio@math.cnrs.fr}.}}
\footnotetext{Charlotte Hardouin, {Institut de Math\'{e}matiques de Toulouse,}
{118 route de Narbonne,
31062 Toulouse Cedex 9, France.}
{e-mail: {\tt hardouin@math.ups-tlse.fr}}.}
\footnotetext{Michael Wibmer, {Lehrstuhl f\"{u}r Mathematik (Algebra),} {RWTH Aachen}, {
52056 Aachen, Germany.} {e-mail: {\tt michael.wibmer@matha.rwth-aachen.de}.}}}

\begin{document}
%\makeatletter\c@page=267\makeatother
\maketitle

\begin{abstract}
We extend and apply the Galois theory of linear differential equations equipped with the action of an endomorphism. The Galois groups in this Galois theory are difference algebraic groups and we use structure theorems for these groups to characterize the possible difference algebraic relations among solutions of linear differential equations. This yields tools to show that certain special functions are difference transcendent. One of our main results is a characterization of discrete integrability of linear differential equations with almost simple usual Galois group, based on a structure theorem for the Zariski dense difference algebraic subgroups of almost simple algebraic groups, which is a schematic version,
in characteristic zero, of a result  due to Z. Chatzidakis, E. Hrushovski and Y. Peterzil.
\end{abstract}

\bibliographystyle{alpha}

%%%%%%%%%%%%%%%%%%%%%%%%%%%%%%%%%%%%%%%%%%%%%%%%%%%%%%%%%%%%%%%%%%%%
%%%%%%%%%%%%%%%%%%%%%%%%%%%%%%%%%%%%%%%%%%%%%%%%%%%%%%%%%%%%%%%%%%%%
%%%%%%%%%%%%%%%%%%%%%%%%%%%%%%%%%%%%%%%%%%%%%%%%%%%%%%%%%%%%%%%%%%%%
%%%%%%%%%%%%%%%%%%%%%%%%%%%%%%%%%%%%%%%%%%%%%%%%%%%%%%%%%%%%%%%%%%%%
%%%%%%%%%%%%%%%%%%%%%%%%%%%%%%%%%%%%%%%%%%%%%%%%%%%%%%%%%%%%%%%%%%%%
%%%%%%%%%%%%%%%%%%%%%%%%%%%%%%%%%%%%%%%%%%%%%%%%%%%%%%%%%%%%%%%%%%%%
%\input{abstract.tex}

\renewcommand{\labelenumi}{{\rm (\roman{enumi})}}
\renewcommand{\labelenumii}{{\rm (\alph{enumii})}}

\maxtocdepth{subsection}
%\tableofcontents

\section{Introduction}

In \cite{articleone} the authors developed a Galois theory for linear differential equations equipped with the action of an endomorphism $\s$. Such actions occur, for example, when considering linear differential equations depending on a parameter $\alpha$, in which case the action of $\s$ is determined by $\s(\alpha)=\alpha+1$. In the $p$-adic analysis of differential equations one may choose $\s$ as a Frobenius lift.

The Galois groups in this Galois theory are linear difference algebraic groups, i.e., subgroups of the general linear group defined by equations which involve powers of $\s$. These Galois groups measure not only the algebraic relations among the solutions of a linear differential equation, but also the difference algebraic relations, i.e., the algebraic relations between all the transforms of the solutions under $\s$. For example, the difference algebraic relation $xJ_{\alpha+2}(x)-2(\alpha+1)J_{\alpha+1}(x)+xJ_{\alpha}(x)=0$ satisfied by the Bessel function $J_\alpha(x)$ will be witnessed by our new Galois group, whereas the usual (differential) Galois group of the Bessel equation does not take it into account.

In this article we prove structural results about difference algebraic groups to apply them to characterize the possible difference algebraic relations among various types of linear differential equations. Thereby we often obtain criteria to decide if certain functions are transformally (i.e., difference algebraically) independent.

This Galoisian approach to transformal independence is analogous to the approaches to differential independence of special functions via the Galois theories in \cite{cassisinger} and \cite{HardouinSinger}. See \cite{DiVizio:ApprocheGaloisienneDeLaTranscendanceDifferentielle} for a survey and \cite{Arreche:AGaloisTheoreticProofOfTheDifferentialTranscendenceOfTheIncompleteGammaFunction} for a recent example.

\medskip

The usual Galois theory of linear differential equations (see e.g. \cite{vdPutSingerDifferential} for an introduction) is aimed at studying the algebraic relations among the solutions of a linear differential system $\de(y)=Ay$, $A\in K^{n\times n}$ over a differential field $K$, i.e., a field equipped with a derivation $\de\colon K\to K$.
In order to study the difference algebraic properties of the solutions, we enrich the situation with the action of a difference operator $\s$.
Our base field $K$ is a differential--difference field, i.e., it is equipped with a derivation $\de\colon K\to K$ and an endomorphism $\s\colon K\to K$ such that $\de$ and $\s$ commute (up to a convenient factor). For example, $K$ could be the field $\C(\alpha,x)$ with derivation $\de=\frac{d}{dx}$ and endomorphism $\s$ given
by $\s(f(\alpha,x))=f(\alpha+1,x)$. One can also consider the $\ds$-field $K=\C(x)$ endowed with the
derivation $\de(x)=1$ and the endomorphism $\s(x)=x^2$.
 To study the difference algebraic relations among the solutions we have to adjoin not only the solutions to the base field but also all their transforms under $\s$. As in the usual Galois theory of linear differential equations, it is crucial to do this
without enlarging the $\de$-constants $K^\de:=\{a\in K|\ \de(a)=0\}$. A differential--difference field obtained in this way is called a $\s$-Picard-Vessiot extensions of $K$ for $\de(y)=Ay$. The automorphism group of a $\s$-Picard-Vessiot extension is naturally a linear difference algebraic group over the difference field $k=K^\de$, i.e., a subgroup of the general linear group $\Gl_{n,k}$ defined by algebraic difference equations. These groups play the role of Galois groups in the Galois theory developed in \cite{articleone} and they encode valuable information about the difference algebraic relations among the solutions. For example, the difference--transcendence degree of a $\s$-Picard-Vessiot extension is the dimension (in the sense of difference algebra) of the corresponding Galois group.

The basic idea of this article is to apply, via the Galois theory from  \cite{articleone}, structural results for difference algebraic groups to study the difference algebraic relations among the solutions of linear differential equations. We will show that often some understanding of the usual Galois group of $\de(y)=Ay$ is already sufficient to obtain an explicit description of the possible difference algebraic relations among the solutions. For example, if $A$ is a diagonal matrix, one knows that the usual Galois group is an algebraic subgroup of $\Gm^n$, and since the difference algebraic subgroups of $\Gm^n$ are well--understood, a description of the possible difference algebraic relations among the solutions can be obtained.

Our explicit descriptions of the possible difference algebraic relations among the solutions often lead to criteria for transformal independence. For example, we prove the following statement (cf. Corollary \ref{cor:forintro}):

\begin{center}
\begin{minipage}{15cm}
\textit{Let $L$ be a differential--difference field extension of $K$ such that $L^\de=K^\de$ is algebraically closed and $\s\colon K\to K$ is surjective. For $Y\in\Gl_n(L)$ with $\de(Y)=AY$ for some $A\in K^{n\times n}$, let $T\subset L$ be a transcendence basis of $K(Y)|K$ and assume that the usual Galois group of $\de(y)=Ay$ is an almost simple algebraic group. If, for every $d\geq 1$, the linear differential equation $\de(B)+BA=\s^d(A)B$ over $K$ has no non-zero solution $B$ which is algebraic over $K$, then $T$ is transformally independent over $K$, i.e., $T,\s(T),\ldots$ are algebraically independent over $K$.
}
\end{minipage}
\end{center}

Applying the above criterion to the Airy equation $\de^2(y)-xy=0$ yields the transformally independence of the Airy functions (see Corollary \ref{cor:Airy}):

\begin{center}
\begin{minipage}{15cm}
\textit{Let $\operatorname{Ai}(x)$ and $\operatorname{Bi}(x)$ be two $\C$-linearly independent solutions of the Airy equation. Then the functions
$\operatorname{Ai}(x),\operatorname{Bi}(x),\operatorname{Bi}'(x), \operatorname{Ai}(x+1),\operatorname{Bi}(x+1),\operatorname{Bi}'(x+1),\operatorname{Ai}(x+2),\ldots$ are algebraically independent over $\C(x)$.
}
\end{minipage}
\end{center}

\medskip

Let us describe the content of the article in more detail. In Section 2 we fix the notation and recall the basic definitions and results from the $\s$-Galois theory of linear differential equations developed in \cite{articleone}. In Section 3 we study three different classes of linear differential systems. First, systems of the form
$$\de(y_1)=b_1,\ldots,\de(y_n)=b_n,$$ where the usual Galois group is a subgroup of $\Ga^n$, secondly, systems of the form
$$\de(y_1)=a_1y_1,\ldots,\de(y_n)=a_ny_n,$$ where the usual Galois group is a subgroup of $\Gm^n$, and finally an equation of the form
$$\de(y)=ay+b,$$ where the usual Galois group is a subgroup of $\Ga\rtimes\Gm$.

In Section 4, we investigate the inverse problem for difference algebraic subgroups of
the additive group $\Ga$. The inverse problem for a continuous parameter
was discussed in \cite{Singerinv}, where it is proved that, having $\Ga$ or $\Gm$ as quotient
is an  obstruction for a differential algebraic group to be a parametrized Galois
group in the sense of \cite{cassisinger}. In this section, we show that unlike the
continuous parameter case, one can realize $\Ga$ as well
as its non reduced subgroups, as $\s$-Galois group.

In Section 5 we introduce and study discrete integrability also called discrete isomonodromy. A linear differential system $\de(y)=Ay$ over a differential--difference field $K$ is called $\s^d$-integrable for an integer $d\geq 0$, if there exists $B\in\Gl_n(K)$ such that $\de(B)+BA=\s^d(A)B$. Note that $\s$-integrability is a necessary condition for the existence of a fundamental solution matrix $Y$ of $\de(y)=Ay$ with $\s(Y)=BY$ derived from the identity $\s(\de(Y))=\de(\s(Y))$. The main result here is a characterization of $\s^d$-integrability via the $\s$-Galois group (Theorem \ref{thm:integra}): The differential system $\de(y)=Ay$ is $\s^d$-integrable if and only if up to extension of the constants and conjugation inside $\Gl_{n}$, the associated $\s$-Galois group $G\leq\Gl_{n}$ satisfies $\s^d(g)=g$ for all $g\in G$.  In \S \ref{subsec:integraexamples}, we give examples
of $\s^d$-integrability as contiguity relations, Frobenius structure  for $p$-adic differential equations and
Lax pairs for lattices in Mathematical physics.

Then, in Section 6, we combine this characterization of $\s^d$-integrability with a classification result for the Zariski dense difference algebraic subgroups of almost simple algebraic groups  to obtain a strong dichotomy for the possible difference algebraic relations among the solutions of linear differential systems $\de(y)=Ay$, whose usual Galois group is an almost simple algebraic group. The philosophy of the result is that either the system is $\s^d$-integrable for some $d\geq 1$ or, there are no (proper) difference algebraic relations among the solutions, i.e., all the difference algebraic relations can be derived from algebraic relations.
We present two versions of this result. One, for simple algebraic groups (Proposition \ref{prop:simpleintegra}) and another one for almost simple algebraic groups (Theorem \ref{thm:almostintegra}).

The relevant structure results for difference algebraic groups have been collected in an appendix.
Difference algebraic groups are the group objects in the category of difference varieties.
A difference variety in the classical sense (see Section 2.6 in \cite{Levin}) is determined by its points in difference fields. There is a one--to--one correspondence between the difference subvarieties of affine $n$-space and the so--called \emph{perfect} difference ideals in the difference polynomial ring in $n$ difference variables (Theorem 2.6.4 in \cite{Levin}).
In this article we have to employ a more general notion of difference variety, which one might term a difference scheme. We consider points in arbitrary difference algebras. There is a one--to--one correspondence between the difference subschemes of affine $n$-space and \emph{all} difference ideals in the difference polynomial ring in $n$ difference variables. The difference schemes corresponding to the classical difference varieties we call perfectly $\s$-reduced.

Difference algebraic groups which are not perfectly $\s$-reduced occur rather frequently as $\s$-Galois groups. For example, if the system $\de(y)=Ay$ has a solution $z$ whose coordinates $z_1,\ldots,z_n$ are algebraic over $K$, then a $\s$-Picard-Vessiot extension $L$ for $\de(y)=Ay$ will contain all the $\s^j(z_i)$ ($j\geq 1$, $1\leq i\leq n$). Unless $\s$ is of a rather special nature the field extension generated by all the $\s^j(z_i)$ will be an infinite algebraic extension of $K$. Thus the relative algebraic closure of $K$ in $L$ is of infinite degree over $K$. This shows that the $\s$-Galois group of $L|K$ is not perfectly $\s$-reduced since a perfectly $\s$-reduced difference algebraic group can have only finitely many components. (Cf. Lemma 6.4 in \cite{articleone}.)

The major bulk of the appendix is concerned with  a classification result for the Zariski dense difference algebraic subgroups of almost simple algebraic groups.  In  \cite{ChatHrusPet}, the authors proved this result for perfectly $\s$-reduced difference algebraic groups. We show that a naive generalization of \cite{ChatHrusPet} does not hold. Thus,  we adapt their arguments to our schematic  framework and  introduce  new  techniques
to obtain Theorems   \ref{thm:classsimple} and \ref{thm:classalmostsimple} as structure theorem.
(See Section \ref{subsubsec:simple} for a detailed discussion).  A further benefit of our version is that it is expected to be applicable (via the Galois theory in \cite{OvchinnikovWibmer:SGaloisTheoryOfLinearDifferenceEquations}) to linear difference equations as well.

%%%%%%%%%%%%%%%%%%%%%%%%%%%%%%%%%%%%%%%%%%%%%%%%%%%%%%%%%%%%%%%%%%%%%%%%%%%%%%%%%%%%%%%%%%%%%%%%
%%%%%%%%%%%%%%%%%%%%%%%%%%%%%%%%%%%%%%%%%%%%%%%%%%%%%%%%%%%%%%%%%%%%%%%%%%%%%%%%%%%%%%%%%%%%%%%%
%\input{preliminaries}

\section{Difference Galois theory of linear differential equations}
\label{sec:recallarticleone}

In this section, we quickly recall some basic facts of difference/differential
algebra as well as some very basic notions of difference algebraic geometry,
mainly in the affine case. We recall also  some fundamental results from
\cite{articleone}, that we will need in the sequel. More precise notions will be recalled along the paper
when needed.
We largely use standard notations of difference and differential algebra as can
be found in \cite{Cohn:difference}, \cite{Levin} and \cite{Kolchin:differentialalgebraandalgebraicgroups}.

\subsection{Differential algebra} \label{subsec:differentialalgebra}

All rings considered in this work are commutative with identity and contain the field of rational numbers.
In particular, all fields are of characteristic zero.
A differential ring (or $\de$-ring for short) is a ring $R$ together with a derivation $\de:R\rightarrow R$.
A $\de$-ring $R$ is said to be $\de$-simple if it does not contain any proper $\de$-ideals,
i.e, proper ideals invariant under the action of $\de$.
The ring of $\de$-constants is $R^\de=\{r\in R| \ \de(r)=0\}$.
We will use more than once the following lemma on $\de$-simple $\de$-rings:

\begin{lemma}[{Lemma 2.3 in \cite{articleone}}]\label{lemma:simple}
Let $R$ be a $\de$-simple $\de$-ring, $k:=R^\de$ and $S$ a $k$-algebra,
considered as a constant $\de$-algebra, i.e., $S^\de=S$. Then $(R\otimes_k S)^\de=S$ and the assignments
$\ida\mapsto R\otimes_k \ida$ and $\idb\mapsto S\cap\idb$ define mutually inverse bijections between
the set of ideals of $S$ and the set of $\de$-ideals of $R\otimes_k S$.
In particular, every $\de$-ideal $\idb$ of $R\otimes_k S$ is generated by $\idb\cap S$ as an ideal.
\end{lemma}

\subsection{Difference algebra}
\label{subsec:differencealgebra}

A difference ring (or $\s$-ring for short) is a ring $R$ together with a ring endomorphism $\s\colon R\rightarrow R$.
We do not assume that $\s$ is an automorphism or injective. A $\s$-ring with $\s$ injective is called $\s$-reduced.
If $\s$ is an automorphism, the $\s$-ring is called inversive. The ring of $\s$-constants is $R^\s=\{r\in R|\ \s(r)=r\}$.
\par
The $\s$-ideals, i.e., the $\s$-invariant ideals, of a $\s$-ring $R$ can be classified through many
properties that are specific to difference algebra.  A $\s$-ideal $\ida \subset R$ is called reflexive if $\s^{-1}(\ida) \subset \ida$, which is equivalent to say
that $R/ \ida$ is $\s$-reduced. Then, a $\s$-ideal $\q$ of $R$ is said to
be $\s$-prime if it is prime and reflexive.
Finally, a $\s$-ideal  $\ida$ is called perfect if
$\s^{\alpha_1}(r)\cdots\s^{\alpha_n}(r)\in\ida$ implies $r\in \ida$,
for all $r\in R$, and $n,\a_1,\dots,\a_r\in \Z$, with $n\geq 1$ and $\alpha_1,\ldots,\alpha_n\geq 0$.
\par
Given two $\s$-rings $(R,\s)$ and ($R^\p,\s^\p)$,
a morphism $\psi:R\rightarrow R'$ of $\s$-rings is a morphism of rings such that $\psi\s=\s'\psi$.
A $\s$-field is a field that is also a $\s$-ring. A $\s$-field extension is an extension of $\s$-fields
such that the inclusion is also a morphism of $\s$-rings.
\par
Let $k$ be a $\s$-field and $R$ a $k$-$\s$-algebra, i.e., a $k$-algebra such that
the morphism $k\to R$ is a morphism of $\s$-rings.
We say that $R$ is $\s$-separable (resp. perfectly $\s$-separable) over $k$ if the zero ideal in
 $R\otimes_k k'$ is reflexive (resp. perfect), for every $\s$-field extension $k'$ of $k$.
\par
If $R$ is a $k$-$\s$-algebra over a $\s$-field $k$ and $B$ a subset of $R$,
then $k\{B\}_\s$ denotes the smallest $k$-$\s$-subalgebra of $R$ that contains $B$.
If $R=k\{B\}_\s$ for some finite subset $B$ of $R$, we say that $R$ is finitely $\s$-generated over $k$.
If $K|k$ is an extension of $\s$-fields and $B\subset K$, then $k\<B\>_\s$ denotes the smallest $\s$-field
extension of $k$ inside $K$ that contains $B$.
The $k$-$\s$-algebra $k\{x\}_\s=k\{x_1,\ldots,x_n\}_\s$ of $\s$-polynomials over $k$
in the $\s$-variables $x_1,\ldots,x_n$ is the polynomial ring over $k$ in the
countable set of algebraically independent variables $x_1,\ldots,x_n,\s(x_1),\ldots,\s(x_n),\ldots,$ with an action of $\s$
as suggested by the names of the variables.

\begin{defn}[Definition 4.1.7 in \cite{Levin}]
Let $L|K$ be a $\s$-field extension. Elements $a_1,\ldots,a_n\in L$ are called \emph{transformally
(or $\s$-algebraically) independent over $K$} if the elements $a_1,\ldots,a_n,\s(a_1),\ldots,\s(a_n),\ldots$
are algebraically independent over $K$. Otherwise, they are called transformally dependent over $K$.
A \emph{$\s$-transcendence basis of $L$ over $K$} is a maximal transformally independent over $K$ subset of $L$.
Any two $\s$-transcendence bases of $L|K$ have the same cardinality and so we can define the
\emph{$\s$-transcendence degree of $L|K$}, or
$\strdeg(L|K)$ for short,
as the cardinality of any $\s$-transcendence basis of $L$ over $K$.
\end{defn}

\subsection{Difference algebraic geometry}

In \cite{articleone}, we work with the formalism of difference group schemes.
They are defined as follows:

\begin{defn} Let $k$ be a $\s$-field. A \emph{group $k$-$\s$-scheme} is a (covariant)
functor $G$ from the category of $k$-$\s$-algebras to the category of groups which is representable
by a  $k$-$\s$-algebra.
I.e., there exists a $k$-$\s$-algebra $k\{G\}$ such that
$G\simeq\Alg_k^\s(k\{G\},-)$, where $\Alg_k^\s$ stands for morphisms of $k$-$\s$-algebras. If
$k\{G\}$ is finitely $\s$-generated over $k$, we say that $G$ is a $\s$-algebraic group over $k$.
\par
By a \emph{$\s$-closed subgroup $H$ of $G$} we mean a group $k$-$\s$-scheme $H$ such that $H(S)$
is a subgroup of $G(S)$ for every $k$-$\s$-algebra $S$. We call $H$ \emph{normal} if $H(S)$ is a
normal subgroup of $G(S)$ for every $k$-$\s$-algebra $S$.
\end{defn}

As in the classical setting, the Yoneda lemma implies that the algebra $k\{G\}$ is
a $k$-$\s$-Hopf algebra, i.e.,  a $k$-$\s$-algebra equipped with the structure of a Hopf algebra
over $k$ such that the Hopf algebra structure maps are morphisms of difference rings.
It also follows immediately that  the category of group $k$-$\s$-schemes is anti-equivalent
to the category of $k$-$\s$-Hopf algebras.
We are not giving further details on this point
(see the Appendix A.2 and A.8 in \cite{articleone}).

\par
A $\s$-closed subgroup $H$ of $G$ corresponds to a $\s$-Hopf-ideal $\mathbb I(H)$ of $k\{G\}$, i.e.,
a Hopf-ideal which is a difference ideal,
that we will call the vanishing ideal of $H$ inside $G$. Then $k\{H\}\cong k\{G\}/ \mathbb I(H)$.
Normal $\s$-closed subgroups of $G$ correspond to normal $\s$-Hopf-ideals, i.e., $\s$-Hopf-ideals which are normal Hopf-ideals.

\par
Difference  properties of the vanishing ideal of a difference group scheme $G$ translate
into geometric properties of the difference group scheme $G$. For instance, a difference analog of
irreducibility is that the zero ideal of $k\{G\}$ is $\s$-prime. The difference analog of
reduced scheme is more subtle and admits two definitions. First of all, we say that a group $k$-$\s$-scheme $G$ is perfectly $\s$-reduced if the zero ideal in $k\{G\}$ is perfect.
Perfectly $\s$-reduced difference schemes correspond to difference varieties in the classical sense (\cite{Cohn:difference}, \cite{Levin}) or in the model theoretic sense (\cite{chatdifffield}, \cite{ChatHrusPet}, \cite{Hrushovski:elementarytheoryoffrobenius},...), where it suffices to focus on the solution set of a system of difference equations with values in a sufficiently big field, i.e., a $\s$-closed field.
In the present work, we encounter a broader class of difference schemes so that we need a weaker notion.
Thus, we say that a a group $k$-$\s$-scheme $G$ is $\s$-reduced if the zero ideal in $k\{G\}$ is reflexive.
If one requires that these two notions remain unchanged by base field extension, we
say absolutely perfectly $\s$-reduced and absolutely $\s$-reduced. Of course, absolutely (resp. absolutely perfectly) $\s$-reduced
difference group schemes are in one to one correspondence with (resp. perfectly) $\s$-separable $k$-$\s$-Hopf algebras.

\medskip
One can attach to any affine group $k$-scheme $\wtilde G$  over $k$, a
group $k$-$\s$-scheme, denoted $[\s]_k\wtilde G$  (see \S A.4 in \cite{articleone}).
Let us recall the construction.
Let $d \geq 0$ be a positive integer. For any $k$-algebra $R$,  we set ${}^{\s^d}R=R\otimes_k k$, where the tensor product
is formed by using $\s^d\colon k\to k$ on the right hand side. We consider ${}^{\s^d}R$ as $k$-algebra via the right factor.
We define  \[R_d=R\otimes_k{}^{\s}R\otimes_k \cdots\otimes_k{}^{\s^d}R.\]
and $[\s]_kR$ as the limit (i.e., the union) of the $R_d$'s ($d\geq 0$).
Let $\wtilde G=\spec(k[\wtilde G])$  be  a group $k$-scheme. Then, we denote by ${}^{\s^d}\wtilde G$ the group $k$-scheme
represented by ${}^{\s^d} k[\wtilde G]$, i.e., the group $k$-scheme
obtained from $\wtilde G$ by extension of scalar via $\s^d:k\rightarrow k$. We let also
$\wtilde G_d=\wtilde G \times {}^{\s}\wtilde G\times \dots \times{}^{\s^d} \wtilde G$ be the group $k$-scheme
represented by $k[\wtilde G]_d$. Finally, we define \textit{the group $k$-$\s$-scheme $[\s]_k \wtilde G$ associated with $\wtilde G$}
as the group $k$-$\s$-scheme represented by  $[\s]_k k[\wtilde G]$.
Notice that if $S$ is a $k$-$\s$-algebra and $S^\sharp$ is the underlying $k$-algebra, then we have
(see the end of \S A.4 in \cite{articleone}):
\beq\label{eq:skG(S)}
[\s]_k\wtilde G(S)=\wtilde G(S^\sharp).
\eeq
By abuse of notation, we will say that $H$ is a $\s$-closed subgroup
of the algebraic group $\wtilde G$ to mean that $H$ is a $\s$-closed subgroup of $[\s]_k\wtilde G$.

\medskip
On the other hand, if $G$ is a group $k$-$\s$-scheme, one can define its $d$-th order Zariski closure $G[d]$
as follows (see Definition A.11 in \cite{articleone}).
Let $\wtilde G$ be a group scheme over $k$ and $G$ a $\s$-closed subgroup of $\wtilde G$.  We denote by $G^\sharp$ the
group $k$-scheme defined by $G^\sharp(S)=\Alg_k(k\{G\},S)$ for any $k$-algebra $S$, where $\Alg_k$
stands for morphism of $k$-algebra. For $d\geq 0$, we consider
the group $k$-scheme $\wtilde G_d$ represented by $k[\wtilde G]_d$. Then, the smallest closed subscheme $G[d]$ of $\wtilde G_d$ such that
$G^\sharp\to \wtilde G_d$ factors through $G[d]\hookrightarrow \wtilde G_d$ is called the \emph{$d$-th order Zariski closure}
of $G$ inside $\wtilde G$. If $\I(G)$ is the vanishing
ideal of $G$ inside $\wtilde G$, then the vanishing ideal of $G[d]$ inside $\wtilde G_d$
is nothing else than $\I(G) \cap k[\wtilde G]_d$.

\medskip
The $\s$-dimension
$\s$-$\dim_k G$ of a $\s$-algebraic group $G$ over a $\s$-field $k$ is defined thanks to a difference Hilbert
polynomial. We refer to Appendix A.7 in \cite{articleone} for a precise definition. As one may expect,
for an algebraic group $\wtilde G$, we have
$\dim_k\wtilde G=\s$-$\dim_k[\s]_k\wtilde G$.

\subsection{$\s$-Picard-Vessiot extensions}

A $\ds$-ring is a ring $R$, that is simultaneously a $\de$ and a $\s$-ring,
such that for some unit $\hslash\in R^\de$ we have
\beq \label{eq:com}
\de(\s(r)) = \hslash\s(\de(r)),
\hbox{~for all $r\in R$.}
\eeq
We set $\hslash_0=1$ and
$\hslash_j=\hslash\sigma(\hslash_{j-1})$, for all integers $j\geq 1$, so that
$\de(\s^j(r)) = \hslash_j\s^j(\de(r))$, for all $r\in R$ and all non-negative integers $j$.
The element $\hslash$ is understood to be part of the data of a $\ds$-ring. So a morphism $\psi:R\rightarrow R'$
of $\ds$-rings is a morphism of rings such that $\psi\s=\s'\psi$, $\psi\de=\de'\psi$ and $\psi(\hslash)=\hslash'$.
Note that condition (\ref{eq:com}) implies that $R^\de$ is a $\s$-ring.
The definitions of $k$-$\ds$-algebra, $\ds$-extension, $\ds$-ideal are the intuitive ones. They
are exposed in detail in \cite{articleone}.

\medskip
Let $K$ be a $\ds$-field and $A\in K^{n\times n}$. We consider the differential system
$\de(y)=Ay$.
If $R$ is a $K$-$\ds$-algebra, then a matrix $Y\in\Gl_n(R)$ such that $\de(Y)=AY$ is called a
fundamental solution matrix for $\de(y)=Ay$. Notice that if $Y,Y'\in\Gl_n(R)$ are two fundamental
solution matrices for $\de(y)=Ay$ in some $K$-$\ds$-algebra then there exists $C\in\Gl_n(R^\de)$ such that $Y'=YC$.

\begin{defn}\label{defi:PVextPVring}
A $\s$-Picard-Vessiot extension of $K$ for $\de(y)=Ay$ is a $\ds$-field extension $L$
of $K$, $\s$-generated by the entries of a fundamental solution matrix and without new $\de$-constants, i.e.,
such that $L^\de = K^\de$.
A $K$-$\ds$-algebra $R$ is called a \emph{$\s$-Picard-Vessiot ring for $\de(y)=Ay$} if it is
$\s$-generated by the entries of a fundamental solution matrix and the inverse of its determinant and it is $\de$-simple.
A $\ds$-field extension $L$ of $K$ is called a \emph{$\s$-Picard-Vessiot extension} if it is a
$\s$-Picard-Vessiot extension for some differential equation $\de(y)=Ay$ with $A\in K^{n\times n}$;
similarly for $\s$-Picard-Vessiot rings.
\end{defn}

We recall the following two fundamental results about $\s$-Picard-Vessiot extensions.

\begin{prop}[{Corollary 1.13 in \cite{articleone}}]\label{cor:existenceofPVextension}
Let $K$ be a $\ds$-field and $A\in K^{n\times n}$. Assume that $K^\de$ is an algebraically closed field.
Then there exists a $\s$-Picard-Vessiot extension for $\de(y)=Ay$ over $K$.
\end{prop}

\begin{prop}[{Proposition 1.5 in \cite{articleone}}]\label{prop:pvringpvext}
Let $K$ be a $\ds$-field and $A\in K^{n\times n}$. If $L|K$ is a $\s$-Picard-Vessiot extension for $\de(y)=Ay$,
with fundamental solution matrix $Y\in\Gl_n(L)$, then
$R:=K\{Y,\frac{1}{\det(Y)} \}_\s$
 is a $\s$-Picard-Vessiot ring for $\de(y)=Ay$.
Conversely, if $R$ is a $\s$-Picard-Vessiot ring for $\de(y)=Ay$ with $R^\de=K^\de$, then the field of fractions
of $R$ is a $\s$-Picard-Vessiot extension for $\de(y)=Ay$.
\end{prop}

\subsection{The $\s$-Galois group and its properties}

If $R\subset S$ is an inclusion of $\ds$-rings, we denote by $\Aut^\ds(S|R)$ the automorphisms of $S$ over $R$ in the
category of $\ds$-rings, i.e., the automorphisms are required to be the identity on $R$ and to commute with $\de$ and $\s$.

\begin{defn}
Let $L|K$ be a $\s$-Picard-Vessiot extension with $\s$-Picard-Vessiot ring $R\subset L$. Set $k=K^\de$.
We define $\sgal(L|K)$ to be the functor from the category of $k$-$\s$-algebras to the category of groups
given by
\[\sgal(L|K)(S):=\Aut^\ds(R \otimes_{k} S | K \otimes_{k} S)\]
for every $k$-$\s$-algebra $S$. Notice that the action of $\de$ on $S$ is trivial, i.e.,
$\de(r\otimes s)=\de(r)\otimes s$ for $r\in R$ and $s\in S$. On morphisms $\sgal(L|K)$ is given by base extension.
We call $\sgal(L|K)$ the \emph{$\s$-Galois group of $L|K$}.
 \end{defn}

\begin{prop}[{Proposition 2.5 in \cite{articleone}}]
\label{prop:defgal}
Let $L|K$ be a $\s$-Picard-Vessiot extension with $\s$-Picard-Vessiot ring $R\subset L$. Then $\sgal(L|K)$ is a
$\s$-algebraic group over $k=K^\de$. More precisely, $\sgal(L|K)$ is represented by the finitely $\s$-generated
$k$-$\s$-algebra $(R \otimes_{K} R)^\de$. The choice of matrices $A\in K^{n\times n}$ and $Y\in\Gl_n(L)$ such
that $L|K$ is a $\s$-Picard-Vessiot extension for $\de(y)=Ay$ with fundamental solution matrix $Y$ defines a $\s$-closed embedding
\[\sgal(L|K)\hookrightarrow\Gl_{n,k} \] of $\s$-algebraic groups.
\end{prop}

In the notation of the proposition above,  we will identify $\sgal(L|K)$ with its image in
$\Gl_{n,k}$.
For $S$ a $k$-$\s$-algebra and $\tau\in\sgal(L|K)(S)$, we will usually
denote by $[\tau]_Y$ the image under the above morphism,
i.e., the matrix $[\tau]_Y$ in $\Gl_{n}(S)$ such that $\tau(Y\otimes 1) =(Y\otimes 1)[\tau]_Y$.
Another choice of fundamental solution matrix yields
a conjugated representation of $\sgal(L|K)$ in $\Gl_{n,k}$.
Therefore sometimes,  we will  consider
$\sgal(L|K)$ as a $\s$-closed subgroup of $\Gl_{n,k}$ without mentioning  the fundamental solution matrix $Y$.

\begin{prop}[{Proposition 2.17 in \cite{articleone}}]\label{prop:dimensiondegtrans}
Let $L|K$ be a $\s$-Picard-Vessiot extension with $\s$-Galois group $G$ and constant field $k=K^\de$.  Then
\[\strdeg(L|K) = \sdim_k(G).\]
\end{prop}

\begin{prop} [{Proposition 2.15 in \cite{articleone}}]\label{prop:Zariskiclosures}
Let $L|K$ be a $\s$-Picard-Vessiot extension with $\s$-field of $\de$-constants $k=K^\de$.
Let $A\in K^{n\times n}$ and $Y\in\Gl_n(L)$ such that $L|K$ is a $\s$-Picard-Vessiot extension
for $\de(y)=Ay$ with fundamental solution matrix $Y$. We consider the $\s$-Galois group $G$ of
$L|K$ as a $\s$-closed subgroup of $\Gl_{n,k}$ via the embedding associated with the choice of $A$ and $Y$.
Set $L_0=K\left(Y\right)\subset L$.
\par
Then $L_0|K$ is a (classical) Picard-Vessiot extension for the linear system $\de(y)=Ay$.
The (classical) Galois group of $L_0|K$ is naturally isomorphic to $G[0]$, the Zariski closure of $G$ inside $\Gl_{n,k}$.
\end{prop}

\subsection{Galois correspondence}

In the notation of Proposition \ref{prop:Zariskiclosures},
let $S$ be a $k$-$\s$-algebra, $\tau\in G(S)$ and $a\in L$.
By definition, $\tau$ is an automorphism of $R\otimes_k S$.
If we write $a=\frac{r_1}{r_2}$ with $r_1,r_2\in R$, $r_2\neq 0$ then, we say that $a$ is invariant under $\tau$ if and only if
$\tau(r_1\otimes 1)\cdot r_2\otimes 1=r_1\otimes 1\cdot \tau(r_2\otimes 1)$ in $R\otimes_k S$.

\par
If $H$ is a subfunctor of $G$, we say that $a\in L$ is invariant under $H$ if $a$ is invariant under every element
of $H(S)\subset G(S)$, for every $k$-$\s$-algebra $S$.
The set of all elements in $L$, invariant under $H$, is denoted with $L^H$. Obviously $L^H$ is an intermediate $\ds$-field of $L|K$.

\par
If $M$ is an intermediate $\ds$-field of $L|K$, then it is immediately clear from Definition \ref{defi:PVextPVring}
that $L|M$ is a $\s$-Picard-Vessiot extension with $\s$-Picard-Vessiot ring $MR$, the ring compositum of $M$ and $R$ inside $L$.
There is a natural $\s$-closed embedding $\sgal(L|M)\hookrightarrow\sgal(L|K)$ whose image consists of precisely
those automorphism that leave invariant every element of $M$.

\begin{thm}[$\s$-Galois correspondence; Theorem 3.2 in {\cite{articleone}}]
\label{theo:Galoiscorrespondence}
Let $L|K$ be a $\s$-Picard-Vessiot extension with $\s$-Galois group $G=\sgal(L|K)$. Then there is an inclusion
reversing bijection between the set of intermediate $\ds$-fields $M$ of $L|K$ and the set of $\s$-closed subgroups $H$ of $G$ given by
\[M\mapsto\sgal(L|M) \text{ and } H\mapsto L^H.\]
\end{thm}

\begin{thm}[Second fundamental theorem of $\s$-Galois theory; Theorem 3.3 in {\cite{articleone}}]
\label{theo:secondfundamentaltheorem}
Let $L|K$ be a $\s$-Picard-Vessiot extension with $\s$-Galois group $G$. Let $K\subset M\subset L$
be an intermediate $\ds$-field and $H\leq G$ a $\s$-closed subgroup of $G$ such that $M$ and $H$
correspond to each other in the $\s$-Galois correspondence.
\par
Then $M$ is a $\s$-Picard-Vessiot extension of $K$ if and only if $H$ is normal in $G$.
If this is the case, the $\s$-Galois group of $M|K$ is the quotient $G/H$.
(See Definition A.41 for the definition and Theorem A.43 in
\cite{articleone} for the existence of the quotient $G/H$ in the category of group $k$-$\s$-scheme.)
\end{thm}

%%%%%%%%%%%%%%%%%%%%%%%%%%%%%%%%%%%%%%%%%%%%%%%%%%%%%%%%%%%%%%%%%%%%%%%%%%%%%%%%%%%%%%%%%%%%%%%%
%%%%%%%%%%%%%%%%%%%%%%%%%%%%%%%%%%%%%%%%%%%%%%%%%%%%%%%%%%%%%%%%%%%%%%%%%%%%%%%%%%%%%%%%%%%%%%%%
%\input{lucia}
\section{Transformally dependent solutions of first order linear differential equations}

In this section, we study the existence of algebraic relations over a $\ds$-field K,
satisfied  by solutions of a differential equation with coefficients in $K$
and their transforms of all  order with respect to $\s$.
We show that these relations reflect the structure of the $\s$-Galois group of the equation.
More precisely, we
use the $\s$-Galois correspondence and the classification of the
$\s$-closed subgroups of $\Ga, \mathbf{G}_m$ and $\Gl_2$
(see the appendix) to characterize the transformal algebraic relations satisfied by
solutions of differential equations.
\par
We focus on two special cases: The field $\C(x)$ equipped with the
the usual derivation and with either the shift operator $\s:x\mapsto x+1$
or a $q$-difference operator $\s:x\mapsto qx$, for some $q\in\C$.
Notice that these two operators cover all the possible non-trivial
automorphisms of $\C(x)$, up to a M\"obius transformation.
\par
Analogous results for differential relations among solutions of
linear difference equations are proved in \S3.1 of \cite{HardouinSinger}.

\subsection{The additive case}

As above, let $K$ be a $\ds$-field,  with $\de\s= \hslash\s\de$ and $\hslash\in k:=K^\de$.
We remind that $\hslash_j\in k$ is defined so that $\de\s^j=\hslash_j\s^j\de$ for any integer $j\geq 1$. It is convenient to set
$\hslash_0=1$.
\par
We consider a system of first order inhomogeneous linear differential equations
\beq\label{eq:inhomogeneaoussystem}
\de(y_1)=b_1,~\ldots,~\de(y_n)=b_n,
\hbox{~with $b_1,\dots,b_n\in K$.}
\eeq
We are going to prove a result on the transformal dependence of a set of solutions
of \eqref{eq:inhomogeneaoussystem}. It can be considered as a $\s$-analog of the following theorem by
Ostrowski \cite{OstrovskiActa}:

\begin{thm}\label{thm:ostrowski}
Let $M$ be a field of meromorphic functions, containing $\C$ and stable by the derivation $\de=\frac{d}{dx}$.
Let $z_1,\dots,z_n$ be meromorphic functions such that $\de(z_i)\in M$, for all $i=1,\dots,n$.
Then $z_1,\dots,z_n$ are algebraically dependent over $M$ if and only if there exist
$\la_1,\dots,\la_n\in\C$, not all equal to zero, such that $\sum_{i=1}^n \la_i z_i\in M$.
\end{thm}

We go back to our notation.
The theorem below relies on the classification of
$\s$-closed subgroups of $\Ga^n$ (see \S\ref{subsec:subgroupsGa} below).

\begin{thm}\label{thm:rank1ga}
Let $L|K$ be a $\ds$-field extension containing
a solution $z_1,\hdots, z_n \in L$ of \eqref{eq:inhomogeneaoussystem} and having the property that
$L^\de=k(:=K^\de)$.
Then, $z_1,\hdots, z_n$ are transformally dependent over $K$ if and only if
there exist a non-zero homogeneous linear $\s$-polynomial
$\cL(X_1,\dots,X_n)\in k\{X_1,\dots,X_n\}_\s$
%a non-negative integer $s$ and, for
%$i=1,\dots,n$ and $j=1,\dots,s$, a family of constants $\lambda_{i,j} \in k$, not all equal to zero,
and an element $g$ of $K$ such that $\cL(b_1,\dots,b_n)=\de(g)$.
\end{thm}

\begin{proof}
Let us first assume that there exist
$\cL(X_1,\dots,X_n)=\sum_{i=1}^n\sum_{j=0}^s \lambda_{i,j} \sigma^j(X_i)\in k\{X_1,\dots,X_n\}_\s$
and $g\in K$ such that
\beq\label{eq:AdditiveTransformallyDependentCondition}
\sum_{i=1}^n\sum_{j=0}^s \lambda_{i,j} \sigma^j(b_i) = \de(g).
\eeq
Then a direct calculation shows that
$$
\de\l(\sum_{i,j}\hslash_j^{-1}\lambda_{i,j} \sigma^j(z_i) -g\r)=0.
$$
Since $L^\de=K^\de$, we conclude that $\sum_{i,j}\hslash_j^{-1}\lambda_{i,j} \sigma^j(z_i)\in K$.
Hence $z_1,\hdots, z_n$ are transformally dependent over $K$.
\par
To prove the inverse implication, we consider the differential system
\beq\label{eq:ga}
\de(y)=Ay,
\hbox{~where~}
A=\operatorname{diag}
\l(\begin{pmatrix}0&b_1\\0&0\end{pmatrix},\dots,\begin{pmatrix}0&b_n\\0&0\end{pmatrix}\r)\in K^{2n\times 2n}
\hbox{~is a diagonal block matrix,}
\eeq
and its fundamental solution $Y:=\operatorname{diag}\l(\begin{pmatrix}1&z_1\\0&1\end{pmatrix},\dots,\begin{pmatrix}1&z_n\\0&1\end{pmatrix}\r)\in\Gl_{2n}(L)$.
\par
First of all, we construct the $\s$-Galois group of $\de(y)=Ay$.
Since $L^\de=K^\de$, the $\ds$-field $K\left\<z_1,\hdots,z_n \right\>_\s\subset L$ is a $\s$-Picard-Vessiot extension for \eqref{eq:ga} and it is not restrictive to assume that
$L=K\left\<z_1,\hdots,z_n \right\>_\s$.
The corresponding $\s$-Picard-Vessiot ring is
$R=K\left\{z_1,\hdots,z_n \right\}_\s$. So let $G=\sgal(L|K)$ denote the $\s$-Galois group of $L|K$, seen
as a $\s$-closed subgroup of $\Gl_{2n,k}$ via the representation associated to $Y$.
\par
We claim that $G$ is a proper $\s$-closed subgroup of $\Ga^n$, where we have identified
the vector group $\Ga^n$ with a $\s$-closed subgroup of $\Gl_{2n,k}$ via the embedding $\Ga^n\hookrightarrow \Gl_{2n,k}$
defined by:
$$
\begin{array}{lcl}
\Ga^n(S)(=S^n) & \longrightarrow & \Gl_{2n}(S)\\
(c_1,\ldots,c_n) & \longmapsto &
        \operatorname{diag}\l(\begin{pmatrix}1&c_1\\0&1\end{pmatrix},\dots,\begin{pmatrix}1&c_n\\0&1\end{pmatrix}\r)
\end{array},
\hbox{~for any $k$-$\s$-algebra $S$.}
$$
Indeed, for any $k$-$\s$-algebra $S$ and $\tau\in G(S)$, the automorphism $\tau$ commutes with $\de$, so that in $R\otimes_k S$
we have $\de(\tau(z_i \otimes 1))=b_i\otimes 1$, for all $i=1,\ldots,n$, and, hence $\de(\tau(z_i\otimes 1)-z_i\otimes 1)=0$. So
$\tau(z_i\otimes 1)=z_i\otimes 1+1\otimes c_i$ for some $c_i\in S$, which identifies $G$ to a $\s$-closed subgroup of $\Ga^n$.
Since $z_1,\hdots, z_n$ are transformally dependent over $K$,
i.e., since $\strdeg(L|K)<n$, Proposition \ref{prop:dimensiondegtrans} implies that
$$
\sdim_k(G)=\strdeg(L|K)<n=\sdim_k(\Ga^n),
$$
and therefore that $G$ is a proper $\s$-closed subgroup of $\Ga^n$.
\par
The group $G$ being a proper $\s$-closed subgroup of $\Ga^n$,
there exists a non-zero linear $\s$-polynomial in $n$ variables, say $\sum_{i,j} \mu_{i,j}\s^j(X_i)\in k\{X_1,\dots,X_n\}_\s$
(see Theorem \ref{thm:clasga} below),
which belongs to the vanishing ideal of $G$ inside $\Ga^n$. This means that for any $k$-$\s$-algebra $S$
and any $\tau=(c_1,\ldots,c_n)\in G(S)$, we must have $\sum_{i,j} \mu_{i,j}\s^j(c_i)=0$.
We set $g=\sum_{i,j} \mu_{i,j}\s^j(z_i) \in R$. Then for any $k$-$\s$-algebra $S$ and any $\tau=(c_1,\ldots,c_n) \in G(S)$,
in $R\otimes_k S$ we have
$$
\tau(g\otimes 1)=\sum_{i,j} \mu_{i,j}\s^j(z_i \otimes 1+1\otimes c_i)=g\otimes 1.
$$
The Galois correspondence (see Theorem \ref{theo:Galoiscorrespondence} above)
implies that $g \in K$. Taking the derivative of $g$,
we find $\sum_{i,j} \hslash_j \mu_{i,j} \sigma^j(b_i) = \de(g)$.
To conclude, it is enough to choose $\cL(X_1,\dots,X_n)=\sum_{i,j} \hslash_j \mu_{i,j} \sigma^j(X_i)$.
\end{proof}

\begin{rem}
Theorem \ref{thm:rank1ga} can be deduced from a more algebraic version of Ostrowski's Theorem,
involving instead of differential fields of meromorphic functions, a differential field extension $L$ of $K$
with $L^\de=k$ (see \S 2 in \cite{Kolchin:AlgebraicGroupsAndAlgebraicDependence}). If
$L$ is a $\ds$-field extension  of $K$, with $L^\de=k$, containing solutions $z_1,\dots,z_n$
of $\de(y_i)=b_i$, it is enough to apply Ostrowski's Theorem
to $z_1,\dots,z_n,\s(z_1),\dots,\s(z_n),\dots,\s^l(z_1),\dots,\s^l(z_n)$,
for some large enough integer $l$, to get Theorem \ref{thm:rank1ga}. However, we want to point out that the existence of
such a $\ds$-field extension $L$ is a very strong hypothesis, which is a priori not guaranteed by classical Picard-Vessiot
theory.  Indeed, for the differential equations \beq \label{eq:prolsyst}\de(y_{i,j})=\hslash_j\s^j(b_i)y_{i,j}\eeq
with $i=1,\dots,n$ and $j=1,\dots,l$, classical Picard-Vessiot theory
 provide us with a $\de$-field $L$ with $L^\de=k$ containing the solutions of \eqref{eq:prolsyst}.
 However, nothing guaranties that $L$ is also a $\s$-field.
 The strength of our approach is to show that
if one assume $k$ to be algebraically closed, one can solve any linear differential equation
in a $\ds$-field $L$ with $L^\de=k$ (see Proposition \ref{cor:existenceofPVextension}). Moreover, we also prefer to give here
a proof of Theorem \ref{thm:rank1ga} relying on the classification of difference group
scheme of $\Ga$ and parametrized Galois correspondence, in preparation of the more
complicated applications to come. We think that a purely differential proof of those statements would be, if possible,
extremely heavy.
%because,
%in the sequel for more complicated differential equations, we will really need
%to work with difference group scheme to get criteria of transformal dependencies.  To our opinion,
%if it ever exists a   purely algebraic proof of these statements, this proof shall be much more heavy.
% Of course,   the multiplicative case (see \S \ref{subsec:multcas}) is the same than the additive one and
%can be also deduced directly from Kolchin's Theorem.
\end{rem}

In the special case $K=\C(x)$, we consider the following two situations
\begin{align}
&\hbox{$\de =\frac{d}{dx}$ and $\s(x)=x+1$;}\label{shift}\\
&\hbox{$\de=\frac{d}{dx}$ and $\s(x)=qx$, where $q\in\C\smallsetminus\{0,1,\hbox{roots of unity}\}$.}\label{qdiff}
\end{align}
Notice that, in the first case, $\s$ and $\de$ commute, so $\hslash=1$, while in the second case we have
$\de\s=q\s\de$, and therefore $\hslash=q$. In both cases,
$\hslash_j=\hslash^j$, for any integer $j\geq 0$.

\begin{exa}
We consider the situation \eqref{shift}.
Let $L$ be the field of meromorphic function over
$\C\smallsetminus\R_{\leq 0}$
and let  $z=\log(x) \in L$ be the principal branch of the logarithm.
Since $L^\de=\C$ and $\de(z) =\frac{1}{x}$, the assumptions of
Theorem \ref{thm:rank1ga} are verified.
If $z$ was transformally dependent over $\C(x)$, then there would exist
a non-zero homogeneous linear $\s$-polynomial $\cL(X)=\sum_{j=0}^s\la_j\s^j(X)\in\C\{X\}_\s$
and $g \in \C(x)$, such that
$$
\sum_{j=0}^s \frac{\lambda_j}{x+j} =\de(g).
$$
Writing the fractional expansion of $\de(g)$, we
see that $\de(g)$ can never satisfy such a relation.
This means that $\log(x)$, $\log(x+1)$, $\hdots$, $\log(x+n)$, $\hdots$ are algebraically independent over $\C(x)$.
\end{exa}

This example illustrates the following more general criterion.

\begin{cor}\label{cor:gaexample}
Let $b \in \C(x)$ and let $L|\C(x)$ be a $\ds$-field extension, with $L^\de=\C$,
containing a solution $z$ of $\de(y)=b$. Then,
in the situation \eqref{shift} (resp. \eqref{qdiff}),
$z$ is transformally dependent  over $\C(x)$ if and only if $b$ has no simple poles
(resp. $b$ has no non-zero simple poles).
\end{cor}

\begin{proof}
It follows from Theorem \ref{thm:rank1ga} that $z$ is transformally dependent over $\C(x)$ if and only if
there exists a non-zero homogeneous linear $\s$-polynomial $\cL(X)=\sum_{j=0}^s\la_j\s^j(X)\in\C\{X\}_\s$
and $g\in \C(x)$, such that $\cL(b)=\de(g)$.
Writing the fractional expansion of $\de(g)$,
we find that $z$ is transformally dependent over $\C(x)$ if and only if  there exists a non-zero homogeneous linear $\s$-polynomial $\cL(X)\in\C\{X\}_\s$, such that $\cL(b)$ has no simple poles.
\par
We consider the case \eqref{shift}.
We claim that $b$ has no simple poles if and only if there exists
a non-zero homogeneous linear $\s$-polynomial $\cL(X)=\sum_{j=0}^s\la_j\s^j(X)\in\C\{X\}_\s$
such that $\cL(b)$ has no simple poles.
In fact, one implication is trivial. So let us assume that
such an $\cL$ exists and that $\gamma$ is a simple pole of $b$. Then, let $n$
be the largest integer such that $\gamma-n$ is a simple pole of $b$.  Then,
$\gamma-n-s$ is a simple pole of $\sum_{j=0}^s\lambda_j\s^j(b)$, contradicting the assumption.
\par
In the situation \eqref{qdiff}, we are going to prove that
$b$ has no non-zero simple poles if and only if there exists
a non-zero homogeneous linear $\s$-polynomial $\cL(X)=\sum_{j=0}^s\la_j\s^j(X)\in\C\{X\}_\s$
such that $\cL(b)$ has no simple poles.
If $b$
has no non-zero simple poles, then we can write $b$ as the sum $\frac{c}{x} + g$,
where $c\in\C$ and $g\in\C(x)$ has no simple poles.
It follows that $q\s(b)-b = qg(qx)-g$ has no simple poles.
The same reasoning as in case \eqref{shift} proves the inverse implication.
\end{proof}

The above corollary allows us to describe precisely which $\s$-closed subgroups of $\Ga$ occur as $\s$-Galois groups of an equation $\de(y)=b$ in the situations \eqref{shift} and \eqref{qdiff}, anticipating on the content of \S \ref{sec:pbinv}.

\begin{cor}\label{cor:shiftigaC}
In the situation \eqref{shift}, let $L|K$ be a $\s$-Picard-Vessiot extension for $\de(y)=b$ with $b\in K=\C(x)$ and let $G$ be the $\s$-Galois group of $L|K$.
\begin{enumerate}
\item
If $b$ has no simple poles then $G$ is trivial, i.e., $L=K$.
\item
If $b$ has a simple pole then $G=\Ga$.
\end{enumerate}
\end{cor}

\begin{proof}
Since the proof is quite similar to the proof of the $q$-difference case (see Corollary
 \ref{cor:qdiffigaC} below),
we refer to the $q$-difference case below for the proof.
%, we will prove only the proposition in the case \eqref{qdiffbis} and
%we refer the reader to the proof of Corollary \ref{cor:gaexample} and its results for the case \eqref{shiftbis}.
\end{proof}

\begin{cor}\label{cor:qdiffigaC}
In the situation \eqref{qdiff}, let $L|K$ be a $\s$-Picard-Vessiot extension for $\de(y)=b$ with $b\in K=\C(x)$ and let $G$ be the $\s$-Galois group of $L|K$.

\begin{enumerate}
\item
If $b$ has no simple poles then $G$ is trivial.
\item
If $b$ has a simple pole in zero but no other simple poles then $G=\Ga^\s$, i.e.,
$G(S)=\{c\in S|\ \s(c)=c\}\leq \Ga(S)$ for every $k$-$\s$-algebra $S$.
\item
If $b$ has a simple pole distinct from zero then $G=\Ga$.
\end{enumerate}
\end{cor}

\begin{proof}
Let $z \in L$ be a solution of $\de(y)=b$. Assume that $b$ has no simple poles. Then there exists a solution of $\de(y)=b$ in $K$. Therefore $z\in K$ and $L=K\langle z\rangle_\s=K$. This proves (i).

If $b$ has a simple pole in zero but no other simple poles, $q\s(b)-b =\de(h)$ for some $h \in K$. Then,
$\de(\s(z)-z)  =\de(h)$, which implies $\s(z)-z \in K$. It follows easily that $G\leq \Ga^\s$. Since $z\notin K$ we must have $G=\Ga^\s$ (as $\Ga^\s$ has no non-trivial $\s$-closed subgroups). This proves (ii).

If $b$ has a simple pole distinct from zero, then $z$ is transformally independent over $K$ by Corollary \ref{cor:gaexample}.
Therefore $G=\Ga$ in this case.
\end{proof}

\subsection{The multiplicative case}\label{subsec:multcas}

Inspired by Ostrowski's Theorem \ref{thm:ostrowski} above, Kolchin proved the following
statement (see \cite{Kolchin:AlgebraicGroupsAndAlgebraicDependence}, page 1156):

\begin{thm}\label{thm:Kolchin}
Let $L|K$ be a $\de$-field extension, with $L^\de=K^\de$, and assume that $z_1,\ldots,z_n\in L^\times$ satisfy $\frac{\de(z_1)}{z_1},\ldots,\frac{\de(z_n)}{z_n}\in K^\times$.
Then $z_1,\ldots,z_n$ are algebraically dependent over $K$
if and only if there exist $r_1,\ldots,r_n\in\mathbb{Z}$, not all equal to zero, such that
$z_1^{r_1}\cdots z_n^{r_n}\in K$.
\end{thm}

We are going to prove a $\s$-analog of Kolchin's theorem:

\begin{thm}\label{thm:rank1}
Let $L|K$ be a $\ds$-field extension with $L^\de =K^\de$. Let $a_1,\hdots,a_n \in K^\times$
and $z_1,\hdots,z_n\in L^\times$ be such that
$$
\de (z_i)=a_iz_i \ \text{ for } i=1, \hdots, n.
$$
Then $z_1,\hdots, z_n$ are transformally dependent over $K$  if and only if
there exist
a non-zero homogeneous linear $\s$-polynomial
$\mathcal{L}(X_1,\hdots,X_n)=\sum_{i=1}^n\sum_{j=0}^s r_{i,j} \sigma^j(X_i)\in \Z\{X_1,\dots,X_n\}_\s$
and a non-zero element
$g \in K$ such that
\beq\label{eq:homogeneousrelation}
\sum_{i=1}^n\sum_{j=0}^s r_{i,j}\hslash_j\sigma^j(a_i) = \frac{\de(g)}{g}.
\eeq
\end{thm}

\begin{rmk}
We are going to prove that \eqref{eq:homogeneousrelation} is equivalent
to $\prod_{i,j}\s^j(z_i)^{r_{i,j}}\in K$, which makes clear that the theorem above is
the analogue of Kolchin's result.
\end{rmk}

\begin{proof}
Assume that \eqref{eq:homogeneousrelation} is satisfied. Then, the element
$$
z:= \frac{1}{g}\prod_{i,j} \s^j(z_i)^{r_{i,j}}\in L^\times
$$
satisfies $\frac{\de(z)}{z}=0$. Since $L^\de=K^\de$, we have $z \in K$.
Hence, we conclude that $z_1,\hdots, z_n$ are transformally dependent over $K$.
\par
Conversely, assume that $z_1,\hdots, z_n$ are transformally dependent over $K$.
The proof is quite similar to the proof of Theorem \ref{thm:rank1ga}, therefore we will
skip some details.
Since $L^\de=k(:=K^\de)$, the $\ds$-field $K\left\< z_1, \hdots,z_n \right\>_\s$
is a $\s$-Picard-Vessiot extension for the differential system $\de(y)=Ay$,
where $A={\rm diag}(a_1,\dots,a_n)$.
We can assume that $L=K\left\< z_1, \hdots,z_n \right\>_\s$.
Let $G=\sgal(L|K)$. Via the fundamental solution matrix $Y:=\diag(z_1,\dots,z_n)$, the $\s$-Galois group $G$ is naturally a $\s$-closed subgroup of $\Gm^n$.
Now, since $z_1,\dots,z_n$ are transformally dependent over $K$, we find that $\strdeg(L|K)<n$. Proposition \ref{prop:dimensiondegtrans} allows us to conclude that
$$
\sdim_k(G)=\strdeg(L|K)<n=\sdim_k(\Gm^n)
$$
and hence that $G$ is a proper $\s$-closed subgroup of $\Gm^n$.
It follows from Theorem \ref{thm:classgm} that there exists a non-trivial multiplicative function
$\psi\in k\{\Gm^n\}=k\{x_1,\hdots,x_n,\frac{1}{x_1},\ldots\frac{1}{x_n}\}_\s$,
of the form
\beq\label{eq:multiplicativerelation}
\psi(x_1,\dots,x_n)=x_1^{r_{1,0}}\cdots x_n^{r_{n,0}}~\s(x_1)^{r_{1,1}}\cdots \s(x_n)^{r_{n,1}} \cdots \s^s(x_1)^{r_{1,s}}\cdots \s^s(x_n)^{r_{n,s}},
\eeq
where $s \in \Z_{\geq 0}$ and $r_{i,j} \in \Z$, for $1\leq i\leq n$, $0\leq j\leq s$,
such that $\psi(c_1,\dots,c_n)=1$, for all $k$-$\s$-algebra $S$ and all $\tau=(c_1,\dots,c_n)\in G(S)\leq\Gm^n(S)$.
Consider the non-zero element $g:=\psi(z_1,\ldots,z_n)$ of $L$.
For any $k$-$\s$-algebra $S$ and any $\tau=(c_1,\ldots,c_n)\in G(S)\subset\Gm^n(S)$,
in the $\ds$-algebra $R\otimes_k S$, we have
$$
\tau(g\otimes 1)=\psi(z_1\otimes c_1,\ldots,z_n\otimes c_n)=\psi(z_1,\ldots,z_n)\otimes \psi(\tau)=g\otimes 1.
$$
It follows from the Galois correspondence (see Theorem \ref{theo:Galoiscorrespondence} above) that $g\in K$.
Taking the logarithmic derivative of $g$,
we find \eqref{eq:homogeneousrelation}.
\end{proof}

As in the previous section, we are going to give a more detailed description of the
case $K=\C(x)$,
under the assumptions \eqref{shift} and \eqref{qdiff}.
It is an elementary observation that a rational function $g\in\C(x)$ is the logarithmic derivative
of a rational function if and only if it is of the form
$g=\sum_{i=1}^k\frac{c_i}{x-\alpha_i}$ with $c_1\ldots c_k\in\mathbb{Z}$ and $\alpha_1,\ldots,\alpha_k\in\C$.
Therefore, the following corollaries yield a very direct and easy to use criterion
to decide whether or not a meromorphic solution of $\de(y)=ay$, with $a\in\C(x)^\times$, is transformally dependent.

\begin{cor}\label{cor:rank1shift}
In the situation \eqref{shift}, let $L|\C(x)$ be a $\ds$-field extension, with $L^\de=\C$, and assume that $z\in L^\times$
satisfies $\de(z)=a z$, with $a\in\C(x)^\times$.
Then, $z$ is transformally dependent over $\C(x)$ if and only if
there exist $P\in \C[x]$, $f \in \C(x)^\times$ and $N\in \Z^\times$ such that $a=P + \frac{1}{N}\frac{\de(f)}{f}$.
\end{cor}

\begin{proof}
We remind that Theorem \ref{thm:rank1} implies that
the solution $z$ is transformally dependent over $\C(x)$ if and only if there exist
a non-zero homogenous $\s$-polynomial
$\cL(X)=\sum_{j=0}^s r_j \s^j (X)\in\Z\{X\}_\s$
and $g \in \C(x)^\times$ such that
\beq\label{eq:multiplicativerelationfromprevioustheorem}
\sum_{j=0}^s r_j \s^j (a) =\frac{\de(g)}{g}.
\eeq
Notice that the rational function $\frac{\de(g)}{g}$ has only simple poles and the residue at each pole is an integer.
%there exist $c_1,\hdots,c_l \in \Z$, not all zero, and $\alpha_1,\hdots,\alpha_l \in \C$
%such that   $\frac{\de(g)}{g} = \sum_{j=1}^l \frac{c_j}{x-\alpha_j}$.
\par
First of all we prove the following claim:
If the solution $z$ is transformally dependent over $\C(x)$ then
$a$ has no poles of order greater than or equal to $2$, i.e., the rational function
$a$ has the form $P + \sum_{h=1}^k\frac{d_h}{x-\gamma_h}$, for some $P \in \C[x]$
and $\gamma_h, d_h \in \C$.
Suppose that $\gamma$ is a pole of maximal order $m \geq 2$ of $a$ and that
$\ga-n$ is not a pole of order $m$ of $a$ for any positive integer $n$.
Notice that $\gamma-j$ is a pole of order $m$ of $\s^j(a)$, hence
$\gamma-s$ is a pole of order $m$ of
$\sum_{j=0}^s r_j \s^j (a) =\frac{\de(g)}{g}$, which is impossible. This prove the claim.
\par
We have proven that $a=P+\sum_{h=1}^k\frac{d_h}{x-\gamma_h}$.
To conclude that there exist $N \in \Z^\times$ and $f \in \C(x)^\times$ such that $ \sum_{h=1}^k\frac{d_h}{x-\gamma_h} =\frac{1}{N}\frac{\de(f)}{f}$, it suffices to show that
$d_h \in \Q$, for all $h=1,\hdots,k$. To this purpose, let us consider
a set $\Gamma$ of representatives of the classes of the poles of $a$ modulo $\Z$ and
write $ \sum_{h=1}^k\frac{d_h}{x-\gamma_h} $ as
\beq\label{eq:decompositionorbits}
\sum_{\gamma \in\Gamma} \sum_{n \in \Z} \frac{d_n(\gamma)}{x- (\gamma-n)}.
\eeq
Only a finite number of $d_n(\gamma)\in\C$ in the expression above are non-zero.
We know from \eqref{eq:multiplicativerelationfromprevioustheorem}
that $\sum_{j=0}^s r_j \s^j (a)$ is the logarithmic derivative of a rational function,
hence
$$
\sum_{j=0}^s r_j \s^j \l(\sum_{n \in \Z} \frac{d_n(\gamma)}{x- (\gamma-n)}\r)
=\sum_{n \in \Z} \frac{ \sum_{j=0}^s d_{n-j}(\gamma)r_j}{x-(\gamma-n)},
\hbox{~for any $\ga\in\Gamma$,}
$$
is also the logarithmic derivative of a rational fraction.
Thus, for any $\gamma\in\Gamma$ and all $n\in\Z$ we have:
\beq \label{eq:comb}
\sum_{j=0}^s d_{n-j}(\gamma)r_j \in \Z.
\eeq
For a fixed $\gamma\in\Gamma$,
we denote by $n_0$ the smallest integer such that $d_{n_0}(\gamma) \neq 0$  and
by $j_0$ the smallest integer such that $r_{j_0} \neq 0$.
The linear relation \eqref{eq:comb} for $n=n_0+j_0$ implies that $d_{n_0}(\gamma)r_{j_0} \in \Z$ and hence that
$d_{n_0}(\gamma) \in \Q$. We assume recursively that $d_k(\gamma)\in \Q$ for all $n_0\leq k \leq p$.
Expression \eqref{eq:comb} for $n =p+1+j_0$ yields
$$
d_{p+1}(\gamma)r_{j_0} +\sum_{j=1}^{s-j_0}d_{p+1-j}(\gamma)r_{j+j_0} \in \Z,
$$
which implies $d_{p+1}(\gamma)r_{j_0}\in\Z$ and hence $d_{p+1}(\gamma) \in \Q$.
We have proved one implication.
\par
Conversely, let us assume that there exist $P \in \C[x], f \in \C(x)^\times$ and $N\in \Z^\times$,
such that $a=P + \frac{1}{N}\frac{\de(f)}{f}$.
Let $d$ be the degree of $P$ and let us denote by $\Q[x]_{\leq d}$ the $\Q$-vector space of polynomials of degree less than or equal to $d$.
It is
a $\Q$-vector space of finite dimension, on which $\s:\Q[x]_{\leq d} \rightarrow \Q[x]_{\leq d}, x \mapsto x+1$ induces a
$\Q$-linear bijection.
By  Cayley-Hamilton theorem, there exists a polynomial $\sum_{j=0}^sr_jX^j\in\Z[X]$, with $r_s\neq 0$,
such that $\sum_{j=0}^s r_j\s^j=0$ on $\Q[x]_{\leq d}$.
In particular, by $\C$-linearity, we have $\sum_{j=0}^s r_j \s^j(P)=0$. The following calculation allows to conclude the proof:
$$
N \sum_{j=0}^s r_j \s^j(a)=\sum_{j=0}^sr_j\frac{\de(\s^j(f))}{\s^j(f)}=\frac{\de(\prod_{j=0}^s\s^j(f^{r_j}))}{\prod_{j=0}^s\s^j(f^{r_j}))}.
$$
\end{proof}

The case \eqref{qdiff}
is slightly more complicated than the previous one, since one has to distinguish the case when $q$ is an algebraic or
a transcendental number.
We start with an example illustrating the problem.

\begin{exa}
Consider the equation $\de(y) =y$, with the notation \eqref{qdiff}.
Since $\de\circ \s=q\s\circ\de$,  Theorem \ref{thm:rank1} above says that $\exp(q^jx)$, for $j\geq 0$, are algebraically dependent
if and only if there exist a non-zero polynomial $\sum_{j=0}^sr_jX^j\in\Z[X]$
and $g\in\C(x)$ such that $\sum_{j=0}^sq^jr_j=\de(g)g^{-1}$.
The latter can be verified if and only if $\de(g)g^{-1}\in\C$, and therefore if and only if $\sum_{j=0}^sq^jr_j=0$.
We deduce that $\exp(q^j x)$, for $j\geq 0$, are algebraically dependent if and only if $q$ is an algebraic number.
\end{exa}

In general, we have:

\begin{cor}\label{cor:rank1q}
In the notation \eqref{qdiff},
let $a \in \C(x)^\times$ and let $L|\C(x)$ be a $\ds$-field extension with $L^\de=\C$, containing a non-zero solution $z$
of $\de(y)=ay$.
We have:
\begin{itemize}
\item
If $q$ is a \emph{transcendental} number then
$z$ is transformally dependent over $\C(x)$ if and only if
there exist $f \in \C(x)^\times,\ N \in \Z^\times$ and $c\in \C$ such that $a=\frac{c}{x} +\frac{1}{N}\frac{\de(f)}{f}$.

\item
If $q$ is an \emph{algebraic} number then
$z$ is transformally dependent over $\C(x)$ if and only if
there exist $f \in \C(x)^\times$, $N \in \Z^\times$, $P\in \C[x]$, $Q\in\C\l[\frac{1}{x}\r]$ such that $a=P+Q+\frac{1}{N}\frac{\de(f)}{f}$.
\end{itemize}
\end{cor}

\begin{proof}
As in the previous corollary, Theorem \ref{thm:rank1} implies that
the solution $z$ is transformally dependent over $\C(x)$ if and only if there exist
a non-zero homogeneous linear $\s$-polynomial
$\mathcal{L}(X)=\sum_{j=0}^s r_{j} \sigma^j(X)\in \Z\{X\}_\s$
and a non-zero element $g \in K$ such that
\beq\label{eq:multiplicativerelationfromprevioustheorembis}
\sum_{j=0}^s q^jr_j \s^j (a) =\frac{\de(g)}{g}.
\eeq
Once more, notice that the rational function $\frac{\de(g)}{g}$
has only simple poles and the residue at each pole is an integer.
Moreover, the fractional expansion of $a$ is the sum of $P(x)\in\C[x]$
and some polar terms of order greater or equal to $1$.
\par
We suppose that $z$ is transformally dependent over $\C(x)$.
Then, the non-zero  poles of $a$ must be of order $1$.
In fact, suppose that there exists a pole $\gamma \neq 0$ of $a$, of order $m\geq 2$. We can choose $\gamma$ such that
$q^{-n}\gamma$ is not a pole of $a$ of order  $m$, for any integer $n\geq 1$.
Then $\sum_{j=0}^sr_jq^j\s^{j}(a)$ has a pole of order $m$ in $q^{-s}\gamma$, which is impossible
because of \eqref{eq:multiplicativerelationfromprevioustheorembis}.
This implies that
$$
a = P+Q+ \sum_{h=1}^k\frac{d_h}{x-\gamma_h},
$$
with $P\in \C[x]$, $Q\in\C\l[\frac{1}{x}\r]$, $\gamma_h \in \C^\times$ and $d_h \in \C$. A reasoning  analogous to the one in  Corollary \ref{cor:rank1shift} shows that $d_h \in  \Q$ for $h=1,\hdots,k$.
So we have proved that
$a =P+Q+\frac{1}{N}\frac{\de(f)}{f}$, for some $f \in \C(x)^\times$ and $N\in\Z^\times$.
If $q$ is algebraic there is nothing more to prove.
If $q$ is transcendental, we have to show that there exists $c \in \C$ such that
$P+Q=\frac{c}{x}$. If not, there exists $n \in \Z$ not equal to $-1$
such that  $x^n$ appears in $P+Q$.
Tracking the coefficient of $x^n$ in \eqref{eq:multiplicativerelationfromprevioustheorembis},
we find
$\sum_{j=0}^sq^{j(1+n)}r_j=0$, which contradicts the assumption on $q$.
We have now proved one implication both under the assumption that $q$ is algebraic and that $q$
is transcendental.
\par
Now we assume that $q$ is transcendental and that $a=\frac{c}{x} +\frac{1}{N}\frac{\de(f)}{f}$, for some $c \in \C,\ N \in \Z^\times$ and $f\in \C(x)^\times$.
Then
$$
Nq\s(a)-Na=\frac{\de(\s(f)/f)}{\s(f)/f},
$$
which, by Theorem \ref{thm:rank1}, implies that $z$ is transformally dependent over $\C(x)$.
\par
Finally, let $q$ be an algebraic number and suppose that there exist $f \in \C(x)^\times$, $N \in \Z^\times$, $P\in \C[x]$ and $Q\in\C\l[\frac{1}{x}\r]$ such that
$a=P+Q+\frac{1}{N}\frac{\de(f)}{f}$.
Let $d$ denote the maximum of the degree of $P$ in $x$ and of the degree of $Q$ in $\frac{1}{x}$ and
$$
\Q(q)[x^{\pm 1}]_{\leq d}:=\Q(q)+\sum_{i=1}^d\Q(q)x^i+\sum_{i=1}^d\Q(q)\frac{1}{x^i}.
$$
The set $\Q(q)[x^{\pm 1}]_{\leq d}$ is a finite dimensional $\Q$-vector space and the $\Q(q)$-linear
map $q\s:f(x)\mapsto qf(qx)$ is a $\Q$-linear endomorphism of $\Q(q)[x^{\pm 1}]_{\leq d}$.
By  Cayley-Hamilton theorem,
there exist a non-zero polynomial $\sum_{j=0}^sr_jX^j\in\Z[X]$, such that
$\sum_{j=0}^sr_jq^j\s^j=0$ on $\Q(q)[x^{\pm 1}]_{\leq d}$.
In particular, by $\C$-linearity, we have
$\sum_{j=0}^sq^jr_j\s^j(P+Q)=0$ and therefore $\sum_{j=0}^sNr_jq^j\s^j(a)= \frac{\de(\prod_{j=0}^s\s^j(f^{r_j}))}{\prod_{j=0}^s\s^j(f^{r_j}))}$.
\end{proof}

\subsection{An analog of a result of Ishizaki}
\label{subsec:Ishizaki}

In 1998, Ishizaki proved a result of differential independence of meromorphic
solutions of order $1$ linear inhomogeneous
$q$-difference equations (see  Theorem 1.2 in \cite{Ishihyper}):

\begin{thm}
Let $q\in\C$, with $|q|\neq 1$, and
let $z(x)$ be a meromorphic function over $\C$, not in $\C(x)$, such that
$z(qx)=a(x)z(x) +b(x)$, for some non-zero $a(x),b(x)\in\C(x)$.
Then $z(x)$
is differentially transcendental over $\C(x)$, i.e.,
the meromorphic function $z(x)$ and all its derivatives
are algebraically independent over $\C(x)$.
\end{thm}

A Galoisian proof of Ishizaki's result is given in Proposition 3.5 in \cite{HardouinSinger}.
In this section we are going to prove a similar result on transformal independence of solutions of
differential equations of the form $\de(y)=ay+b$. The ingredients of our proof are also difference analogues of the ones used in \cite{HardouinSinger}. For a general difference operator $\s$,
Proposition \ref{prop:generalisationishizaki} and \ref{prop:generalisationhishizakidependent}
give partial answers  to the difference version of Ishizaki's result whereas
in the case of a $q$-difference operator with transcendental $q$, Corollary \ref{cor:qishi}
together with Proposition \ref{prop:generalisationishizaki} give a complete solution.

\medskip
First of all we prove a lemma on rationality of solutions of equations of the form
$\de(y)=ay+b$, which does not involve the action of $\s$.

\begin{lemma}\label{lemma:solrat}
Let $L|K$ be an extension of $\de$-fields and assume that $v\in L$ satisfies $\de(v)=av$,
for some $a\in K$. If the equation $\de(y)-cy=b$, for $b,c\in K$, has a solution in $K(v)$,
then it already has a solution in $K$.
\end{lemma}

\begin{proof}
First assume that $v$ is transcendental over $K$.
Let $g \in K(v)$ be a solution of $\de(y)-cy=b$ and let us write $g= \frac{P}{Q} +R$ where $P,Q,R \in K[v]$,
$Q\neq 0$ and $\deg_v(P) <\deg_v(Q)$.
Then  $b=\de(g)-cg = \frac{\wtilde P}{\wtilde Q}  + \de(R) -cR$,
where $\wtilde P,\wtilde Q\in K[v]$ and $\deg_v(\wtilde P) <\deg_v(\wtilde Q)$.
Comparing the degree in $v$, we find $\de(R)-cR= b$.
This last equality implies that the constant term $r_0\in K$ of $R\in K[v]$, as a polynomial in $v$, satisfies $\de(r_0)-cr_0=b$.
\par
Now assume that $v$ is algebraic over $K$. Let $N$ denote the normal closure of $K(v)|K$. The derivation $\de\colon K(v)\to K(v)$ extends uniquely to a derivation $\de\colon N\to N$. Moreover, $\de\colon N\to N$ commutes with every automorphism of $N|K$. The operator $\mathcal{R}\colon N\to K$ given by $\mathcal{R}(d)=\frac{1}{[N:K]}\sum_{\tau}\tau(d)$, where $\tau$ ranges over all elements in the Galois group of $N|K$, is $K$-linear, the identity on $K$ and commutes with $\de$. Thus, if $z\in K(v)\subset M$ is a solution of $\de(y)-cy=b$, then $\mathcal{R}(z)\in K$ is a solution of $\de(y)-cy=b$.
\end{proof}

Let $K$ be a $\ds$-field and let $k:=K^\de$.
We consider a differential equation $\de(y)=ay+b$, with $a,b\in K$ and the associated differential system
\beq \label{eqn:associatedsystem}
\de\begin{pmatrix}y_1 \\ y_2 \end{pmatrix}= \begin{pmatrix} a & b \\
0 & 0 \end{pmatrix} \begin{pmatrix} y_1 \\ y_2  \end{pmatrix} .
\eeq
By a $\s$-Picard-Vessiot extension $L|K$ for $\de(y)=ay+b$, we  mean a
$\s$-Picard-Vessiot extension $L|K$ for \eqref{eqn:associatedsystem}.
So let $Z:=\begin{pmatrix}z_1&z_2\\c_1&c_2\end{pmatrix}$ be a fundamental solution of \eqref{eqn:associatedsystem}
with coefficients in $L$.
It follows directly from \eqref{eqn:associatedsystem} that $c_1,c_2\in k$ and that
$z_1c_2-z_2c_1 \neq 0$.
 Then, in the case $c_1c_2 \neq 0$, the matrix
$$
\begin{pmatrix}z_1c_1^{-1}-z_2c_2^{-1}& z_2c_2^{-1}\\0&1\end{pmatrix}
$$
is a fundamental solution matrix of \eqref{eqn:associatedsystem}.
So we can consider a fundamental solution matrix of \eqref{eqn:associatedsystem}, with coefficients in $L$, of the form
\beq\label{eq:goodsolution}
Y:=\begin{pmatrix} v& z \\ 0& 1 \end{pmatrix},
\hbox{~with $\de(v)=av$, $v\neq 0$, and $\de(z)=az+b$.}
\eeq
(The cases $c_1c_2 =0$ are similar and yield immediately to the existence of a fundamental solution matrix $Y$
of the same form as the one above).
Notice that the $\ds$-field extension $K\langle v\rangle_\s\subset L$ does not depend on the particular choice of
the non-zero solution $v\in L^\times$ of $\de(v)=av$, in fact two such solutions coincide up to a multiplicative factor in $k\subset K$.
\par
We consider $G:=\sgal(L|K)$ as a $\s$-closed subgroup of $\Gl_{2,k}$ via the fundamental solution matrix \eqref{eq:goodsolution} (I.e. we identify $G$ with its image in $\Gl_{2,k}$, via $\tau \mapsto [\tau]_Y$).
The action of an element $\tau\in G(S)$ on $R\otimes_k S$, where $R=K\{v,v^{-1}, z\}_\s$
is the $\s$-Picard Vessiot ring of $L$,
is given by $\tau(v\otimes 1)= v\otimes \alpha$ and $\tau(z\otimes 1)=v\otimes \beta +z\otimes 1$, for some $\a,\be\in S$.
Let $\mathbb G$ denote the algebraic subgroup of $\Gl_{2,k}$ given by
$$
\mathbb{G} (S) = \left\{ \begin{pmatrix}\alpha & \beta \\ 0 &1 \end{pmatrix} \Big|\ \alpha \in S^\times,\ \beta \in S \right\},
\hbox{~for any $k$-$\s$-algebra $S$.}
$$
For $\tau\in G(S)\leq\Gl_2(S)$ we have $\tau(Y\otimes 1)=(Y\otimes 1)[\tau]_Y$.
We see that $[\tau]_Y\in\mathbb{G}(S)$.
Therefore $G$ is contained in $\mathbb{G}$.
As in Section \ref{subsubsec:classg2}, let $\mathbb{G}_u$ denote the algebraic subgroup of $\mathbb{G}$ given by
$$
\mathbb{G}_u(S)=\left\{ \begin{pmatrix}1 & \beta \\ 0 &1 \end{pmatrix} \Big|\ \beta \in S \right\},
\hbox{~for any $k$-$\s$-algebra $S$,}
$$
and set $G_u=G\cap\mathbb{G}_u$. Since $\tau(v\otimes 1)= v \otimes \alpha$, it is clear from the Galois correspondence that
$G_u=\sgal(L|K\langle v\rangle_\s)$.
Moreover, we have $\de\l(\frac{z}{v}\r)=\frac{b}{v}$, so that $L|K\langle v\rangle_\s$ is a
$\s$-Picard-Vessiot extension for $\de(y)=\frac{b}{v}$.
The action of an element $\tau=\begin{pmatrix}1 & \beta \\ 0 &1 \end{pmatrix}\in G_u(S)=\sgal(L|K\langle v\rangle_\s)(S)$ on $L$ is given by
$\tau\l(\frac{z}{v} \otimes 1 \r)=\frac{z}{v}\otimes 1+ 1\otimes\beta$.
The situation is summarized in the following picture:
$$
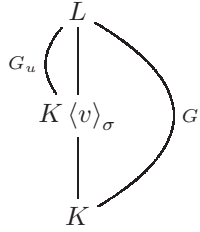
\xymatrix{  L \ar@{-}[d]  \ar@/_1pc/@{-}[d]_{G_u}  \ar@/^3pc/@{-}[dd]^{G} \\
K\left\<v \right \>_\s \ar@{-}[d]   \\
K }
$$
We remind that, if $G_u$ is properly contained in $\mathbb{G}_u$, it follows from
Theorem \ref{thm:classg2} that we only have two possible cases:
\begin{enumerate}
\item There exists an integer $n\geq 0$ such that $\s^n(\beta)=0$ for all $\begin{pmatrix}1 & \beta \\ 0 &1 \end{pmatrix}\in G_u(S)$ and all $k$-$\s$-algebras $S$.
\item There exist integers $n>m\geq 0$ such that $\s^n(\alpha)=\s^m(\alpha)$ for all $\begin{pmatrix}\alpha & \beta \\ 0 &1 \end{pmatrix}\in G(S)$ and all $k$-$\s$-algebras $S$.
\end{enumerate}

\begin{prop}\label{prop:generalisationishizaki}
Let $L|K$ be a $\ds$-field extension such that $k:=K^\de=L^\de$.
We fix $v,z\in L$, $v\neq 0$,
such that $\de(z)=az +b$ and $\de(v)=av$, with $a,b\in K$. We assume that $z \notin K$ and $K$ is inversive.
If $v$ is transformally independent over $K$, then $z$ is transformally independent over $K$.
\end{prop}

\begin{rmk}
Let us suppose that $z$ is transformally dependent over $K$, and hence over $K\<v\>_\s$.
Since $\de\l(\frac{z}{v}\r)=\frac{b}{v}$,
Theorem \ref{thm:rank1ga}
implies that there exist a non-zero homogeneous linear $\s$-polynomial $\sum_{j=0}^s\la_j\s^j(X)\in k\{X\}_\s$  and $g\in K\<v\>_\s$
such that
$$
\sum_{j=0}^s\la_j\s^j\l(\frac{b}{v}\r)=\de(g).
$$
The expression above, if non-trivial, provides a $\s$-relation for $v$,
with coefficients in $K$, and proves that $v$ is transformally dependent over $K$.
The problem of this argument is that there is no reason for the expression above to be non-trivial,
so that it does not provide a proof of Proposition \ref{prop:generalisationishizaki}.
\end{rmk}

\begin{proof}
As we have already done in several proofs, we can assume that $L$ is a $\s$-Picard-Vessiot extension
for $\de(y)=ay+b$.
First assume that $G_u=\mathbb{G}_u$. Then Proposition \ref{prop:dimensiondegtrans} implies that $\strdeg(L|K\langle v\rangle_\s)=\sdim(\mathbb{G}_u)=1$.
It follows that $z$ is transformally independent over $K\langle v\rangle_\s$. In particular, $z$ is transformally independent over $K$.
\par
If $G_u$ is a proper $\s$-closed subgroup of $\mathbb{G}_u$, we have to consider the two cases above.
In the first case, we have
$$
\tau\l(\s^n\l(\frac{z}{v}\otimes 1 \r)\r)=\s^n\l(\frac{z}{v}\otimes 1+1 \otimes \beta\r)=\s^n\l(\frac{z}{v}\otimes 1\r),
\hbox{~for any $\tau=\begin{pmatrix}1 & \beta \\ 0 &1 \end{pmatrix}\in G_u(S)$ and any $k$-$\s$-algebra $S$.}
$$
The Galois correspondence implies that
$\s^n\l(\frac{z}{v}\r) \in K\left\<v \right\>_\s$ and thus $\s^n(z) \in  K\left\<v \right\>_\s$. Since $z
\notin K$ and $K$ is inversive, we get that $\s^n (z) \notin K$.  Since $v$ is transformally independent over $K$, every element of the field $K\left\<v \right\>_\s$, that does not belong to
$K$, is transformally independent over $K$. Hence $z$ is transformally independent over $K$.
\par
In the second case,
$$
\tau\left(\frac{\s^n(v)}{\s^m(v)}\otimes 1 \right)=\frac{\s^n( v \otimes \alpha)}{\s^m( v \otimes \alpha)}=\frac{\s^n(v)}{\s^m(v)} \otimes 1,
\hbox{~for any $ \tau=\begin{pmatrix}\alpha & \beta \\ 0 &1 \end{pmatrix}\in G(S)$ and any $k$-$\s$-algebra $S$.}
$$
 Therefore $\s^n(v)=f\s^m(v)$ for some $f\in K^\times$. This contradicts the assumption that $v$ is transformally independent over $K$, so the second case can not occur.
\end{proof}

When we specialize Proposition \ref{prop:generalisationishizaki} to $K=\C(x)$
endowed with the shift or the $q$-difference operator, we find:

\begin{cor} \label{cor:ishispecial}
%In the situation \eqref{shift} (resp. \eqref{qdiff}),
Let $L|\C(x)$ be a $\ds$-field extension with $L^\de=\C$ and assume that
$z\in L$ satisfies $\de(z)=az+b$ with $a,b\in\C(x)$, $a\neq 0$. If one of the following hypothesis
is satisfied:
\begin{enumerate}
\item
assumption \eqref{shift} and $a\neq P + \frac{1}{N}\frac{\de(f)}{f}$
for any $P\in \C[x]$, $f \in \C(x)^\times$ and $N\in \Z^\times$;

\item
assumption \eqref{qdiff}, $q$ is transcendental and
$a \neq \frac{c}{x} +\frac{1}{N}\frac{\de(f)}{f}$, for any $c \in \C$, $N\in\Z^\times$
and $f \in \C(x)^\times$;

\item
assumption \eqref{qdiff}, $q$ is algebraic and
$a \neq P+Q +\frac{1}{N}\frac{\de(f)}{f}$, for any $P\in \C[x]$, $Q\in\C\l[\frac{1}{x}\r]$,
$N\in\Z^\times$ and $f \in \C(x)^\times$;
\end{enumerate}
then $z$ is transformally independent over $\C(x)$.
\end{cor}

\begin{proof}
We can assume without loss of generality that there exists a non-zero $v\in L$ with $\de(v)=av$. Indeed, since $L^\de=\C$ is algebraically closed we can replace $L$ with a $\s$-Picard-Vessiot extension for $\de(y)=ay$ of $L$. It then suffices to combine Proposition \ref{prop:generalisationishizaki} above with
Corollary \ref{cor:rank1shift} and Corollary \ref{cor:rank1q}.
\end{proof}

When $v$ is transformally dependent, we have only a partial result:

\begin{prop}\label{prop:generalisationhishizakidependent}
Let $K$ be a $\ds$-field, with $\s\de=\de\s$  and $k:=K^\de$.
Let $L|K$ be a $\s$-Picard-Vessiot extension for the equation $\de(y)=ay+b$, with $a,b\in K$.
Let  $z\in L$, with $z\not\in K$, such that $\de(z)=az +b$.
\par
If there exists $f \in K^\times$ and $\lambda \in k^\s$ such that $a = \lambda + \frac{\de(f)}{f}$, then
$z$ is transformally dependent over $K$ if and only if there exist
a non-zero homogenous linear $\s$-polynomial $\cL(X)\in k\{X\}_\s$
and an element $g \in K$
such that $\cL\l(\frac{b}{f}\r)=\de(g)-\la g$.
\end{prop}

\begin{proof}
The assumption that there exist $f \in K^\times$ and $\lambda \in k^\s$ such that $a = \lambda + \frac{\de(f)}{f}$
implies that any solution $v$ of $\de(y)=ay$ in $L$ is transformally dependent over $K$.
We fix a non-zero solution $v\neq 0$ of $\de(y)=ay$ in $L$ and we
set $\wtilde v=\frac{v}{f}$ and $\wtilde z=\frac{z}{f}$.
Then $\wtilde v$ and $\wtilde z$ satisfy the differential equations
$$
\de\l(\wtilde v\r)
=\left(a-\frac{\de(f)}{f}\right)\wtilde v
=\lambda\wtilde v
\hbox{~~and~~}
\de\l(\wtilde z\r)=\lambda\wtilde z+\frac{b}{f}.
$$
Since $\de$ and $\s$ commute to each other and $\s(\la)=\la$, we have that
$\de(\s(\wtilde v))=\s(\la)\s(\wtilde v)=\la\s(\wtilde v)$.
Therefore $\wtilde v$ and $\s(\wtilde v)$ satisfy the same differential equation
and there exists
$\mu \in k^\times$ such that $\s(\wtilde v)=\mu\wtilde v$.
Hence $K\<v\>_\s=K(v)$.
\par
Assume that there exist a non-zero homogeneous linear $\s$-polynomial $\cL \in k\{X\}_\s$
and $g \in K$ such that
$\cL\l(\frac{b}{f}\r) =\de(g)-\lambda g$. Then
$$
\de\l(\cL\l(\wtilde z\r)-g\r)
=\cL\l(\lambda\wtilde z+\frac{b}{f}\r)-\de\l(g\r)
=\lambda\l(\cL\l(\wtilde z\r)-g\r).
$$
Consequently, there exists $\kappa \in k$ such that  $\cL(\wtilde z)-g=\kappa \wtilde v$. Since $\wtilde v$ is transformally dependent over $K$, this shows that $z$ is transformally dependent over $K$.
\par
Conversely, assume that $z$ is transformally dependent over $K$. Since $L|K\left\<v \right\>_\s$ is a $\s$-Picard-Vessiot extension for $\de(y) =\frac{b}{v}$
and $\frac{z}{v}$ is transformally dependent over $K\left\<v \right\>_\s$,
Theorem \ref{thm:rank1ga} implies that
there exist a non-zero homogeneous linear $\s$-polynomial $\cL \in k\{X\}_\s$ and
$g \in K\left\<v\right\>_\s$ such that $\cL\l(\frac{b}{v}\r)=\de(g)$.
Since $\frac{b}{v}=\frac{b/f}{\wtilde v}$ and $\sigma(\wtilde v)=\mu\wtilde v$, there exists a non-zero homogeneous linear $\s$-polynomial $\wtilde{\cL}\in k\{X\}_\s$ such that
$\wtilde\cL\l(\frac{b}{f}\r)=\wtilde v\cL\l(\frac{b}{v}\r)$. Then
$$
\wtilde{\cL}\l(\frac{b}{f}\r)=\de\l(g\r)\wtilde v=\de\l(g\wtilde v\r)-\lambda g\wtilde v.
$$
So the equation $\de\l(y\r)-\lambda y=\wtilde{\cL}\l(\frac{b}{f}\r)$ has a solution in $K\left\<v\right\>_\s = K\l(v\r)$.
Now  Lemma \ref{lemma:solrat} provides an element $\wtilde{g} \in K$ such
that $\wtilde{\cL}\l(\frac{b}{f}\r) =\de\l(\wtilde{g}\r)-\lambda\wtilde{g}$.
\end{proof}

In order to apply Proposition \ref{prop:generalisationhishizakidependent}, we want $\s$ and $\de$
to commute: To ensure that
this hypothesis is satisfied in the $q$-difference case,
we take $\de=x\frac{d}{dx}$, rather than $\de=\frac{d}{dx}$.
Moreover we assume $q\in\C$ to be a
transcendental number.
\par
Notice that, if $\frac{dy}{dx}=by$, with $b\in\C(x)$, then $\de(y)=xby$.
Therefore a solution $v$ of $\de(y)=ay$, in a convenient extension of $\C(x)$, is transformally
dependent over $\C(x)$ if and only if
$a = c +\frac{1}{N}\frac{\de(f)}{f}$ for some $c \in \C$, $N\in\Z^\times$ and $f \in \C(x)^\times$
(see Corollary \ref{cor:rank1q}).

\begin{cor}\label{cor:qishi}
Let $L|\C(x)$ be a $\ds$-field extension, with $\de=x\frac{d}{dx}$, $L^\de=\C$ and $\s:f(x)\mapsto f(qx)$,
for some transcendental $q\in\C$,
and assume that
$z\in L$, $z\notin\C(x)$, satisfies $\de(z)=az+b$, with $a,b\in\C(x)$.
If $a = c +\frac{1}{N}\frac{\de(g)}{g}$ for some $c \in \C$, $N\in\Z^\times$ and $g \in \C(x)^\times$, then there are two cases:
\begin{enumerate}
\item
If $\frac{1}{N}\frac{\de(g)}{g}= \frac{\de(f)}{f}$, for some $f \in \C(x)^\times$, then $z$ is
transformally dependent over $\C(x)$
if and only if there exist a non-zero homogeneous linear $\s$-polynomial $\cL \in \C\{X\}_\s$
and $h \in \C(x)$
such that $\mathcal{L}\l(\frac{b}{f}\r) =\de(h)-ch$.

\item
If $\frac{1}{N}\frac{\de(g)}{g} \neq \frac{\de(f)}{f}$, for all $f \in \C(x)^\times$, then $z$ is
transformally dependent over $\C(x)$
if and only if there exists  $h \in \C(x)$ such that $b=\de(h)-ah$.
\end{enumerate}
\end{cor}

\begin{rmk}
For a general difference operator $\s$, we only have
a partial answer to the question of transformal dependencies of solutions of $\de(y)=ay +b$.
On the other hand, for a $q$-difference operator with $q$ transcendent,
Corollary \ref{cor:qishi}
together with Corollary \ref{cor:ishispecial} (ii)
describe completely the transformal relations of the solutions in terms of relations satisfied by the
coefficients $a$ and $b$.
\end{rmk}

\begin{proof}
As in the proof of Corollary \ref{cor:ishispecial} we can assume that there exists a non-zero $v\in L$ with $\de(v)=av$. The first assertion follows immediately from Proposition \ref{prop:generalisationhishizakidependent} above. Let us prove the second one.
Suppose that $a = c +\frac{1}{N}\frac{\de(g)}{g}$,
with $\frac{1}{N}\frac{\de(g)}{g}\neq\frac{\de(f)}{f}$, for all $f \in \C(x)^\times$, so that $v$ is transformally dependent over $\C(x)$.
If there exists $h \in \C(x)$, such that $b=\de(h)-ah$, then $z-h$ is a solution of $\de(y)=ay$.
Hence there exists $\lambda \in \C$ such that $z-h=\lambda v$, which proves that
$z$ is transformally dependent over $\C(x)$.
\par
Conversely,
let us assume that $z$ is transformally dependent over $K=\C(x)$.
With no loss of generality, we can assume that $L=K\<v,z\>_\s$ is a $\s$-Picard-Vessiot extension for
$\de(z)=az+b$.
We refer to the discussion and the notation at the beginning of the section about the
devissage of $G=\sgal(L|K)$.
We remind that the action of an element $\tau\in G(S)$ on $R\otimes_k S$, where $R=K\{v,v^{-1}, z\}_\s$
is the $\s$-Picard Vessiot ring of $L$ and $S$ is a $k$-$\s$-algebra,
is given by $\tau(v\otimes 1)= v \otimes \alpha$ and $\tau(z\otimes 1)= v \otimes \beta +z \otimes 1$, for some $\a,\be\in S$.
Since $L|K\left\<v \right\>_\s$ is a $\s$-Picard-Vessiot extension for $\de(y) =\frac{b}{v}$ and $\frac{z}{v}$ is transformally dependent over $K\left\<v \right\>_\s$, we know that $G_u=G\cap\mathbb{G}_u$ is properly contained in $\mathbb{G}_u$.
According to Theorem \ref{thm:classg2} we have to study two cases:
For all $\tau\in G(S)$ and all $k$-$\s$-algebras $S$,
either there exists an integer $n\geq 0$ such that $\s^n(\beta)=0$ or
there exist integers $n>m\geq 0$ such that $\s^n(\alpha)=\s^m(\alpha)$.
\par
In the first case, we have
$$
\tau\l(\s^n\l(\frac{z}{v} \otimes 1\r)\r)=\s^n\l(\frac{z}{v} \otimes 1+1\otimes\beta\r)=\s^n\l(\frac{z}{v}\otimes 1\r),
\hbox{~for any $\tau\in G_u(S)$ and any $k$-$\s$-algebra $S$.}
$$
The Galois correspondence implies that
$\s^n\l(\frac{z}{v}\r) \in K\left\<v \right\>_\s$ and thus $\s^n(z) \in  K\left\<v \right\>_\s$.
This means that $\de(y)-\s^n(a)y=\s^n(b)$ has a solution $\s^n(z)\in K\left\<v\right\>_\s$.
There exists a non-negative integer $m$ such that $\s^n(z)\in K(v,\ldots,\s^m(v))$.
Applying recursively Lemma \ref{lemma:solrat},
we find that there already exists a solution
$\wtilde{h}\in K$ of $\de(y)-\s^n(a)y=\s^n(b)$.
So, if $h\in K$ with $\s^n(h)=\wtilde{h}$, then $\de(h)-ah=b$.
\par
Now consider the second case. Since $K$ is inversive, we know that $G$ is $\s$-reduced.
So we can assume that $m=0$ (see Proposition 4.3 in \cite{articleone}).
Since for any $\tau\in G(S)$ we have
$$
\tau\l(\frac{\s^n(v)}{v} \otimes 1\r)=
\frac{\s^n(\tau(v\otimes 1))}{\tau(v \otimes 1)}=\frac{\s^n(v)}{v}\otimes 1,
$$
we conclude that $\s^n(v)=hv$ for some $h\in\C(x)^\times$.
This implies
$\s^n\l(\frac{\de\l(v\r)}{v}\r)=\frac{\de\l(h\r)}{h}  + \frac{\de\l(v\r)}{v}$, which yields $\s^n\l(a\r)-a=\frac{\de\l(h\r)}{h}= \frac{1}{N}\frac{\de\l(\s^n\l(g\r)/g\r)}{\s^n\l(g\r)/g}$. So there exists
$\mu \in \C^\times$  such that $\frac{\s^n\l(g\r)}{g} =\mu h^N$. It is now easy to see that this last equality implies that there exists
 $f \in \C\l(x\r)^\times$ with
$g=f^N$. But then $\frac{1}{N}\frac{\de\l(g\r)}{g}=\frac{\de\l(f\r)}{f}$. So the second case can not occur.
\end{proof}

\begin{rmk}
The existence of a non-zero homogeneous linear $\s$-polynomial $\cL \in \C\{X\}_\s$,
such that
$\cL(b)=\de(y)-ay$ has a solution in $\C(x)$ does not imply that $b=\de(y)-ay$ has a solution $\C(x)$,
differently from the case of linear difference equations with a differential parameter.
See Lemma 6.4 in \cite{HardouinSinger}.
Take for example $a=1$ and $b=x$. Then $\frac{1}{q}\s(x)-x=0$, but the differential equation
$x\frac{d}{dx} y-y=x$ has no solution in $\C(x)$.
\end{rmk}

%%%%%%%%%%%%%%%%%%%%%%%%%%%%%%%%%%%%%%%%%%%%%%%%%%%%%%%%%%%%%%%%%%%%%%%%%%%%%
%%%%%%%%%%%%%%%%%%%%%%%%%%%%%%%%%%%%%%%%%%%%%%%%%%%%%%%%%%%%%%%%%%%%%%%%%%%%%%
\section{Inverse problem for $\s$-closed subgroups of $\mathbf G_a$}\label{sec:pbinv}
%%%%%%%%%%%%%%%%%%%%%%%%%%%%%%%%%%%%%%%%%%%%%%%%%%%%%%%%%%%%%%%
%%%%%%%%%%%%%%%%%%%%%%%%%%%%%%%%%%%%%%%%%%%%%%%%%%%%%%%%%%%%%%%

Let $k(x)$ be a differential field with $\de(x)=1$ and $k=k(x)^\delta$ algebraically closed. In the classical Galois theory of linear differential systems with coefficients in  $k(x)$, the inverse problem consists in determining all the linear algebraic groups, which
occur as Galois group of a linear differential system over $k(x)$. This question, which is also connected
to the Riemann-Hilbert correspondence, was addressed by many mathematicians, most recently by  J. Hartman, J. Kovacic, C. Mitschi,   J.P. Ramis, M.F. Singer,... In \cite{Hartinv}, many references are given and it
is proved that any linear algebraic group over $k$ is the Galois group of a linear differential system over
$k(x)$.

Recently Cassidy and Singer developed a parametrized Galois theory of linear
differential systems, which takes into account  continuous actions on auxiliary parameters.
For instance, if $k$ above is a ``sufficiently large'' $\frac{d}{dt}$-field extension
of $\C(t)$, Cassidy and Singer are able to attach to a differential system $\de(y)=A(x,t) y$,
a parametrized Galois group, which is a linear $\frac{d}{dt}$-differential algebraic group
in the sense of Kolchin. In this setting, an inverse problem
was also addressed. The first answers were given in \cite{Dreyfus:density} and \cite{MitschiSinger:MonodromyGroupsOfParameterizedLinearDifferentialEquationsWithRegularSingularities}. Using transcendental descriptions
via monodromy and Stokes, the authors were able to show, for instance, that a  linear
$\frac{d}{dt}$-differential algebraic group occurs as parametrized Galois group in the sense of \cite{cassisinger} if and only if it posses   a  Kolchin dense finitely generated subgroup. A purely
algebraic characterization was then  given in \cite{Singerinv}, where it is proved that
a linear $\frac{d}{dt}$-differential algebraic group $G$ over $k$ is
a parametrized Galois group if and only if its identity component has no quotients isomorphic to $\Ga$
or $\Gm$.

The aim of the present section is to investigate some similar questions for  a
discrete action on the parameter. Since $\Ga$ appears as an obstruction in the continuous
parameter case, we restrict our study of the inverse problem to $\s$-closed subgroups
of $\Ga$. We show that one can realize $\Ga$ itself over some $k(x)$ and, thereby, that any of
the $\s$-closed subgroups of $\Ga$ is a $\s$-Galois group over some $\ds$-field
extension of $k(x)$, by $\s$-Galois correspondence.

\medskip
In this section, we assume that we are in the following situation:

\begin{quote}
($\mathcal H_1$)
Let $K$ be a $\ds$-field, with $k=K^\de$ algebraically closed and aperiodic with respect to $\s$,
i.e., for any positive integer $n$, there exists $a\in k$ such that
$\s^n(a)\neq a$.
We consider a $\s$-Picard-Vessiot extension $L|K$, having $\s$-Galois group isomorphic to a
$\s$-closed subgroup of the additive group $\Ga$.
\end{quote}

In the first subsection, we give a decomposition theorem of such extensions (Proposition \ref{prop:invpred}) into a ``$\s$-infinitesimal'' and a perfectly $\s$-reduced part. Moreover,
under additional assumptions on $K$, we show that $L$ is the $\s$-Picard-Vessiot extension of
some differential equation of the form $\de(y)=g$. In a second section, we finally restrict ourselves
to the case $K=k(x)$,  where $k(x)$ is a $\ds$-field extension of $\C(t,x)$ endowed with
$\de(x)=1, \s(x)=x$ and $\de(t)=0, \s(t)=t+1$. This situation can be seen as the discrete counterpart of
the one studied in \cite{Singerinv}. Assuming that $k$ is linearly
$\s$-closed, we show that any perfectly $\s$-reduced $\s$-closed subgroup of $\Ga$ is a $\s$-Galois
group over $k(x)$ (see Proposition \ref{prop:inverseproblemH2}).

\subsection{Some structure theorems}

The next lemma is a first step in the description of the structure of the extension $L|K$.

\begin{lemma}\label{lemma:inverseproblem}
Assume that we are in the situation ($\mathcal H_1$).
Then there exist $t_1,\dots,t_s\in L$ such that $L=K\<t_1,\dots,t_s\>_\s$,
and $\de(t_1),\dots,\de(t_s)\in K$.
\end{lemma}

\begin{proof}
We can suppose that $L =K \left \<Y \right \>_\s$ where $Y \in \Gl_n(L)$ is a fundamental solution
matrix of a linear differential system $\de(y)=Ay$ of order $n$, with coefficients in $K$.
We denote by $\sgal(L|K)[0]$
the Zariski closure of  $\sgal(L|K)$ inside  $\Gl_{n,k}$ and  by $k\left[\sgal(L|K)[0]\right]$
its $k$-Hopf algebra.
\par
Let $G^\p$ be a $\s$-closed subgroup of $\Ga$
and $\phi: G^\p \rightarrow \sgal(L|K)$ be an isomorphism  of group $k$-$\s$-schemes
(see \S A.8 in \cite{articleone}), so that
$$
\phi^*:k\{\sgal(L|K)\}\rightarrow k\{G^\p\}
$$
is an isomorphism of $k$-$\s$-Hopf algebra.
We consider the $\s$-Picard-Vessiot ring $R =K\left\{ Y, \frac{1}{\det(Y)} \right \}_\s$
inside $L$ (see  Corollary \ref{prop:pvringpvext}).
As in Proposition \ref{prop:defgal}, we can identify $k\{\sgal(L|K)\}$ with $k\{Z,\frac{1}{\det(Z)}\}_\s$ where $Z=(Y \otimes 1)^{-1}(1 \otimes Y) \in (R\otimes R)^\de$.
Moreover, $k\{G^\p\}$ is a quotient of the ring of $\s$-polynomials in one variable,
i.e., it has the form $k\{\overline{x}\}_\s$.
Therefore, there exists $r \in \Z_{\geq 0}$ such that $\phi^*(Z)$
is an element of the ring $k[\overline{x},\s(\overline{x}),\dots,\s^r(\overline{x})]$.
This means in particular that the map $\phi^*$
injects $k\l[\sgal(L|K)[0]\r]$ in a sub-$k$-algebra of $k\left[\overline{x},\s(\overline{x}),\dots,\s^r(\overline{x})\right]$, i.e.,
$\sgal(L|K)[0]$ can be identified  with a quotient of some $\Ga^r$.
The quotients of an algebraic vector group are   algebraic vector groups, hence there exists $s \in \Z_{>0}$ such that $\sgal(L|K)[0]$ is isomorphic to  $\Ga^s$.
\par
We have proved that $\sgal(L|K)[0]\cong \mathbf G_a^s$ is the differential Galois group
of the Picard-Vessiot ring $R_0:=K[Y,\frac{1}{\det(Y)}]\subset R$, over $K$.
Finally, $\Ga^s$ has a trivial first Galois cohomology group,
therefore the $\sgal(L|K)[0]$-torsor $R_0$ is trivial
(see Theorem 1.28 \cite{vdPutSingerDifferential}). In other words, there exist
$t_1,\hdots,t_s \in R_0$ such that
\begin{itemize}
\item $R_0 =K[t_1, \hdots,t_s]$,
\item for all $\tau \in \sgal(L|K)[0](k)=\Aut^\de(R_0|K)$, we have $\tau(t_i) =t_i + c_i(\tau)$ with $c_i(\tau ) \in k$.
\end{itemize}
For any $\tau \in \sgal(L|K)[0](k)$ and any $i=1,\dots,s$, we have
$\tau(\de(t_i))=\de(t_i + c_i(\tau))=\de(t_i)$.
The differential Galois correspondence (see Proposition 1.34 in \cite{vdPutSingerDifferential}) implies
that $\de(t_i) \in K$, for all $i=1,\hdots,s$.
\par
To summarize, we have found $t_1,\hdots,t_s \in L$ such that
$K\l[Y,\frac{1}{\det Y}\r]=K[t_1, \hdots,t_s]$, and hence such that
$L=K \left \<t_1,\hdots,t_s  \right \>_\s$,
with $\de(t_i)\in K$ for all $i=1,\hdots,s$.
\end{proof}

Now, we treat the case where the $\s$-Galois group is the whole $\Ga$.

\begin{prop}\label{prop:gared}
Under the assumption ($\mathcal H_1$), suppose that
$\sgal(L|K)$ is isomorphic to $\Ga$.
Then, there exists $\theta \in L$ such that $L=K\left\< \theta \right\>_\s$
and $\de(\theta)\in K$.
\end{prop}

\begin{proof}
It follows from Lemma \ref{lemma:inverseproblem} that
we can find $t_1,\hdots,t_s \in L$ such that $L=K\l\<t_1,\hdots,t_s\r\>_\s$,
with $\de(t_i)\in K$. Moreover, $R=K\l\{t_1,\hdots,t_s\r\}_\s$ is a $\s$-Picard-Vessiot ring of $L$.
For every $k$-$\s$-algebra $S$ and  every $\tau \in \sgal(L|K)(S)$,
there exist $c_{1}(\tau), \hdots,c_{s}(\tau) \in S^s$,
such that in $R\otimes_k S$ we have: $\tau(t_i \otimes 1)=t_i \otimes 1 +1 \otimes c_{i}(\tau)$ for all $i=1,\hdots, s$.
This induces a $\s$-morphism
$\phi : \sgal(L|K ) \rightarrow \Ga^s$,
such that $\phi(\tau) =(c_{1}(\tau), \hdots,c_{s}(\tau))$,
for every $k$-$\s$-algebra $S$ and every $\tau \in \sgal(L|K)(S)$.
If $\tau \in \sgal(L|K)(S)$ is the identity on $t_1\otimes 1,\dots, t_s\otimes 1$, then it is the identity on
$R\otimes S$. This proves that $\phi$ has a trivial kernel.
\par
Now, we identify $\sgal(L|K )$ with $\Ga$ and $\tau \in \sgal(L|K)$ with its image, say $c(\tau)$, in $\Ga$. Then, since $\phi$ is a morphism of $\s$-algebraic groups from  $\Ga$ to  $\Ga^s$ \footnote{ In zero characteristic,  morphisms of $\s$-algebraic groups from $\Ga$ to $\Ga$ are given by linear $\s$-polynomials. This can be proved as in the algebraic case \cite[Theorem 8.4]{Waterhouse:IntrotoAffineGroupSchemes}. We omit this proof here.},   there exist homogeneous linear
$\s$-operators $L_1,\hdots,L_s$, i.e. elements of the skew euclidean ring $k[\s]$,
such that, for any $k$-$\s$-algebra $S$, we have:
$$
\begin{array}{rccc}
\phi:&\mathbf G_a(S) &\longrightarrow&\mathbf G_a^s(S)\\
&\tau&\longmapsto& (L_1(c(\tau)),\dots,L_s(c(\tau))).
\end{array}
$$

Since $\phi$ is injective, the $\s$-operators $L_1,\hdots,L_s$ do not annihilate simultaneously, and hence
$1$ belongs to the ideal $k[\s]L_1+\hdots+k[\s]L_s$ of $k[\s]$.
Therefore there exists $N_1,\dots,N_s\in k[\s]$ such that
$N_1 L_1 + \hdots N_s L_s =1$ in $k[\s]$.
We set $\theta := \sum_{i=1}^s N_i(t_i) \in L$. Then,
for every $k$-$\s$-algebra $S$ and  every $\tau \in \sgal(L|K)(S)$, we have
$\tau(\theta \otimes 1) = \sum_{i=1}^s N_i(t_i\otimes 1 +
1\otimes L_i(c(\tau)))= \sum_{i=1}^s N_i(t_i \otimes 1) +  1\otimes c(\tau)=\theta \otimes 1 + 1 \otimes c(\tau)$ in $R\otimes_k S$.
This means, that the only element of $\sgal(L|K)$ that fix $\theta$ is the identity.
By the $\s$-Galois correspondence, we finally get that $L=K\left \< \theta \right\>_\s$.
It follows from the construction itself that we have $\de(\theta) \in K$.
\end{proof}

The situation is more complicated when $\sgal(L|K)$ is isomorphic to a proper $\s$-closed subgroup of
$\mathbf G_a$.

\begin{prop}
\label{prop:invpred}
Under the assumption ($\mathcal H_1$), suppose that
$\sgal(L|K)$ is isomorphic to a proper subgroup of $\mathbf G_a$.
Then, we have:
\begin{enumerate}
\item
There exist $\theta \in L$ and $n \in \Z_{\geq 0}$, such that
$\s^n(L)\subset K\left\< \theta \right\>_\s$ and
$\de(\theta)\in K$.
\item
The $\ds$-field $K\left\< \theta \right\>_\s$ is a $\s$-Picard-Vessiot extension
of $K$ for $\de(y)=\de(\theta)$ and we have
a short exact sequence
\beq\label{eq:exactsequence}
0 \rightarrow \sgal(L| K\left\< \theta \right\>_\s) \rightarrow  \sgal(L| K) \rightarrow  \sgal( K\left\< \theta \right\>_\s|K) \rightarrow 0.
\eeq
Moreover, the  group  $ \sgal(L| K\left\< \theta \right\>_\s)$ is $\s$-infinitesimal, i.e.,
$\sgal(L| K\left\< \theta \right\>_\s)(S)$ is trivial for any $\s$-reduced $k$-$\s$-algebra $S$.
In addition, we can choose $\theta$ such that
$ \sgal( K\left\< \theta \right\>_\s|K)$ is perfectly $\s$-reduced.
\end{enumerate}
\end{prop}

\begin{rmk}
The previous proposition gives a  decomposition of the $\s$-Picard-Vessiot extension $L|K$, with $\s$-Galois group isomorphic to a proper  subgroup of $\Ga$,
in a tower of $\s$-field extensions such that $ K\left\< \theta \right\>_\s|K$ is
perfectly $\s$-separable and $L| K\left\< \theta \right\>_\s$ is what we could call a
``purely $\s$-inseparable extension''.
\end{rmk}

\begin{proof}
Lemma \ref{lemma:inverseproblem} implies that there exist $t_1,\dots,t_s\in L$
such that $L=K\<t_1,\dots,t_s\>_\s$,
and $\de(t_1),\dots,\de(t_s)\in K$.
Notice that the assumption on $\sgal(L|K)$ and Proposition \ref{prop:dimensiondegtrans}
imply that
$$
\strdeg(L|K) = \sdim_k(\sgal(L|K))=0.
$$
Therefore, it follows from Theorem 4.5.4 in \cite{Levin} that
there exist $c_1,\hdots,c_s \in k$ and $n \in \Z_{\geq 0}$
such that, if we take $\theta =\sum_{i=1}^s c_i t_i$, then  $\s^n(L) \in K\left\< \theta \right\>_\s$.
By construction, $\de(\theta)\in K$. This proves (i).
\par
Notice that $K\left\< \theta \right\>_\s^\de \subset L^\de =k$, which implies
that $K\left\< \theta \right\>_\s$ is a $\s$-Picard-Vessiot extension for $\de(y)=\de(\theta)$.
The second fundamental theorem of the $\s$-Galois correspondence
(see Theorem \ref{theo:secondfundamentaltheorem}) gives immediately the exact sequence of group $k$-$\s$-schemes \eqref{eq:exactsequence}.

\par Now, we show that $\sgal(L| K\left\< \theta \right\>_\s)$ is $\s$-infinitesimal.  Let $S$ be a $\s$-reduced $k$-$\s$-algebra, let $R$ be the $\s$-Picard-Vessiot ring inside $L$ and let
$\tau \in \sgal(L| K\left\< \theta \right\>_\s)(S)$.
Since $\s^n(L) \in K\left\< \theta \right\>_\s$,  the action of $\tau$ on $\s^n(R)\otimes S$ is trivial.
Since $\tau$ commutes with $\s$ and $\s$ is injective on $S$, we deduce that $\tau$ is the identity
by looking at its action on a fundamental solution matrix for $L|K$.
\par
It remains to prove that we can choose $\theta$ such that $\sgal(K\left\< \theta \right\>_\s|K)$ is perfectly $\s$-reduced.
The $\ds$-field extension  $K\left\< \theta \right\>_\s|K$ is a $\s$-Picard-Vessiot extension for $\de(y)=\de(\theta)$,
with fundamental solution matrix $\Theta:= \begin{pmatrix}1&  \theta \\
0 & 1 \end{pmatrix}$. We can embed  $\sgal(K\left\< \theta \right\>_\s|K)$ into $\Ga$ via $\Theta$ (we identify $\Ga$ with its image in $\Gl_{2,k}$ via $a \mapsto \begin{pmatrix}1& a \\
0 & 1 \end{pmatrix}$).
If $\sgal(K\left\< \theta \right\>_\s|K) =\Ga$, we are done.
Otherwise, Theorem \ref{thm:clasga}
implies that
$\I( \sgal(K\left\< \theta \right\>_\s| K)) \subset k\left\{x\right\}_\s=k\{\Ga\}$ is generated, as $\s$-ideal, by a $\s$-polynomial of the form
$\sum_{i=0}^s \lambda_i \s^i(x)$, for some $\lambda_0,\hdots,\lambda_s \in k$ with $\lambda_s \neq 0$.
Let $r=\min\{i=0,\dots,s\vert \lambda_i\neq 0\}$ be minimal.
We set $\theta_1 := \s^r(\theta)$. Then, $\de(\theta_1) \in K$ and $\s^{n+r}(L)\subset K\left\< \theta_1 \right\>_\s$.
The $\s$-Galois correspondence for the $\s$-Picard-Vessiot extension
$K\left\< \theta \right\>_\s|K$ implies that
$\sum_{i=0}^{s-r} \lambda_{i+r} \s^i(\theta_1)$ belongs to $K$.
 Using again the $\s$-Galois correspondence for the $\s$-Picard-Vessiot extension
$K\left\< \theta_1 \right\>_\s| K$, we see that $\sum_{i=0}^{s-r} \lambda_{i+r} \s^i \in \I(\sgal(K\left\< \theta_1 \right\>_\s| K))$.
The classification
of the $\s$-closed subgroups of $\Ga$ (Corollary \ref{cor:ssubgroupofGa}) implies that $ \I(\sgal(K\left\< \theta_1 \right\>_\s| K))$ is generated by a
$\s$-polynomial of the form $\sum_{i=0}^l \mu_i\s^i$ with $\mu_0\neq 0$ and that $\sgal(K\left\< \theta_1 \right\>_\s| K)$
is perfectly $\s$-reduced.
\end{proof}

Now, we consider a more restrictive hypothesis on $K$, namely
we show that, if $K$ is inversive, then only  perfectly $\s$-reduced $\s$-closed subgroups of
$\Ga$ can be realized as $\s$-Galois groups.

\begin{prop}\label{prop:gareduced}
Assume that
the extension $L|K$ verifies ($\mathcal H_1$) and moreover that the $\s$-field $K$ is inversive.
Then $\sgal(L|K)$ is perfectly $\s$-reduced and there exists $\theta \in L$
such that
\begin{itemize}
\item $\de(\theta)\in K$ and
\item $L$ is a $\s$-Picard-Vessiot extension for $\de(y)=\de(\theta)$.
\end{itemize}
\end{prop}

\begin{proof}
If $\sgal(L|K)$ is isomorphic to $\Ga$ then there is nothing to prove (see Proposition \ref{prop:gared}).
If  $\sgal(L|K)$ is isomorphic to a proper $\s$-closed subgroup of $\Ga$, then, by Proposition \ref{prop:invpred},
there exist $\theta \in L$, $n \in \Z_{\geq 0}$ such that $\de(\theta)\in K$ and
$\s^n(L) \in K\left\< \theta \right\>_\s$.
Moreover, we can choose $n$ and $\theta$ such that $\sgal(K\left\<\theta\right\>_\s|K)$ is a perfectly $\s$-reduced proper $\s$-closed subgroup of $\Ga$.
\par
To conclude it suffices to prove that $K\<\theta\>_\s$ is inversive. In fact, this immediately implies
that $L=K\<\theta\>_\s$.
Embedding $\sgal(K\left\<\theta\right\>_\s|K)$ into $\Ga$ (via $\theta$), we see that
there exist $\lambda_0,\hdots,\lambda_s \in k$ with $\lambda_0 \neq 0$
such that $\sum_{i=0}^s \lambda_i \s^i(x)\in k\{x\}_\s$
generates $\I(\sgal(K\left\<\theta\right\>_\s|K))$.
Using the $\s$-Galois correspondence, we see that
$\sum_{i=0}^s \lambda_i \s^i (\theta) \in K$. Since $K$ is inversive and $\lambda_0 \neq 0$, this implies
that $K\left\<\theta\right\>_\s$ is inversive
and therefore that $L=K\left\<\theta\right\>_\s$.
\end{proof}

\subsection{A discrete parameter inverse problem over $\C(x,t)$}

We now restrict our attention to the following situation:

\begin{quote}
($\mathcal H_2$)
Let $\C(t)$ be  the $\s$-field with  $\s|_\C =id$ and
 $\s(t)=t+1$.
Let  $k$ be  an algebraically closed inversive linearly $\s$-closed $\s$-field extension of $\C(t)$. We now assume
that $K=k(x)$ endowed with the derivation $\de(x)=1, \de(c)=0$ for all $c \in k$ and the endomorphism extending the action of $\s$ on $k$ to $K$ with $\s(x)=x$.
\end{quote}

Since the assumption $(\mathcal H_2)$ satisfies the hypothesis of Proposition \ref{prop:gareduced}, we know that only perfectly $\s$-reduced $\s$-closed subgroups
of $\Ga$ can occur as $\s$-Galois groups for differential equations of the form $\de(y)=g$ over $K$.
 The following proposition shows that any of these subgroups appears.

\begin{prop}\label{prop:inverseproblemH2}
Suppose that $K$ verifies ($\mathcal H_2$).
Let $G$ be a  perfectly $\s$-reduced $\s$-closed subgroup of $\Ga$. Then,
there exists $g\in K$ such that $\de(y) =g$ has $\s$-Galois group $G$.
\end{prop}

\begin{proof}
First of all, let us realize $\Ga$ itself. Let $L|K$ be a $\s$-Picard-Vessiot extension for
$\de(y)=\frac{1}{x+t}$. Reasoning as in Corollary \ref{cor:shiftigaC}, we see that
$\sgal(L|K)=\Ga$. Notice that the coefficients of the equation lie in $\Q(x,t)$.

Now, let $G$ be a proper perfectly $\s$-reduced $\s$-closed subgroup of $\Ga$. Since $G$ is perfectly $\s$-reduced, we know that  there exist $\lambda_0,\hdots,\lambda_s \in k$,
with $\lambda_0\lambda_s\neq 0$,
such that $\I( G) \subset k\{x\}_\s=k\{\Ga\}$ is generated by $\mathcal{L} :=\sum_{i=0}^s \lambda_i \s^i(x)$.
Since $k$ is linearly $\s$-closed, we can find $c_1,\hdots,c_s \in k$, linearly independent over $k^\s$,
whose $k^\s$-span is the group $G(k)$
(see Corollary \ref{cor:linsclosedsubgroupsGa}).
\par
Let us consider the differential equation
$$
\de(y) =\sum_{i=1}^s \frac{c_i}{x+i}=:g.
$$

Let $L|K$ be a $\s$-Picard-Vessiot extension for $\de(y)=g$,
$z \in L$ be a solution and $R=K\{z\}_\s$ its $\s$-Picard-Vessiot ring,
in the sense of \S\ref{subsec:Ishizaki}.
We have
$$
\mathcal{L}(\de(z)) =\cL(g)= \sum_{i=1}^s \frac{\mathcal{L}(c_i)}{x+i} =0,
$$
so that $\de(\mathcal{L}(z)) =0$ and hence $\mathcal{L}(z) \in k$.
We embed
$\sgal(L|K)$ into $\Ga$, via its action on the solution $z$, i.e.,
for any $k$-$\s$-algebra $S$ and any $\tau \in\sgal(L|K)(S)$,
there exists $c(\tau)\in S$ such that $\tau(z\otimes 1) =z\otimes 1 + 1 \otimes c(\tau) \in R\otimes_k S$.
We write $\tau(z)=z+c(\tau)$ to simplify the notation.
Since $\cL(z)\in k$, we have $\cL(z)=\tau(\cL(z))=\cL(\tau(z))=\cL(z)+\cL(c(\tau))$,
hence $\cL(c(\tau))=0$, which means that $\sgal(L|K) \subset G$.
To conclude we have to prove the inverse inclusion.
Let $\cL_1$ be a homogeneous linear $\s$-polynomial in $k\{x\}_\s$,
generating the vanishing ideal of $\sgal(L|K)$ inside $\Ga$ and $z$
be a solution of $\de(y)=g$, as above.
By $\s$-Galois correspondence
we find that $\mathcal{L}_1(z)$ belongs to $K$ and, thus,
that
$$
\de(\cL_1(z))=\cL_1(\de(z))=\cL_1(g)=\sum_{i=1}^s \frac{\mathcal{L}_1(c_i)}{x+i}\in\de(K).
$$
Reasoning as in  Corollary \ref{cor:gaexample},
we find that  $\mathcal{L}_1(c_i)=0$, for all $i=1,\hdots,s$. This implies that $\cL_1(c)=0$ for every $c\in G(S)$ and any $k$-$\s$-algebra $S$ (Corollary \ref{cor:linsclosedsubgroupsGa}).
%Since the $\s$-ideal $[\mathcal{L}]$ is perfect and $k$ is linearly $\s$-closed,
%every $\s$-polynomial that vanishes on $G(k)$
%must lie in $[\mathcal{L}]$.
Thus $G \subset \sgal(L|K)$.\end{proof}

%%%%%%%%%%%%%%%%%%%%%%%%%%%%%%%%%%%%%%%%%%%%%%%%%%%%%%%%%%%%%%%%%%%%%%%%%%
%%%%%%%%%%%%%%%%%%%%%%%%%%%%%%%%%%%%%%%%%%%%%%%%%%%%%%%%%%%%%%%%%%%%%%%%%%
\section{Discrete integrability}

The terminology of this section is borrowed from differential equations depending on a differential parameter.
Indeed, Y. Sibuya proves the equivalence between the differential analogue of the
definition below, usually called integrability, and the notion of isomonodromy. See Theorem A.5.2.3 in \cite{SibuyaBook}.
\par
Integrability has been studied from a Galoisian point of view in \cite{cassisinger},
\cite{HardouinSinger} and \cite{GorchinskiyOvchinnikov:IsomonodromicDifferentialEquationsAndDifferentialCategories} for equations depending on differential parameters (see also \cite{MitschiSinger:MonodromyGroupsOfParameterizedLinearDifferentialEquationsWithRegularSingularities}
and \cite{Dreyfus:density}).
Here we consider the dependence on a difference parameter. We prove some results that are analogous to
the ones in the papers above. However our result on the descent of integrability is not inspired by any results
in the previous literature and could actually be reproduced in the case of a differential parameter,
improving some applications in the papers cited above.

\subsection{Definition, first properties}

First of all, we introduce the notion of $\s^d$-integrability that we are going to discuss in this section.

\begin{defn}
Let $K$ be a $\ds$-field, $A \in K^{n\times n}$, for some positive integer $n$,
and $d\in\Z_{>0}$. We say that $\de(y)=Ay$ is \emph{$\s^d$-integrable (over $K$)},
if there exists $B\in\Gl_n(K)$, such that
\begin{equation} \label{eq:systemintegrable}
\l\{\begin{array}{l}
\de(y)=Ay\\
\s^d(y)=By
\end{array}\r.
\end{equation}
is compatible, i.e.,
\beq\label{eq:integrabilty}
\de(B)+BA=\hslash_d\s^d(A)B,
\eeq
where $\hslash_d=\hslash\s(\hslash)\cdots\s^{d-1}(\hslash)$.
\end{defn}

The compatibility condition \eqref{eq:integrabilty}
means that the coefficients of the system
``respect'' the commutativity condition $\de\circ\s^d=\hslash_d\s^d\circ\de$.
Indeed, in the notation of the definition above, let $L$ be a $\s$-Picard-Vessiot extension
for the differential system $\de(y)=Ay$. If a system of the form \eqref{eq:systemintegrable} has a
fundamental solution $Y\in\Gl_n(L)$, then
the matrix $B$ verifies  the differential equation \eqref{eq:integrabilty}.
In fact, the commutativity relation $\de\circ\s^d=\hslash_d\s^d\circ\de$ also holds in $L$,
therefore we must have $\de\circ\s^d(Y)=\hslash_d\s^d\circ\de(Y)$, and hence \eqref{eq:integrabilty} is verified.
The inverse implication is also true, if we make some extra assumption on the field $k=K^\de$.

\begin{prop}
Let $K$ be a $\ds$-field, $\de(y)=Ay$ be a linear differential equation with $A\in K^{n\times n}$
and $L$ be a $\s$-Picard-Vessiot extension of $K$ for $\de(y)=Ay$.
We suppose that there exists a matrix $B\in \Gl_n(K)$ verifying \eqref{eq:integrabilty}.
If the field $k=K^\de$ is linearly $\s^d$-closed (see Definition \ref{defi:linearlysclosed})
then there exists a solution of the system
\eqref{eq:systemintegrable} in $\Gl_n(L)$ .
\end{prop}

\begin{proof}
We suppose that \eqref{eq:integrabilty} is verified and that $Y\in\Gl_n(L)$ is a fundamental solution of $\de(y)=Ay$.
Then:
$$
\de\l(\s^d\l(Y\r)\r)=\hslash_d\s^d(A)\s^d(Y)=\l(\de(B)B^{-1}+BAB^{-1}\r)\s^d(Y),
$$
and hence
$$
\de\Big(B^{-1}\s^d(Y)\Big)=A\Big(B^{-1}\s^d(Y)\Big).
$$
We conclude that there exists $C\in\Gl_n(k)$ such that
$B^{-1}\s^d(Y)=YC$. Since $k$ is linearly $\s^d$-closed, there exists $D\in\Gl_n(k)$ such that
$\s^d(D)=C^{-1}D$. The matrix $YD\in\Gl_n(L)$ is a solution of \eqref{eq:systemintegrable}.
\end{proof}

 The situations, where the integer $d$ is strictly greater than one, have to
 be taken into consideration. Indeed, the following example shows that, for any prime integer $p$,
 one can construct a differential system which is $\s^p$-integrable and not $\s^d$-integrable for $d <p$.

 \begin{exa}
 Let $p$ be a prime integer and let $K=\C(x)$ be the $\ds$-field endowed with $\de(x)=1$
 and $\s(x)=x+ 2\pi$. The linear differential system $\de(y)=\begin{pmatrix}  0 & \frac{-1}{p^2} \\
 1 & 0 \end{pmatrix}y$ admits $\begin{pmatrix} \frac{1}{p} \cos (x/p) &  \frac{-1}{p} \sin(x/p)\\
 \sin(x/p) & \cos(x/p) \end{pmatrix}$ as fundamental solution matrix. One can then easily see
 that this differential system is $\s^p$-integrable and not $\s^d$-integrable for $d <p$.
 \end{exa}
Then,  a natural question is  the existence of a bound for such an  integer $d$, in terms
 of the coefficients of the matrix. One could hope to get this bound by taking a closer look at
 the local monodromies. Nonetheless, in the applications to come,  in order to prove that we have no transformal relations between the solutions and thus  no $\s^d$-integrability, we will treat
 the integer $d$  as a parameter.

\subsection{Some examples}\label{subsec:integraexamples}

In this section, we show that $\s^d$-integrability is a situation, which occurs rather frequently
in mathematics as well as in mathematical physics. As we show in the sequel, contiguity relations, Frobenius structure, non linear difference-differential equations,...are intimately related to the notion
of discrete integrability. We also propose some examples of $\s^d$-integrability in higher dimension,
i.e., for integrable systems of linear partial differential equations with one single discrete parameter. Even if we don't develop here this multivariable version of the theory, we believe that
there  should
not be a lot of difficulties to generalize \cite{articleone} to that framework.

\par
Notice that in $p$-adic differential equation theory, the
$\s^d$-integrability with respect to an action of the lift of the Frobenius
from the positive characteristic, is called
\emph{strong Frobenius structure}.
In such context, R. Crew, in \cite{CrewAnnalesENS}, has proved a weak form
of the ``only if'' part of Theorem \ref{thm:integra} below.

\subsubsection{Contiguity relations}

The case of contiguity relations for hypergeometric equations is an example of the situation described in the previous definition.

\begin{exa}
Let $\C(\a,x)$ be a rational function field in two variables, equipped with the derivation
$\de=\frac{d}{dx}$ and the $\C(x)$-linear automorphism $\s$ such that $\s(\a)=\a+1$.
We consider the hypergeometric differential equation
$\de(y)=\frac{\a}{1-x}y$, whose solution is given by $\sum_{n\geq 0}\frac{(\a)_n}{n!}x^n$.
It is easy to see that $\s(y)=\frac{1}{1-x}y$ and $\de(y)=\frac{\a}{1-x}y$ are compatible and that therefore $\de(y)=\frac{\a}{1-x}y$ is $\s$-integrable.
\end{exa}

\begin{exa}
Bessel's equation
$$x^2\de^2(y)+x\de(y)+(x^2-\alpha^2)y=0$$
is a classical example of a differential equation, which is $\s$-integrable with respect to the shift $\s\colon \alpha\mapsto\alpha+1$. The well--known contiguity relations satisfied by the Bessel functions $J_\alpha$ and $Y_\alpha$ can be written in matrix form as $\s(Y)=BY$, where
$$
B=
\begin{pmatrix}
\frac{\a}{x} & -1 \\
\frac{-\a(\a +1)}{x^2} +1 & \frac{\a +1}{x}
\end{pmatrix} \text{ and } Y=\begin{pmatrix} J_\alpha & Y_\alpha \\
J'_\a & Y'_\a\end{pmatrix}
$$
is a fundamental solution matrix. The $\s$-Galois group of Bessel's equation is $\Sl_2^\s$. See Example 2.12 in \cite{articleone}.
\end{exa}

In higher dimension,  contiguity relations exist for non resonant $A$-hypergeometric equations (see \cite[Theorem 6.9.1]{Dworkhypergeo}) as well as difference equations compatible with trigonometric
KZ-differential equations (c.f. \cite{TaraVarchenko}). It would be very interesting to understand these
results in light of a generalization of our  parametrized Galois theory.

\subsubsection{$\s^d$-integrability and Lax pairs for non linear difference-differential equations}

Many non linear difference-differential equations, also called
semi-discrete equations, appear as evolution equations of interesting physical phenomena.
In \cite{AbLad17, AbLad16}, Ablowitz and Ladik introduced mixed Lax pairs, i.e., compatible pair
of a linear differential system and a linear difference system of the following type

$$\frac{\partial}{\partial t} \Psi(n,\lambda,t) =M_n \Psi(n,\lambda,t) \mbox { and }
\Psi(n+1, \lambda,t) =L_n \Psi(n,\lambda,t),$$
where $\lambda$ is the spectral parameter,  $M_n =\begin{pmatrix} A_n(\lambda) & B_n(\lambda)  \\
C_n(\lambda) & D_n(\lambda) \end{pmatrix}$, $L_n =\begin{pmatrix} \lambda I & u_n \\
v_n & \lambda^{-1} I \end{pmatrix}$  and $u_n,v_n, A_n, c_n, D_n$ are matrices
such that the compatibility condition $$\frac{\partial}{\partial t} L_n +L_nM_n-M_{n+1}L_n$$
is satisfied. The initial non linear difference-differential equation becomes now
the compatibility condition of the Ablowitz-Ladik pair. In  \cite[\S 4]{Tsuchpair},
the authors provides explicit  pairs for the Toda lattice, the Belov-Chaltikian lattice, the relativistic
Toda lattice...Of course, the Ablowitz-Ladik pairs are peculiar cases of $\s^d$-integrable systems.
Our parametrized Galois theory could perhaps give another perspective to these questions of integrability : for instance, to find the smallest field of definitions of the pairs, or perhaps to study some periodic
phenomena. For instance, in \cite[equations (6) and (7)]{TamHu}, the
authors write a $\s^2$-integrable differential system.

Connected to these questions of representation of a non-linear
evolution equation in terms of compatible pairs of linear systems, one find the   isomonodromic deformations
and Painlev\'{e} equations. It may be interesting to give an interpretation of the Schlessinger transformations in terms of discrete parametrized Galois theory. Indeed, in \cite{Klardispainlev}, the B\"{a}cklund transformations of a differential Painlev\'{e} equation
give rise to a discrete Painlev\'{e} equation. We think that this process can be seen as an action of a discrete operator
on a partial  differential Lax pair attached to the differential  Painlev\'{e} equation. In that framework,  the discrete
Painlev\'{e} equation should be a defining equation of the $\s$-Galois group of the partial differential Lax pair.

\subsubsection{Strong Frobenius structure of Dwork's exponential}

Let $k$ be a ultrametric field of characteristic zero,
complete with respect to a discrete valuation and let $\F_q$ be its residue field, of characteristic $p>0$ and cardinality $q$.
We denote by $|~|$ the $p$-adic norm of $k$, normalized so that $|p|=p^{-1}$.
The ring $\cE_k^\dag$ of all $f=\sum_{n\in\Z}a_nx^n$, with $a_n\in k$, such that
\begin{itemize}
\item
there exists $\veps>0$, depending on $f$, such that
for any $1<\varrho<1+\veps$, we have $\lim_{n\to \pm\infty}|a_n|\varrho^n=0$;
\item
$\sup_n|a_n|$ is bounded;
\end{itemize}
is actually an henselian field with residue field $\F_q((x))$.
\par
We consider the base field $K = \cE_k^\dag$ endowed with the derivation $\de =x \frac{d}{dx}$ and
an action of a lifting $F$  of the Frobenius automorphism of $\F_q$.
Namely, we consider an endomorphism $F$ of $k$ such that $|F(a)-a^p|<1$, for any $a\in k$, $|a|\leq 1$.
We extend the action of $F$ to $K$ by setting $F(x)=x^p$, so that $F(\sum_{n\in\Z}a_nx^n)=\sum_{n\in\Z}F(a_n)x^{pn}$. We have:
$$\de \circ F = p F \circ \de.$$

\par
In the framework of $p$-adic differential equations (see \cite{Kedlaya:padicdifferentialequations} for an introduction to this topic),
the $F^d$-integrability
of a differential equation is called Frobenius structure  (Definition 17.1.1 in \cite{Kedlaya:padicdifferentialequations}). This notion plays a crucial role in the study of $p$-adic
differential equations since it is the analogue of monodromy in the complex setting. The following examples
illustrate  how the $F$-Galois group distinguishes among differential equations the ones with Frobenius structure,
unlike the ``usual" Galois group.
\par
Let us suppose that there exists $\pi\in k$ such that $\pi^{p-1}=-p$.
It is well known that $\exp(\pi x):=\sum_{n\geq 0}\frac{\pi^n x^n}{n!}$,
solution of $\de(y)=\pi x y$, has radius of convergence $1$,
and hence that it does not belong to $K$. On the other hand, we have $\exp(\pi x)^p\in K$.
This shows that the usual Galois group of $\de(y)=\pi x y$ is the cyclic group of order $p$.

In this particular case, some sharp $p$-adic estimates show that
$\frac{F(\exp(\pi x))}{\exp(\pi x)} \in K$ (see Chap. 2 in \cite{DGS}).
Therefore the $\s$-Galois group is $\mu_p^\s$. Cf. Example 2.14 in \cite{articleone}.

%$F(x)x^{-1}-1$ belongs to the vanishing ideal of
%$F-\Gal(L|K)$ inside $k\{x,\frac{1}{x}\}_\s$, i.e, $F-\Gal(L|K)$ is  the
%constant group inside $\mathbf{G}_m$.

%\par
%If we suppose that there exists $\a\in k$ such that $p^{-1}|\pi|<|\a|<|\pi|$, then $\exp(\a x)$ is solution of
%$\de(y)=\a xy$ and verifies $\exp(\a x)^p\in K$.   Let $L=K\left\<exp(\a x\right\>_F$ be the $\delta F$-field generated
%by $exp(\a x)$ inside $K((x))$. Since $K((x))^\de=k$, the extension $L|K$ is a $F$-Picard-Vessiot extension for
%$\de(y)=\a x y$. In this case too, we deduce that
%there exist $M\geq 0$ and $d\geq 1$ such that
%$\de(y)=p^M F^M(\a)x^{p^M}$ is $F^d$-integrable.
%In this case we know that $M\geq 1$. The polynomial $F^{M+d}(x)F^M(x)^{-1}-1$ belongs
%to  the vanishing ideal of $F-\Gal(L|K)$ inside $k\{x,\frac{1}{x}\}_\s$.
%Thus the $F$-Galois group distinguishes the equations $\de(y)=\pi x y$ and $\de(y)=\a x y$,
%unlikely the usual Galois group.

\subsection{Order 1 differential equations with cyclic Galois group of prime order}

In the appendix, we prove that any proper $\s$-closed subgroup of $\mu_p$, the algebraic group of $p$-th roots of unity, contains the $\s$-polynomial $\s^m(x)-\s^{m+d}(x) \in k\{x,x^{-1}\}_\s=k\{\Gm\}$ in its vanishing ideal inside $\Gm$, for convenient integers $d\geq 1$ and $m\geq 0$. (See Proposition \ref{prop:mup}). We are going to use this fact
in the proof of the following proposition.

\begin{prop}
Let $K$ be a $\ds$-field and let $L|K$ be
a $\s$-Picard-Vessiot extension for a differential equation $\de(y)=ay$ with $a \in K^{\times}$.
We assume that the Zariski closure of $\sgal(L|K)$, i.e.,
the ``usual" Galois group of $\de(y)=ay$, is $\mu_p$, for some prime number $p$.
If $\sgal(L|K)$ is a proper subgroup of $\mu_p$,
 there exist two integers $M\geq 0$ and $d\geq 1$ such that the
system $\de(y)=\hslash_m\s^m(a)y$ is $\s^d$-integrable for all $m\geq M$.
\end{prop}

\begin{proof}
Let $z\in L^\times$ be a non-zero
solution of $\de(y)=ay$ and $R=K\{z,z^{-1}\}_\s$ be the $\s$-Picard-Vessiot ring of $L$.
We identify $\sgal(L|K)$ to a $\s$-closed subgroup of $\mathbf G_m$ via $z$.
It follows from Proposition \ref{prop:mup} that
there exist integers
$d\geq 1$ and $M\geq 0$ such that $\s^{M+d}(x)-\s^M(x)$ belongs to $\I(\sgal(L|K))$. Therefore $\s^{m+d}(x)-\s^m(x)\in\I(\sgal(L|K))$ for any $m\geq M$. Now, for all $k$-$\s$-algebras $S$ and for all $\tau \in \sgal(L|K)(S)$, in $R\otimes_k S$ we have
$\tau(z\otimes 1)=z\otimes c$, for some $c\in S$, with $\s^{m+d}(c)=\s^m(c)$.
Therefore
$$
\tau(\s^{m+d}(z)\s^m(z)^{-1}\otimes 1) = \s^{m+d}(z)\s^m(z)^{-1}\otimes \s^{m+d}(c)\s^m(c)^{-1} =\s^{m+d}(z)\s^m(z)^{-1}\otimes 1.
$$
This means that $\s^{m+d}(z)\s^m(z)^{-1} \in K$
by Galois correspondence. Since $\s^m(z)$ is a
solution of  the equation
$\de(y)=\hslash_m\s^m(a)y$, we find that this system is $\s^d$-integrable for all $m\geq M$.
\end{proof}

\begin{exa}
In the notation introduced for the Dwork exponential, we can consider the differential equation $\de(y)=\frac{y}{p}$.
Its usual differential Galois group is $\mu_p$, since the cyclic extension
$K(x^{1/p})$ is generated by a solution of $\de(y)=\frac{y}{p}$.
We have $F(x^{1/p})=x$, which obviously is $F$-integrable.
\end{exa}

%%%%%%%%%%%%%%%%%%%%%%%%%%%%%%%%%%%%%%%%%%%%%%%%%%%%%%%%%%%%%%%%%%%%%%%%%%
\subsection{A descent result for $\s^d$-integrability}
%%%%%%%%%%%%%%%%%%%%%%%%%%%%%%%%%%%%%%%%%%%%%%%%%%%%%%%%%%%%%%%%%%%%%%%%%%

Our first concern is to show that ``being $\s^d$-integrable'' descents along a $\s$-separable $\s$-field extension of the $\de$-constants
(see \S\ref{subsec:differencealgebra} for the definitions).

\medskip
Let $K$ be a $\ds$-field. We would like to extend the $\s$-field of $\de$-constants $k:=K^\de$ to a $\s$-field extension $\widetilde{k}$ of $k$.
The ring $K\otimes_k\widetilde{k}$ is naturally a $\ds$-ring (where we consider $\widetilde k$ as $\de$-constant).
Because $k$ is relatively algebraically closed in $K$ (Corollary to Proposition 4, Chapter II, Section 4, p. 94 in \cite{Kolchin:differentialalgebraandalgebraicgroups}), we know that
$K\otimes_k\widetilde{k}$ is an integral domain.
Let
$$
\widetilde{K}:=\quot(K\otimes_k\widetilde{k})
$$
denote the quotient field of $K\otimes_k\widetilde{k}$.
To ensure that $\s$ is injective on $K\otimes_k\widetilde{k}$,
and hence that it extends to $\wtilde K$,
we have to assume that $\wtilde{k}$ is $\s$-separable over $k$. Note that by Corollary A.14 in \cite{articleone},
this is automatic if $k$ is inversive. Then $\widetilde{K}$ is naturally a $\ds$-field. Because
$K\otimes_k\widetilde{k}$ is $\de$-simple by Lemma \ref{lemma:simple}, it follows that $\widetilde{K}^\de=\widetilde{k}$.

\begin{prop} \label{prop:descentintegrability}
Let $K$ be a $\ds$-field and $A\in K^{n\times n}$. Let $\widetilde{k}$ be a $\s$-separable $\s$-field extension of $k:=K^\de$ and
$\widetilde{K}=\quot(K\otimes_k\widetilde{k})$.
Then the equation $\de(y)=Ay$ is $\s^d$-integrable over $K$ if and only if $\de(y)=Ay$ is $\s^d$-integrable over $\widetilde{K}$.
\end{prop}

To prove the proposition above we need a lemma on the behavior of the
vector space of solutions with respect to an extension of the constant field:

\begin{lemma} \label{lemma:descent solution to linear delta equation}
Let $K$ be a $\de$-field and $A\in K^{n\times n}$. Let $\wtilde k$ be a ($\de$-constant) field extension of $k:=K^\de$ and let
$\widetilde{K}$ denote the quotient field of $K\otimes_k\widetilde{k}$.
Let us denote by $\wtilde{ \mathbb V}$ (resp. $\mathbb V$) the solution space of $\de(y)=Ay$ in $\wtilde{K}^n$ (resp.
$K^n$). Then $\wtilde{\mathbb{V}}\simeq\mathbb{V} \otimes_k \wtilde{k}$.
\end{lemma}

\begin{proof}
It is clear that the canonical map $\mathbb{V} \otimes_k \wtilde{k}\to \wtilde{\mathbb{V}}$ is injective.
Indeed, if elements from $K^n\subset(K\otimes_k\wtilde{k})^n=K^n\otimes_k\widetilde{k}$ are linearly independent over $k$, they are also linearly independent over $\wtilde{k}$.
\par
It remains to see that $\mathbb{V} \otimes_k \wtilde{k}\to \wtilde{\mathbb{V}}$ is surjective. Let $z\in\mathbb{\wtilde{V}}\subset\wtilde{K}^n$.
We know from Lemma \ref{lemma:simple}
that $S:=K\otimes_k \widetilde k$ is a $\de$-simple $\de$-ring.
We claim that $z\in S^n$.
Set $\ida:=\{r\in S|\ rz\in S^n\}\subset S$. Then $\ida$ is a non-zero ideal
of $S$ and we will now show that $\ida$ is a $\de$-ideal of $S$: For $r\in\ida$ we have
$\de(rz)=r\de(z)+\de(r)z=rAz+\de(r)z$. As $\de(rz)$ and $rAz=Arz$ lie in $S^n$, it follows that $\de(r)z$ lies in $S^n$. By the simplicity of $S$, we see that $1\in\ida$, i.e., $z\in S^n$.
\par
Now, let $(\lambda_i)$ be a $k$-basis of $\widetilde{k}$ (which can be infinite, of course). We can write $z=\sum z_i\otimes \lambda_i$ for some $z_i\in K^n$.
Then $\de(z)=\sum\de(z_i)\otimes\lambda_i$. On the other hand, $\de(z)=Az=\sum Az_i\otimes\lambda_i$. Comparing the
coefficients yields $\de(z_i)=Az_i$ for all $i$. This shows that $z$ lies in the image $\mathbb{V} \otimes_k \wtilde{k}\to \wtilde{\mathbb{V}}$.
\end{proof}

\begin{proof}[Proof of Proposition \ref{prop:descentintegrability}]
One implication is tautological, so let us prove the non-trivial one. Assume that
$\de(y)=Ay$ is $\s^d$-integrable over $\widetilde{K}$.
This means that there exists $\wtilde B\in\Gl_n(\wtilde K)$ such that
$\de(\wtilde B)+\wtilde BA=\hslash_d\s^d(A)\wtilde B$. The latter equation, when considered as an equation in $\widetilde{B}$, is a
linear differential system over $K$ of order $n^2$.
Let $\mathbb{V}\subset K^{n\times n}$ (resp. $\wtilde{ \mathbb V}\subset \wtilde{K}^{n\times n}$)
denote the corresponding solution vector space over $k$ (resp. over $\wtilde k$).
We know from Lemma \ref{lemma:descent solution to linear delta equation} that
$\wtilde{ \mathbb V}=\mathbb{V} \otimes_k \wtilde{k}$.
We have to find a $B\in\mathbb{V}$ with non-zero determinant. Let $v_1,\ldots,v_m\in \wtilde{K}^{n\times n}$ be a $k$-basis of $\mathbb{V}$.
Then $v_1,\ldots,v_m$ is also a $\widetilde{k}$-basis of $\wtilde{ \mathbb V}$. It follows from
the Wronskian lemma ( Lemma 1.7, p.7 in \cite{vdPutSingerDifferential})
that $v_1,\ldots,v_m$ are also $\widetilde{K}$-linearly independent. With respect to the basis $v_1,\ldots,v_m$ of $\mathbb{V}\otimes_k\widetilde{K}\subset\widetilde{K}^{n\times n}$ the determinant $\det\colon\mathbb{V}\otimes_k\widetilde{K}\to \widetilde{K}$ is given by a polynomial (in $m$-variables) with coefficients in $\widetilde{K}$.
Because the determinant does not vanish on all of $\widetilde{\mathbb{V}}$, this shows that the determinant can not vanish on all of $\mathbb{V}$. So $\de(y)=Ay$ is $\s^d$-integrable over $K$.
\end{proof}

\subsection{Galoisian characterization of $\s^d$-integrability}

Let $k$ be a $\s$-field and let $d\geq 1$. For every $k$-$\s$-algebra $S$, we set
$$
\Gl_{n,k}^{\s^d}(S)=\{g\in\Gl_n(S)|\ \s^d(g)=g\}.
$$
Then $\Gl_{n,k}^{\s^d}$ is a $\s$-closed subgroup of $\Gl_{n,k}$.
If $G\leq \Gl_{n,k}$ is a $\s$-closed subgroup and $h\in\Gl_n(k)$, then the $\s$-closed subgroup $hGh^{-1}\leq \Gl_{n,k}$ obtained from $G$ by conjugation with $h$, is given by
$$
(hGh^{-1})(S)=\{hgh^{-1}|\ g\in G(S)\}\leq\Gl_n(S),
$$
for every $k$-$\s$-algebra $S$.

\begin{defn}\label{defn:sconstantgroup}
Let $G$ be a $\s$-closed subgroup of $\Gl_{n,k}$. We call $G$ a \emph{$\s^d$-constant} subgroup of $\Gl_{n,k}$ if $G$ is contained in $\Gl_{n,k}^{\s^d}$. If $\wtilde{k}$ is a $\s$-field extension of $k$, we say that \emph{$G$ is conjugated over $\wtilde{k}$ to a $\s^d$-constant group} if there exists $h \in \Gl_n(\wtilde{k})$ such that $hG_{\wtilde{k}}h^{-1}\leq \Gl_{n,\wtilde{k}}$ is $\s^d$-constant.
\end{defn}

We prove a result on $\s^d$-integrability, which is analogous to Proposition 2.9 in \cite{HardouinSinger},
and \S1.2.1 in  \cite{diviziohardouindependencies}. The statement below is more general than the cited
results, because it relies on Proposition \ref{prop:descentintegrability}.

\begin{thm}\label{thm:integra}
Let $L|K$ be a $\s$-Picard-Vessiot extension for $\de(y)=A y$, with $A\in K^{n\times n}$. Then $\de(y)=Ay$ is $\s^d$-integrable over $K$
if and only if there exists a $\s$-separable $\s$-field extension $\widetilde{k}$ of $k:=K^\de$, such that
the $\s$-Galois group $\sgal(L|K)$ is conjugated over $\wtilde{k}$ to a $\s^d$-constant subgroup
of $\Gl_{n,\wtilde{k}}$.
\end{thm}

First of all, we are going to prove one implication of Theorem \ref{thm:integra}
(see Lemma \ref{lemma:sintegrableimpliesconstant} below). To this purpose we need a few lemmas.

\begin{lemma} \label{lemma:baseextensionforsPV}
Let $L|K$ be a $\s$-Picard-Vessiot extension and $\wtilde{k}$ a $\s$-separable $\s$-field extension of $k=K^\de$.
Then $\widetilde{L}=\quot(L\otimes_k\widetilde{k})$ is a $\s$-Picard-Vessiot extension of $\widetilde{K}=\quot(K\otimes_k\widetilde{k})$ and the $\s$-Galois group $\wtilde{G}$ of $\wtilde{L}|\wtilde{K}$ is obtained from the $\s$-Galois group $G$ of $L|K$ by base extension, i.e., $\widetilde{G}=G_{\wtilde{k}}$.
\end{lemma}

\begin{proof}
As $\widetilde{L}^\de=\wtilde{k}=\wtilde{K}^\de$, it is clear that $\wtilde{L}|\wtilde{K}$ is a $\s$-Picard-Vessiot extension.
Let $R\subset L$, respectively $\wtilde{R}\subset\wtilde{L}$ denote the $\s$-Picard-Vessiot rings. Then $\widetilde{R}$ is obtained from $R\otimes_k\widetilde{k}$ by localizing at the multiplicatively closed set of all non-zero divisors of $K\otimes_k\wtilde{k}$.
It follows that, for every $\widetilde{k}$-$\s$-algebra $S$,
$$
G_{\wtilde{k}}(S)=\Aut^{\ds}(R\otimes_kS|K\otimes_kS)
=\Aut^{\ds}((R\otimes_k\wtilde{k})\otimes_{\widetilde{k}}S|(K\otimes_k\widetilde{k})\otimes_{\wtilde{k}}S)=
\Aut^{\ds}(\widetilde{R}\otimes_{\widetilde{k}}S|\widetilde{K}\otimes_{\wtilde{k}}S)=\widetilde{G}(S).
$$
This ends the proof.
\end{proof}

\begin{lemma} \label{lemma:solvelinearinssep}
Let $k$ be a $\s$-field, $d\geq 1$ and $B\in\Gl_n(k)$. There exists a $\s$-separable $\s$-field extension $\wtilde{k}$ of $k$ such that $\s^d(y)=By$ has a fundamental solution matrix in $\wtilde{k}$.
\end{lemma}

\begin{proof}
Let $X_1,\ldots,X_d$ denote $n\times n$ matrices of indeterminates over $k$. Extend $\s$ from $k$ to $S:=k[X_1,\ldots,X_d]$ by
$\s(X_1)=X_2,\ldots,\s(X_{d-1})=X_d,\s(X_d)=BX_1$. Obviously $S$ is $\s$-separable over $k$. It follows from
Lemma A.16 in \cite{articleone}
that the quotient field $\wtilde{k}$ of $S$ is also $\s$-separable over $k$. Clearly $\s^d(X_1)=BX_1$.
\end{proof}

\begin{lemma} \label{lemma:linearlysclosedstableunderpower}
Let $k$ be a linearly $\s$-closed $\s$-field. Then $k$ is linearly $\s^d$-closed for every $d\geq 1$.
\end{lemma}
\begin{proof}
Assume that $d\geq 1$ and $B\in\Gl_n(k)$. We have to find a $Y\in\Gl_n(k)$ with $\s^d(Y)=BY$.
Let $y_1,\ldots,y_d$ denote vectors of size $n$. The linear system of order $nd$ given by
$$\s(y_1)=y_2,\ldots,\s(y_{d-1})=y_d,\s(y_d)=By_1$$ has a fundamental solution matrix $Z\in\Gl_{nd}(k)$. Since the first $n$ rows of $Z$ are linearly independent, there exist $n$ columns of $Z$, such that the corresponding vectors of the first $n$ entries are linearly independent. These form a fundamental solution matrix for $\s^d(y)=By$.
\end{proof}

So we are ready to prove the first part of Theorem \ref{thm:integra}:

\begin{lemma} \label{lemma:sintegrableimpliesconstant}
Let $L|K$ be a $\s$-Picard-Vessiot extension for $\de(y)=A y$, with $A\in K^{n\times n}$. If $\de(y)=A y$ is $\s^d$-integrable over $K$, then, for a suitable $\s$-separable $\s$-field extension $\widetilde{k}$ of $k=K^\de$, the $\s$-Galois group $\sgal(L|K)$ is conjugated over $\wtilde{k}$ to a $\s^d$-constant subgroup of $\Gl_{n,\wtilde{k}}$. If $k$ is linearly $\s$-closed, we can take $\wtilde{k}$ equal to $k$.
\end{lemma}

\begin{proof}
Let $B\in\Gl_n(K)$ be such that the system
\beq\label{eq:integrablesys}
\l\{\begin{array}{l}
\de(y)=Ay\\
\s^d(y)=By
\end{array}\r.
\eeq
is integrable. Let us also fix a fundamental solution matrix $Y\in\Gl_n(L)$ for $\de(y)=Ay$.
The integrability condition \eqref{eq:integrabilty} implies that:
\begin{align*}
\de\l((BY)^{-1}\s^d(Y)\r)&=\de((BY)^{-1})\s^d(Y)+(BY)^{-1}\de(\s^d(Y))=\\
&=-(BY)^{-1}\de(BY)(BY)^{-1}\s^d(Y)+(BY)^{-1}\hslash_d\s^d(\de(Y))=\\
 &=-(BY)^{-1}\left((\de(B)Y+BAY)(BY)^{-1}\s^d(Y)-\hslash_d\s^d(A)\s^d(Y)\right)=\\
&=-(BY)^{-1}\left((\de(B)+BA)B^{-1}-\hslash_d\s^d(A)\right)\s^d(Y)=0.
\end{align*}
Thus, there exists
$D\in\Gl_n(k)$ such that
\begin{equation} \label{eq:followsfromintegrability}
\s^d(Y)=BYD.
\end{equation}
By Lemma \ref{lemma:solvelinearinssep} there exists a $\s$-separable $\s$-field extension $\wtilde{k}$ of $k$ and a matrix $U\in\Gl_n(\wtilde{k})$ with $\s^d(U)=D^{-1}U$.

Set $\widetilde{K}=\quot(K\otimes_k\widetilde{k})$ and $\widetilde{L}=\quot(L\otimes_k\widetilde{k})$. We will now work in the $\s$-Picard-Vessiot extension $\wtilde{L}|\wtilde{K}$ (Lemma \ref{lemma:baseextensionforsPV}).
The matrix $V:=YU\in\Gl_n(\wtilde{L})$
is solution of \eqref{eq:integrablesys}, in fact:
$$
\de(V)=\de(Y)U=AV
$$
and
$$
\s^d(V)=\s^d(Y)D^{-1}U=BYU=BV.
$$
For $\tau\in\sgal(\wtilde L|\wtilde K)$ and $S$ a $\wtilde k$-$\s$-algebra, let $[\tau]_V\in\Gl_n(S)$ be such that $\tau(V\otimes 1)=(V\otimes 1)[\tau]_V$.
We have
$$
\tau(\s^d(V\otimes 1))=\tau(BV\otimes 1)=(BV\otimes 1)[\tau]_V
$$
and
$$
\tau(\s^d(V\otimes 1))=\s^d(\tau(V\otimes 1))=\s^d((V\otimes 1)[\tau]_V) =(BV\otimes 1)\s^d([\tau]_V).$$
We deduce that $[\tau]_V=\s^d([\tau]_V)$. Therefore, if we consider $\wtilde{G}:=\sgal(\wtilde L|\wtilde K)$ as a $\s$-closed subgroup of $\Gl_{n,\wtilde k}$ via the embedding associated with the choice of the fundamental solution matrix $V$, then $\wtilde{G}$ is a $\s^d$-constant subgroup of $\Gl_{n,\wtilde k}$. We consider the $\s$-Galois group $G:=\sgal(L|K)$ as a $\s$-closed subgroup of $\Gl_{n,k}$ via the embedding associated with the choice of the fundamental solution matrix $Y$. Since $V=YU$, it is clear from Lemma \ref{lemma:baseextensionforsPV} that $G_{\wtilde k}$ is conjugated to $\wtilde G$ via $U\in\Gl_n(\wtilde{k})$.
\end{proof}

\begin{rmk}
If $k$ is linearly $\s$-closed, we can choose $U\in\Gl_n(k)$ and $\wtilde{k}=k$ by Lemma \ref{lemma:linearlysclosedstableunderpower}.
\end{rmk}

The following lemma is a crucial step in the proof of the inverse implication in Theorem \ref{thm:integra}.

\begin{lemma} \label{lemma:constantimpliesintegrable}
Let $L|K$ be a $\s$-Picard-Vessiot extension for $\de(y)=Ay$ where $A \in K^{n\times n}$.
If the $\s$-Galois group $\sgal(L|K)$ is conjugated (over $k=K^\de$) to a $\s^d$-constant subgroup of $\Gl_{n,k}$, then $\de(y)=A y$ is $\s^d$-integrable over $K$.
\end{lemma}

\begin{proof}
Since conjugation in $\Gl_n(k)$ corresponds to a change of fundamental solution matrix, we can find a fundamental solution matrix $Y\in\Gl_n(L)$ such that $G:=\sgal(L|K)$ is a $\s^d$-constant subgroup of $\Gl_{n,k}$,
 with respect to the embedding $\tau\mapsto [\tau]_Y$ determined by $Y$.
(As before, for a $k$-$\s$-algebra $S$ and $\tau\in\sgal(L|K)(S)$ we denote by $[\tau]_Y\in\Gl_n(S)$ the matrix such that $\tau(Y\otimes 1)=Y\otimes 1[\tau]_Y$.)
We set $B=\s^d(Y)Y^{-1}\in\Gl_n(L)$.
The matrix $B$ is an invariant of $\sgal(L|K)$, in fact for every $k$-$\s$-algebra $S$ and every $\tau\in\sgal(L|K)(S)$,
we have:
\begin{align*}
\tau(B\otimes 1) & =\tau\l(\s^d(Y\otimes 1)(Y\otimes 1)^{-1}\r)
=\s^d(\tau(Y\otimes 1))\tau\l(Y^{-1}\otimes 1\r)
=\s^d\l(Y\otimes 1 [\tau]_Y \r)[\tau]_Y^{-1}Y^{-1}\otimes 1 \\
& =\s^d(Y)Y^{-1}\otimes 1=B\otimes 1,
\end{align*}
since $\s^d([\tau]_Y)=[\tau]_Y$. It follows from the $\s$-Galois correspondence (Theorem \ref{theo:Galoiscorrespondence}) that $B\in\Gl_n(K)$.
Since $Y\in\Gl_n(L)$ satisfies $\de(Y)=AY$ and $\s^d(Y)=BY$ it follows easily from $\de(\s^d(Y))=\hslash_d\s^d(\de(Y))$ that $\de(y)=Ay$ is $\s^d$-integrable.
\end{proof}

Finally, we are able to prove Theorem \ref{thm:integra}.

\begin{proof}[Proof of Theorem \ref{thm:integra}]
Let $\wtilde{k}$ be a $\s$-separable $\s$-field extension of $k$ such that $\sgal(L|K)$ is conjugated over $\wtilde{k}$ to a $\s^d$-constant group. Consider the $\s$-Picard-Vessiot extension $\wtilde{L}|\wtilde{K}$ obtained from $L|K$ by extending the $\de$-constants from $k$ to $\wtilde{k}$ (Lemma \ref{lemma:baseextensionforsPV}).
Since $\sgal(\wtilde{L}|\wtilde{K})$ is conjugated over $\wtilde{k}=\wtilde{K}^\de$ to a $\s^d$-constant group, Lemma \ref{lemma:constantimpliesintegrable} implies that $\de(y)=Ay$ is $\s^d$-integrable over $\wtilde{K}$. It follows from Proposition \ref{prop:descentintegrability} that $\de(y)=Ay$ is $\s^d$-integrable over $K$.
The other implication is immediate from Lemma \ref{lemma:sintegrableimpliesconstant}.
\end{proof}

\begin{cor}
Let $L|K$ be a $\s$-Picard-Vessiot extension for $\de(y)=A y$, with $A\in K^{n\times n}$. Assume that $k:=K^\de$ is linearly $\s$-closed. Then $\de(y)=Ay$ is $\s^d$-integrable (over $K$)
if and only if the $\s$-Galois group $\sgal(L|K)$ is conjugated (over $k$) to a $\s^d$-constant subgroup
of $\Gl_{n,k}$.
\end{cor}
\begin{proof}
This is clear from Lemmas \ref{lemma:sintegrableimpliesconstant} and \ref{lemma:constantimpliesintegrable}.
\end{proof}

\begin{rmk}
It may appear quite difficult to determine a priori, given a differential system $\de(y)=Ay$,  if this system is $\s^d$-integrable
for some $d\geq 1$. However, we will show in the sequel that the algebraic structure
of the ``usual" Galois group of the differential system may help to answer this type of questions and to determine
for which $d$ this $\s^d$-integrability occurs.
\end{rmk}

%%%%%%%%%%%%%%%%%%%%%%%%%%%%%%%%%%%%%%%%%%%%%%%%%%%%%%%%%%%%%%%%%%%%%%%%%%%%%%%%%%%%%%%%%%%%%%%%
%%%%%%%%%%%%%%%%%%%%%%%%%%%%%%%%%%%%%%%%%%%%%%%%%%%%%%%%%%%%%%%%%%%%%%%%%%%%%%%%%%%%%%%%%%%%%%%%
%\input{charlotte-simplegroups}
%%%%%%%%%%%%%%%%%%%%%%%%%%%%%%%%%%%%%%%%%%%%%%%%%%%%%%%%%%%%%%%%%%%%%%%%%%%%%%%%%%%%%%%%%%%%%%%%%%
\section{Discrete integrability in the case of almost simple Galois groups}
%%%%%%%%%%%%%%%%%%%%%%%%%%%%%%%%%%%%%%%%%%%%%%%%%%%%%%%%%%%%%%%%%%%%%%%%%%%%%%%%%%%%%%%%%%%%%%%%%%

We begin this section with the study of the discrete integrability of differential system
having a simple classical Galois group. This is equivalent to ask that the $\s$-Galois group is
a Zariski dense $\s$-closed subgroup of a simple algebraic group. This
situation is very rigid, as proved in Theorem \ref{thm:classsimple}, and the $\s$-Galois group
has to be either equal to the whole simple algebraic group or conjugated to a $\s^d$-constant group for some
positive integer $d$. Combined with our $\s^d$-integrability criteria, we find that
either, we have a very  strong transformal relation between the solutions of our differential system, namely the $\s^d$-integrability,
or we have no transformal relations at all. In Proposition \ref{prop:simpleintegra} below we precise
this discussion.

Recall that a linear algebraic group $H$ over a field $k$ (i.e., an affine group scheme of finite type over $k$) is called \emph{simple} if it is non-commutative, connected and every normal closed subgroup is trivial. If $H$ is non-commutative, connected and every normal closed connected subgroup is trivial, then $H$ is called \emph{almost simple}. We say that $H$ is \emph{absolutely (almost) simple} if the base extension of $H$ to the algebraic closure of $k$ is (almost) simple.

\begin{prop}\label{prop:simpleintegra}
Let $K$ be an \emph{inversive} $\ds$-field, $A\in K^{n\times n}$ and $L|K$ a $\s$-Picard-Vessiot extension for $\de(y)=Ay$.
We assume that the Zariski closure $H$ of $\sgal(L|K)$ inside $\Gl_{n,k}$ is an absolutely simple algebraic group of dimension $t \geq 1$
over $k=K^\de$.
Then the following statements are equivalent:
\begin{enumerate}

\item
$\sgal(L|K)$
is a proper  $\s$-closed subgroup of $H$.

\item
The $\s$-transcendence degree of  $L|K$
is strictly  less than $t$.

\item
There exists $d \in \Z_{>0}$ such that the system $\de(y)=Ay$ is $\s^d$-integrable.
\end{enumerate}
\end{prop}

\begin{proof}
Let $Y\in\Gl_n(L)$ be a fundamental solution matrix for $\de(y)=Ay$ and $G=\sgal(L|K)$.
Assumption (iii) implies the existence of matrices $B\in\Gl_n(K)$ and $D\in\Gl_n(k)$ such that $\s^d(Y)=BYD$. (See equation (\ref{eq:followsfromintegrability}).) Therefore $\strdeg(L|K)=0<t$, and we see that (iii) implies (ii).

We know by Proposition \ref{prop:dimensiondegtrans} that
$\strdeg(L|K) = \sdim_k(G)$.
Therefore the assumption (ii) means that we have $\sdim_k(G)=\strdeg(L|K) <t$ and therefore $G$ is a proper subgroup of $H$.
\par
It remains to prove that (i) implies (iii). We first reduce to the case that $k=K^\de$ is algebraically closed. Since $K$ is inversive also $k=K^\de$ is inversive. Let $\wtilde{k}$ denote an algebraic closure of $k$ and extend $\s$ from $k$ to $\wtilde{k}$. Then $\wtilde{k}$ is automatically inversive. Thus also $\wtilde{K}:=\quot(K\otimes_k\wtilde{k})$ is inversive. We know from Lemma \ref{lemma:baseextensionforsPV} that $\wtilde{L}=\quot(K\otimes_k\wtilde{k})$ is a $\s$-Picard-Vessiot extension whose $\s$-Galois group $\wtilde{G}$ equals $G_{\wtilde{k}}$. It is easy to see that the Zariski closure of $\wtilde{G}$ in $\Gl_{n,\wtilde{k}}$ equals $H_{\wtilde{k}}$. Since $G$ is properly contained in $H$, also $\wtilde{G}$ is properly contained in $H_{\wtilde{k}}$. Thus, by assumption, it follows that $\de(y)=Ay$ is $\s^d$-integrable over $\wtilde{K}$ for some convenient integer $d\geq 1$. But then Proposition \ref{prop:descentintegrability} implies that $\de(y)=Ay$ is $\s^d$-integrable over $K$.

So from now we can assume that $k$ is algebraically closed. Since $K$ is inversive, Corollary 4.4 in \cite{articleone} asserts that
$G=\sgal(L|K)$ is $\s$-reduced. Therefore, Theorem \ref{thm:classsimple} says that there exists a $\s$-separable $\s$-field extension $\wtilde k$ of $k$ and an integer $d \geq 1$ such that $G$ is conjugated in $\Gl_{n,\wtilde{k}}$ to a $\s^d$-constant group. Theorem
\ref{thm:integra} ends the proof.
\end{proof}

The previous proposition allows us to determine very easily if the solutions of a differential system $\de(y)=Ay$, with  simple usual Galois group, are transformally independent or not. Indeed, if we prove that, for all $d \geq 0$, the differential system
$\de(y) +yA =\hslash_d\s^d(A)y$ has no rational solution $Y \in \Gl_n(K)$, there will be no (proper) transformal relations between the solutions of $\de(y)=Ay$. Thus, we have reduced the question of the algebraic independence between the solutions and their successive transforms via $\sigma$, which is a non-linear problem,
into a question of existence of rational solutions to a given \emph{linear} differential system.
This last question possesses many algorithmic answers and we refer to \S 4.1 of \cite{vdPutSingerDifferential}  for a detailed exposition of this subject.
\par
However, in many situations of interest, the usual Galois group of a differential system
is not simple but almost simple. See for example \cite{BeukersHeckman:MonodromyForTheHypergeometricFunction}. Unfortunately, as we show in the appendix, Proposition \ref{prop:simpleintegra} fails, a priori, if the
Zariski closure of $\sgal(L|K)$ is only  almost-simple and not simple. This is mainly due to the fact that, contrary to the difference varieties studied in ACFA, the group $k$-$\s$-schemes associated to   finite algebraic groups are not necessarily $\s$-constants.
\par
We  propose now two ways of getting rid of this problem:
We first explore a naive less effective strategy to deal with the group $\Sl_{n,k}$ and then prove a more general statement
for almost simple groups (see Theorem \ref{thm:almostintegra}), based on a generalization of a result in \cite{ChatHrusPet} (see Theorem \ref{thm:classalmostsimple}).
To conclude this section, we apply Theorem \ref{thm:almostintegra} to second order differential equations, especially to the Airy equation.
\par
We point out that Theorem \ref{thm:classalmostsimple} on the classification of $\s$-closed subgroups of almost simple groups should also
give applications of the theory developed in \cite{OvchinnikovWibmer:SGaloisTheoryOfLinearDifferenceEquations}.

\subsection{A naive approach to $\Sl_{n,k}$ through symmetric powers}

Corollary \ref{cor:slnsympower} below yields to a characterization of the $\s$-algebraic relations
satisfied by  the solutions in terms of $\s^d$-integrability of  the $n$-th symmetric power of the
initial differential system (see for instance Theorem 3.4 in \cite{SingUlm} for definition).
Once again, this  method allows us to translate the question of the existence of transformal relations among
solutions of our initial differential system in terms of the much easier problem of  existence
of rational solutions to an auxiliary differential system.  However, if the order of the initial
system is $n$,  the order of the auxiliary differential system  will be  smaller or equal to
$\binom{2n-1}{n}$. Computations may be then hard to manage.

\begin{cor}\label{cor:slnsympower}
Let $K$ be an \emph{inversive} $\ds$-field such that $k=K^\de$ is algebraically closed,
$\mathcal{L}(y) =a_0y+a_1\de(y)+\dots+\de^n(y)$ a linear differential equation with coefficients in $K$
and $\mathcal{L}^{\circledS n}$ the $n$-th symmetric power of $\mathcal{L}$.
We denote by
$$
A_{\mathcal{L}}=
\begin{pmatrix}
0&1&\hdots & 0 \\
\vdots& \ddots& \ddots&\vdots \\
0&\hdots& 0&1 \\
-a_0&-a_1& \hdots& -a_{n-1}
\end{pmatrix} \hbox{~~~(resp. $A_{\mathcal{L}^{\circledS n}}$)}
$$
the companion matrix of $\mathcal{L}$ (resp. of $\mathcal{L}^{\circledS n}$)
and by $L|K$ a $\s$-Picard-Vessiot extension for $\de(y)=A_{\mathcal{L}}y$.
We assume that the Zariski closure
of $\sgal(L|K)$ inside $\Gl_{n,k}$ is $\Sl_{n,k}$.
Then the following statements are equivalent:
\begin{enumerate}

\item
$\sgal(L|K)$
is a proper  $\s$-closed subgroup of $\Sl_{n,k}$.

\item
The $\s$-transcendence degree of  $L|K$
is strictly  less than $n^2-1$.

\item
There exists $d \in \Z_{>0}$ such that the system $\de(y)=A_{\mathcal{L}^{\circledS n}}y$ is $\s^d$-integrable.
\end{enumerate}
\end{cor}

\begin{proof}
Let $y_1,\dots,y_n \in L$ be $n$ linearly independent solutions of $\mathcal{L}$ in $L$.
By definition (see Theorem 3.4 in \cite{SingUlm}), the symmetric power of order $n$ of
$\mathcal{L}$ is a linear differential equation with coefficients in $K$,
whose solution space is spanned by the products of the form $y_{i_1}y_{i_2}\dots y_{i_n}$, where
$(i_1,\dots,i_n)$ runs in $\{1,\dots,n\}^n$. Let $M$ be the $\ds$-field extension of $K$
generated by the $y_{i_1}y_{i_2}\dots y_{i_n}$'s.
%i.e., $M =K\< y_{i_1}y_{i_2}\dots y_{i_n} \mbox{ with }(i_1,\dots,i_n)\in \{1,\dots,n\}^n\>_\s$.
Then $M$ is a $\s$-Picard-Vessiot extension for
$\de(y)=A_{\mathcal{L}^{\circledS n}}y$ over $K$. By Theorem \ref{theo:secondfundamentaltheorem}, the group $k$-$\s$-scheme $\sgal(L|M)$ is a normal subgroup of $\sgal(L|K)$ and the quotient $\sgal(L|K)/\sgal(L|M)$ is isomorphic to $\sgal(M|K)$. Since $K$ is inversive, Corollary 4.4 in \cite{articleone} implies that $\sgal(L|K)$ and $\sgal(M|K)$ are  absolutely $\s$-reduced over $k$.  The Zariski closure of $\sgal(L|K)$ (resp. $\sgal(M|K)$) corresponds to the classical Galois group of $\de(y)=A_\mathcal{L}y$ (resp. $\de(y)=A_{\mathcal{L}^{\circledS n}}y$), whose $k$-points are the $K$-$\de$-automorphism
$\Aut^\de(L_0|K)$ (resp. $\Aut^\de(M_0|K)$)  of $L_0=K\langle y_1,\dots,y_n\rangle_\de$ (resp. $M_0=K\langle y_{i_1}y_{i_2}\dots y_{i_n}|\ (i_1,\dots,i_n)\in \{1,\dots,n\}^n\rangle_\de $)
(see Proposition \ref{prop:Zariskiclosures}).
Since $M_0$ is a classical Picard-Vessiot extension,
the usual Galois correspondence shows that $\Aut^\de(L_0|M_0)$ is a normal algebraic subgroup of $\Aut^\de(L_0|K)=\Sl_n(k)$. The center of $\Aut^\de(L_0|K)$ is the cyclic group $\mu_n(k)$ of order $n$. Moreover, the action of $\lambda \in \mu_n(k)$ as element of $\Aut^\de(L_0|K)$ is determined by its action on the $y_i$'s,
which is just the multiplication by $\lambda$.
Then,  $\mu_n (k) \subset\Aut^\de(L_0|K) $ fixes $M_0$. Thus, the algebraic group $\Aut^\de(L_0|M_0)$
is a normal subgroup of $\Sl_n(k)$ containing $\mu_n(k)$.  The group $\Aut^\de(L_0|M_0)$ differs
from $\Sl_n(k)$ because otherwise $M_0=K$ and the $y_i$'s
would be algebraically dependent over $K$.  Since $\mu_n(k)$ is maximal
among the  normal subgroups of $\Sl_n(k)$, we find that $\Aut^\de(L_0|M_0)=\mu_n(k)$.
Then, the Zariski closure of $\sgal(M|K)$ is $\operatorname{PSl}_{n}(k)$, i.e., the projective special linear algebraic group over $k$.
\par
Finally, the  $\ds$-field extension $L|M$ is an algebraic field extension and  $\strdeg(L|K)=\strdeg(M|K)$ and $\sdim (\sgal(L|K))=\sdim (\sgal(M|K))$. This shows that $\sgal(L|K)$ is a proper $\s$-closed subgroup of $\Sl_{n,k}$ if and only if  $\sgal(M|K)$ is a proper $\s$-closed subgroup of $\operatorname{PSl}_{n,k}$.
We conclude the proof by applying Proposition \ref{prop:simpleintegra} to  $\sgal(M|K)$, a  $\s$-reduced, Zariski dense $\s$-closed subgroup of the simple algebraic group $\operatorname{PSl}_{n,k}$.
\end{proof}

As mentioned above, the third point in the equivalence of Corollary \ref{cor:slnsympower} may be
difficult to test because of the  order of the $n$-th symmetric power of the initial differential equation.
Therefore we propose below a more general method.

\subsection{Characterization of $\s$-integrability for almost simple groups}

For a differential system $\de(y)=Ay$ with almost simple usual Galois group, we prove in Theorem \ref{thm:almostintegra} that the existence of transformal relations between the solutions in some $\s$-Picard-Vessiot extension $L$ is equivalent to the $\s^ d$-integrability of $\de(y)=Ay$ over the relative algebraic closure $K'$ of $K$ inside $L$. That is, we extend Proposition \ref{prop:simpleintegra} to the case of almost simple groups by allowing the $\s^ d$-integrability to be algebraic and not only rational.
\par
Unlike to the question of rational solutions, there is, to our knowledge, no general algorithmic
answer to the existence  of algebraic solutions of differential equations.
However, necessary conditions for their existence exist: We cite for instance the easy implication in the
Grothendieck conjecture on $p$-curvatures, see \cite{Katzdiffalg}. In \cite{ChenKauersSing} the reader will find some criteria
for the existence of algebraic solution based on residues.
\par

We will need a simple lemma on extensions of $\s$-fields. Recall that a $\s$-ring $R$ is called a $\s$-domain if it is an integral domain and $\s\colon R\to R$ is injective. Moreover, an extension of $\s$-fields $L|K$ is called $\s$-regular if $L\otimes_K\wtilde{K}$ is a $\s$-domain for every $\s$-field extension $\wtilde{K}$ of $K$.

\begin{lemma} \label{lemma:sintegral}
Let $L|K$ be an extension of $\s$-fields such that $K$ is inversive and let $K'$ denote the relative algebraic closure of $K$ in $L$. Then $L|K'$ is $\s$-regular.
\end{lemma}
\begin{proof}
By Lemma A.13 (ii) in \cite{articleone} it suffices to a find an inversive algebraically closed $\s$-field extension $\wtilde{K}$ of $K'$ such that $L\otimes_{K'}\wtilde{K}$ is a $\s$-domain. Let $\wtilde{K}$ denote an algebraic closure of $K$ containing $K'$ and extend $\s$ from $K'$ to $\wtilde{K}$. Since $K$ is inversive also $\wtilde{K}$ is inversive. As the extension $L|K'$ is regular, $L\otimes_{K'}\wtilde{K}$ is a field and consequently $\s$ is automatically injective on $L\otimes_{K'}\wtilde{K}$.
\end{proof}

\par Now, we are able to state our main theorem.

\begin{thm}\label{thm:almostintegra}
Let $K$ be an \emph{inversive} $\ds$-field, $\de(y)= Ay $ a differential system with $A \in K^{n\times n}$
and $L|K$ a $\s$-Picard-Vessiot extension for $\de(y)=Ay$.
We assume that the Zariski closure $H$
of $\sgal(L|K)$ inside $\Gl_{n,k}$ is an absolutely
almost simple algebraic group of dimension $t \geq 1$
over $k=K^\de$. Let $K'$ be the relative algebraic closure of $K$ inside $L$.
Then the following statements are equivalent:
\begin{enumerate}

\item $\sgal(L|K')$ is a proper $\s$-closed subgroup of $H$.

\item
The $\s$-transcendence degree of  $L|K$
is strictly  less than $t$.

\item
There exists $d \in \Z_{>0}$ such that the system $\de(y)=Ay$ is $\s^d$-integrable over $K'$.
\end{enumerate}

\end{thm}

\begin{proof}
As in the proof of Proposition \ref{prop:simpleintegra} condition (iii) implies that $\strdeg(L|K')=0$. But then $\strdeg(L|K)=\strdeg(L|K')=0<t$. Thus (iii) implies (ii).

Condition (ii) implies that $\sdim(\sgal(L|K))=\strdeg(L|K)<t$. Therefore $\sgal(L|K)$ is properly contained in $H$. But then also $\sgal(L|K')\leq\sgal(L|K)$ is properly contained in $H$.

It remains to show that (i) implies (iii). Let us first show that the Zariski closure $H'$ of $G'=\sgal(L|K')$ in $\Gl_{n,k}$ equals $H$.
Suppose that $H'$ is properly contained in $H$. Then, since $H$ is connected, the dimension $t'$ of $H'$ is strictly smaller than $t$.
Let $Y\in\Gl_n(L)$ be a fundamental solution matrix for $\de(y)=Ay$. Since $K'|K$ is algebraic we arrive at the contradiction $$\trdeg(K(Y)|K)=\trdeg(K'(Y)|K')=t'<t=\trdeg(K(Y)|K).$$
So $H=H'$ and condition (i) means that $G'$ is properly contained in $H'$.

By Lemma \ref{lemma:sintegral} we know that $L|K'$ is $\s$-regular. By Proposition 4.3 (iii) in \cite{articleone} this implies that $G'$ is absolutely $\s$-integral. To apply Theorem \ref{thm:classalmostsimple} we need to go the algebraic closure of $k$. So let $\wtilde{k}$ denote the algebraic closure of $k$ and extend $\s$ from $k$ to $\wtilde{k}$. Note that since $K$ is inversive also $k$ and $\wtilde{k}$ are inversive. Extending the $\de$-constant from $k$ to $\wtilde{k}$ as in Lemma \ref{lemma:baseextensionforsPV} we obtain $\s$-Picard-Vessiot extensions $\wtilde{K}\subset\wtilde{K'}\subset\wtilde{L}$ with field of $\de$-constants $\wtilde{k}$. Since $G'$ is absolutely $\s$-integral $\wtilde{G'}:=\sgal(\wtilde{L}|\wtilde{K'})=G'_{\wtilde{k}}$ is $\s$-integral. Since $G'$ is properly contained in its Zariski closure $H=H'$, also $\wtilde{G'}$ is properly contained in its Zariski closure $H_{\wtilde{k}}$. We can thus apply Theorem \ref{thm:classalmostsimple}
to $\wtilde{G'}\leq\Gl_{n,\wtilde{k}}$ and combine it with Theorem \ref{thm:integra} to find that $\de(y)=Ay$ is $\s^d$-integrable over $K'$ for a suitable integer $d\geq 1$.
\end{proof}

Given a differential system $\de(y)=Ay$ over a $\ds$-field $K$, it might be difficult to obtain information about the relative algebraic closure $K'$ of $K$ in a suitable $\s$-Picard-Vessiot extension. Nevertheless, Theorem \ref{thm:almostintegra} implies the following transformal independence criterion.
\begin{cor} \label{cor:forintro}
Let $K$ be an inversive $\ds$-field, $A\in K^{n\times n}$, $L$ a $\ds$-field extension of $K$ with $L^\de=K^\de$ and $Y\in\Gl_n(L)$ with $\de(Y)=AY$. Let $T$ be a transcendence basis of $K(Y)$ over $K$ and assume that the (usual differential) Galois group of $K(Y)|K$ is an absolutely almost simple algebraic group. If, for every $d\geq 1$, the linear differential equation $\de(B)+BA=\hslash_d\s^d(A)B$ over $K$ has no non-zero solution $B$ which is algebraic over $K$, then $T$ is transformally independent over $K$.
\end{cor}
\begin{proof}
We may assume that $L=K\langle Y\rangle_\s$. So $L|K$ is a $\s$-Picard-Vessiot extension. Suppose that $T$ is transformally dependent over $K$. Then $\strdeg(L|K)<|T|=\trdeg(K(Y)|K)$. Thus condition (ii) of Theorem \ref{thm:almostintegra} is satisfied and it follows that $\de(y)=Ay$ is $\s^d$-integrable over the relative algebraic closure $K'$ of $K$ in $L$ for some $d\geq 1$. This contradicts the assumption that $\de(B)+BA=\hslash_d\s^d(A)B$ has no non-zero algebraic solutions.
\end{proof}

\subsection{The case of $\Sl_2$}

We now apply Theorem \ref{thm:almostintegra} to the almost simple algebraic group $\Sl_2$.

\begin{cor}\label{cor:paramkov}
Let $K=k(x)$ be a field of rational functions
equipped with the derivation $\de=\frac{d}{dx}$ and an automorphism $\s$ commuting with $\de$,
such that $k\subset\C$ be an algebraically closed inversive $\s$-field.
We assume that the differential equation $\de^2(y)-ry=0$, with $r(x) \in K$, has (usual) Galois group $\operatorname{Sl}_2(k)$ and
we denote by $L|K$ one if its $\s$-Picard-Vessiot extensions.
Let  $K'$ be the relative algebraic closure of $K$ in $L$.
We have:
\begin{itemize}

\item
If the $\s$-transcendence degree of $L|K$ is strictly  less than $3$,
there exists $s \in \Z_{>0}$ such that the differential system
\beq\label{eq:intcond}
\l\{\begin{array}{l}
\de^2(b) +(\s^s(r)-r)b=2\de(d) \\ \\
\de^2(d) +(\s^s(r)-r)d=2 \s^s(r) \de(b) +\de(\s^s(r))b
\end{array}\r.
\eeq
has a non-zero algebraic  solution $(b,d)\in (K')^2$.

\item
If we can find a solution $(b,d)\in (K')^2$ of \eqref{eq:intcond}, such that the matrix
$B =\begin{pmatrix} d-\de(b)&b \\  \s^s(r)b-\de(d) & d \end{pmatrix} $ is invertible, then
the $\s$-transcendence degree of  $L|K$ is strictly less than $3$.
\end{itemize}
\end{cor}

\begin{proof}
The group $\Sl_{2,k}$ is an almost simple algebraic group of dimension $3$.
Since $K$ is an inversive $\s$-field and $k$ is algebraically closed \footnote{This assumption insures that
the usual Galois group is well-defined, i.e., unique up to isomorphism.}, we can apply Theorem \ref{thm:almostintegra} and
obtain the equivalence between the following statements:
\begin{itemize}

\item
The $\s$-transcendence degree of $L|K$ is strictly  less than $3$.

\item
There exist $B\in\Gl_n(K')$ and $s \in \Z_{>0}$ such that the system
$\l\{\begin{array}{l}
\de(y)=Ay:=\begin{pmatrix} 0&1\\r&0\end{pmatrix}y\\
\s^s(y)=By
\end{array}\r.$ is compatible.
\end{itemize}
This means that $A$ and $B$ verify
$\de(B)+BA=\s^s(A)B$.
If we write
$$
B =\begin{pmatrix} a&b \\c& d \end{pmatrix},
$$
the integrability condition  of order $s$
is equivalent to the following system of differential equations
$$
\l\{\begin{array} {l}
\de(a)=c-rb \\
\de(b)=d-a \\
\de(c)=\s^s(r)a-rd \\
\de(d)=\s^s(r)b-c
\end{array}\r..
$$
Eliminating $a$ and $c$, we obtain the differential system \eqref{eq:intcond}. This proves the first statement.
\par
Suppose that we have a solution $(b,d)\in (K')^2$ such that the matrix $B$ is invertible, then the system $\de(y)=Ay$ is $\s^s$-integrable and therefore $\strdeg(L|K)< 3$.
\end{proof}

Assuming that a second order linear differential equation has usual Galois group $\operatorname{Sl}_2(k)$ is not a restrictive assumption,
as we know from a famous paper  of Kovacic (\cite{Kovalg}, \cite{Kovacicalg2005}),
that provides an algorithm to test whether a differential equation of order $2$ with coefficients in $k(x)$
has liouvillian solutions or not. Based on this property, the algorithm actually calculates the Galois group of the differential equation.
\par
Let $K=k(x)$, where $k\subset\C$ is an algebraically closed field, and let $\de=\frac{d}{dx}$, as above.
Starting with a differential equation $\de^2(y)+a\de(y)+by=0$ of order $2$ with coefficients in $K(x)$, one computes
its normal form
\beq\label{eq:normalform}
\de^2(z)-rz=0,
\eeq
by means of a change of unknown function $z=fy$, where $f$ is a solution of $\de(f)=\frac{a}{2}f$ and $r=b-\frac{1}{4}a^2-\frac{1}{2}\de(a)$.
In the form \eqref{eq:normalform},
the existence of liouvillian solutions is equivalent to finding solutions of the Riccati's equation
$$
u^2 -\de(u)=r(x)
$$
in $\bigcup_{m\in \Z_{\geq 0}} \C[x^{\frac{1}{m}}]$. Kovacic's algorithm says, in particular, that:

\begin{prop}\label{prop:kovacicSl2}
If $\de^2(y)-ry=0$ has no liouvillian solutions, then its (usual) Galois group
is $\operatorname{Sl}_2(k)$.
\end{prop}

\begin{rmk}
In \cite{kovthomas} (see also \cite{Arreche:ComputingTheDifferentialGaloisGroupOfSecondOrder}), the author gives a  version of the Kovacic algorithm with differential parameters.
The author gives, for instance, a characterization
of the differential dependence w.r.t. the parameters $\ul t=(t_1,\dots,t_k)$ of the solutions  of  $\de^2(y)-r(x,\ul t)y=0$,
in terms of existence of rational solutions of auxiliary differential equations. A similar result should be true also in the present setting.
\end{rmk}

\subsection{Airy's equation}

For a differential system $\de(y)=Ay$ with usual almost simple Galois group,
Theorem \ref{thm:almostintegra} shows that the transformal relations satisfied by the solutions of
$\de(y)=Ay$ in some $\s$-Picard-Vessiot extension $L$ are determined  by the existence of  solutions
of an auxiliary linear differential system, lying  in the relative algebraic
closure $K'$ of $K$ inside $L$. The field $K'$ is not easy to handle in general but if the solutions
in $L$ are sufficiently regular, one can easily show that $K'=K$.
Below we illustrate this remark and combine it with Corollary \ref{cor:paramkov} to characterize the transformal
relations satisfied by the solutions of the Airy equation.

\begin{lemma} \label{lemma:ratfunctrelalgclosed}
Let $\C(x)$ be the field of rational functions over $\C$ and let $M$ be the field of
meromorphic functions over $\C$. Then, $\C(x)$ is relatively algebraically closed in $M$.
\end{lemma}

\begin{proof}
Let $f \in M$ be algebraic over $\C(x)$. Consider the function field $E=\C(x)(f)$. It is a finite separable extension of the rational function field $\C(x)$. Then, Corollary III.5.8 in  \cite{Stichtenoth} implies
that either $E=\C(x)$ or $E|\C(x)$ is ramified. This is an easy consequence of the Riemann-Hurwitz genus formula. Since $E$ is a sub-field of $M$ and since  a meromorphic function has no branch point, $E|\C(x)$ can not be ramified and then $f \in \C(x)$.
\end{proof}

We are now able to apply Corollary \ref{cor:paramkov} to the  case of the Airy equation
\beq\label{eq:airy}
\de^2(y)-xy =0.
\eeq
Notice that it has an irregular singularity at $\infty$,
and that all the other points of $\mathbb A^1_\C$ are ordinary.
This immediately implies that \eqref{eq:airy} admits a basis of solutions $A(x)$ and $B(x)$ in the field $M$ of meromorphic functions over $\C$.

\begin{cor}\label{cor:Airy}
Let $\C(x)$ be the field of rational function over the complex numbers,
equipped with the derivation $\de=\frac{d}{dx}$ and the automorphism $\sg:f(x)\mapsto f(x+1)$,
and $M$ be the field of meromorphic functions over $\C$.
In the notation above, let $L = \C(x)\<A(x), B(x), \de(A(x)), \de(B(x))\>_\s \subset M $ be the $\s$-Picard-Vessiot extension for the Airy equation contained in $M$.
Then, $\sgal(L|\C(x))$ is equal to  $\operatorname{Sl}_{2,\C}$ and the functions $A(x)$, $B(x)$ and $\de(B(x))$ are transformally independent over $\C(x)$.
\end{cor}

\begin{proof}
By Example 4.2.9 in \cite{katzcalgal}, the usual Galois group
of the Airy equation is $\operatorname{Sl}_2(\C)$. Moreover by Lemma \ref{lemma:ratfunctrelalgclosed}, the
field $\C(x)$ is relatively algebraically closed in $L$.
\par
We can then apply Corollary \ref{cor:paramkov} to the case $K=\C(x)$, $\s(x)=x+1$ and $r(x) =x$.
Then, if there exists $s \in \Z_{>0}$ such that the system \eqref{eq:intcond} has a rational solution, we have to find
$(b,d) \in \C(x)^2$ such that
$$
\l\{\begin{array}{l}
\de^2(b) +sb=2\de(d) \\
\de^2(d) +sd=2 (x+s) \de(b) +b
\end{array}\r..
$$
If we eliminate $d$ from the previous system, we get
\beq\label{eq:accessoryorder4eq}
\de^4(b)-(4x+2s)\de^2(b)-6\de(b)+s^2b=0.
\eeq
Let us assume that this last equation has a non-zero solution $b_0\in\C(x)$.
It is easy to see that $b_0$ has no pole. In fact, if $\alpha$ was a pole of order $r$ of $b_0$ then it would be a pole of order $r+4$ of
$\de^4(b_0)-(4x+2s)\de^2(b_0)-6\de(b_0)+s^2b_0=0$, and this is impossible. Therefore $b_0$ must be a polynomial.
Finally, if $b_0\neq 0$, then $ \de^4(b_0)-(4x+2s)\de^2(b_0)-6\de(b_0)+s^2b_0$ is a non-zero polynomial
of the same degree than $b_0$, which is again impossible. So we conclude that $b_0=0$ and that the differential system above does not have any rational solution.
Therefore $\strdeg(L|\C(x))=3$, which allows to conclude.
Since the relative algebraic closure of $K$ in $L$ equals $K$, condition (i) in Theorem \ref{thm:almostintegra} implies that $\sgal(L|K)$ is not a proper $\s$-closed subgroup of $\Sl_{2,\C}$. Therefore $\sgal(L|\C(x))=\Sl_{2,\C}$.
\end{proof}

%%%%%%%%%%%%%%%%%%%%%%%%%%%%%%%%%%%%%%%%%%%%%%%%%%%%%%%%%%%%%%%%%%%%%%%%%%%%%%%%%%%%%%%%%%%%%%%%
%%%%%%%%%%%%%%%%%%%%%%%%%%%%%%%%%%%%%%%%%%%%%%%%%%%%%%%%%%%%%%%%%%%%%%%%%%%%%%%%%%%%%%%%%%%%%%%%
\appendix

\section{Some classification results for $\s$-algebraic groups}
\label{sec:classgroup}

In this appendix we have collected some results on the classification of $\s$-closed subgroups of certain algebraic groups. The results are typically of the following type: If $G$ is a $\s$-closed subgroup of a certain algebraic group, then $G$ is of a very special form. In the main body of the text, via the $\s$-Galois correspondence, this
yields results of the type: If there is a transformal relation between the solutions of a certain linear differential equation, then there is a transformal relation of a very special form.

We have made it a general assumption (in Section \ref{subsec:differentialalgebra}) that all fields are of characteristic zero. However, in this appendix we do make this assumption but rather state the characteristic zero assumption whenever it is convenient.

\subsection{The $\s$-closed subgroups of $\Ga^n$}
\label{subsec:subgroupsGa}

Let $k$ be a $\s$-field. As usual we denote by $\Ga$ the additive algebraic group over $k$. We think of the vector group $\Ga^n$ as a $\s$-algebraic group, i.e., $\Ga^n(S)=S^n$ for any $k$-$\s$-algebra $S$.
The $\s$-coordinate ring of $\Ga^n$ is $k\{\Ga^n\}=k\{ x_1,\ldots,x_n \}_\s$.
A linear, homogenous $\s$-polynomial $p=\sum_{i,j}\lambda_{i,j}\s^j(x_i)\in k\{ x_1,\ldots,x_n \}_\s$ defines a $\s$-closed subgroup of $\Ga^n$.
The following theorem provides a converse.

\begin{thm}\label{thm:clasga}
Let $k$ be a $\s$-field of characteristic zero and let $G$ be a $\s$-closed subgroup of $\Ga^n$.
Then, there exists a set $F$ of linear homogeneous $\s$-polynomials such that
$$G(S)=\{g\in\Ga^n(S)|\ p(g)=0 \text{ for } p\in F\}$$
for any $k$-$\s$-algebra $S$.
\end{thm}
\begin{proof}
\par For $d \in \N$, let $G[d]$ be the $d$-th order Zariski closure of $G$ inside $\mathbf{G}_a^n$.
By Lemma A.39 in \cite{articleone},
$G[d]$ is an algebraic subgroup of $(\mathbf{G}_a^n)_d=\mathbf{G}_a^{n(d+1)}$.
The claim now follows from the fact that the algebraic subgroups of $\mathbf{G}_a^{n(d+1)}$ are defined by linear homogeneous polynomials.
% Since the vanishing ideal  of
%  a subgroup scheme inside  $\mathbf{G}_a^{n(d+1)}$ is generated by linear homogeneous polynomials, we get that, for all $d \in \N$, the ideal $\I(H[d])= \I(H)\cap k[\mathbf{G}_a^n]_d$
%  is generated by linear homogeneous $\s$-polynomials of order less than or equal to $d$. This implies that $\I(H)$ itself is generated as a $\s$-ideal by
% linear homogeneous $\s$-polynomials .
% \par By \cite[Theorem 2.4.11]{Levin}, one can find a finite set of  linear homogeneous $\s$-polynomials $L_1, \ldots, L_k$ in $\I( H)$ such that  $\Sigma:=\{ L_1,\ldots,L_k \}$ is a characteristic set of
% $\I(H)$. Let
% $[\Sigma]$ be the $\s$-ideal generated by $\Sigma$. We have $[\Sigma] \subset \I(H)$.
%   Since the elements $L_i$ of $\Sigma$ are linear homogeneous $\s$-polynomials, their initials $I_{L_i}$
% belongs to $k\subsetneq \{0\}$. Now, let $C \in\I(H)$. By \cite[Theorem 2.4.1]{Levin}, there exist a $\s$-polynomial $C_0 \in k\{X\}_\s$, reduced with respect to $\Sigma$, and an element $J $
% either equal to $1$ or to a product of finitely many $\s$-iterates of $I_{L_i}$  such that $JC -C_0 $ belongs to
% $[\Sigma]$.  Since $J \in k\subsetneq \{0\}$, we get that $C_0$ belongs to $\I(H)$ and is reduced w.r.t. $\Sigma$. By \cite[Proposition 2.4.4]{Levin}, $C_0=0$ which means that
% $C \in [\Sigma]$. Therefore,  $\I(H)= [\Sigma]$.
\end{proof}

\begin{rem}
In positive characteristic, the situation is slightly more complicated because one has to take into account the Frobenius. For example, if $k$ is a $\s$-field of characteristic $p>0$, then
$$G(S)=\{g\in S|\ g+\s(g)^p=0\}\leq\Ga(S)$$ for any $k$-$\s$-algebra $S$,
defines a $\s$-closed subgroup $G$ of $\Ga$.
\end{rem}

\begin{cor} \label{cor:ssubgroupofGa}
Let $k$ be a $\s$-field of characteristic zero and $G$ a $\s$-closed subgroup of $\mathbf{G}_a$, not equal to $\mathbf{G}_a$. Then, there exists a non-zero, linear, homogenous $\s$-polynomial $p=a_n\s^n(x)+a_{n-1}\s^{n-1}(x)+\cdots+a_0x\in k\{x\}_\s$ such that
$$G(S)=\{ g\in S|\ p(g)=0\}$$ for any $k$-$\s$-algebra $S$. Moreover, if $a_0\neq 0$ then $G$ is $\s$-integral, i.e., the $\s$-ideal $\I(G)$ is prime and reflexive.
\end{cor}
\begin{proof}
By Theorem \ref{thm:clasga}, the vanishing ideal $\I(G)\subset k\{x\}_\s=k\{\mathbf{G}_a\}$ of $G$ is generated as a difference ideal by linear, homogeneous $\s$-polynomials. Let $p\in\I(G)$ be a non-zero, linear, homogeneous $\s$-polynomial of minimal order.
So
$$
p=a_n\s^n(x)+a_{n-1}\s^{n-1}(x)+\cdots+a_0x\in k\{x\}_\s,\quad a_n\neq 0,
$$
with $n$ minimal.
We have to show that $\I(G)=[p]$. Suppose there exists a non-zero, linear, homogeneous $\s$-polynomial in $\I(G)$ which does not lie in $[p]$.
Among these $\s$-polynomials choose one, say $q$, of minimal order. If $q=b_m\s^m(x)+\cdots+b_0x$, $b_m\neq 0$, then necessarily $m\geq n$ and $q-b_m\s^{m-n}(a_n)^{-1}\s^{m-n}(p)\in\I(G)$ has order strictly less than $m$. So $q-b_m\s^{m-n}(a_n)^{-1}\s^{m-n}(p)\in [p]$. But then also $q\in[p]$; a contradiction.

Let us now assume that $a_0\neq 0$. We have to show that $G$ is $\s$-integral. The coordinate ring $k\{G\}=k\{x\}_\s/[p]\simeq k[\overline{x},\s(\overline{x}),\ldots,\s^{n-1}(\overline{x})]$ is a polynomial ring in $n$ variables over $k$. To see that $\s$ is injective on $k\{G\}$ is suffices to note that $\s(\overline{x}),\s^2(\overline{x}),\ldots,\s^n(\overline{x})=-\frac{1}{a_n}(a_0\overline{x}+\cdots+a_{n-1}\s^{n-1}(\overline{x}))$ are algebraically independent over $k$.
\end{proof}

We recall the definition of linearly $\s$-closed and $\s$-closed field:

\begin{defi}\label{defi:linearlysclosed}[Cf. {Def. 3.1, p. 1330 in \cite{KowPillsigma}.}]
A $\s$-field $k$ is called \emph{linearly $\s$-closed} if every linear system of
difference equations over $k$ has a fundamental solution matrix in $k$. That is,
for every $B\in\Gl_n(k)$ there exists $Y\in\Gl_n(k)$ with $\s(Y)=BY$. We say that
a $\s$-field $k$ is $\s$-closed\footnote{A $\s$-closed $\s$-field is also called a model of ACFA, in model theory language.} if every system
of difference polynomial equations with coefficients in $k$, which posses a solution
in some $\s$-field extension of $k$, has a solution in $k$ (see also \S 1.1 in \cite{ChatHrusPet}).
\end{defi}

\begin{rmk}
The reader familiar with the Galois theory of linear difference equations might
be slightly puzzled by the above definition, because usually one is looking for
fundamental solution matrices in finite products of fields, rather than fields
(\cite{SingerPut:difference}). The definition makes sense because we do not make
any assumption on the constants of $k$. (The constants of a linearly $\s$-closed
$\s$-field can be non-algebraically closed.) Note that a $\s$-closed $\s$-field is linearly $\s$-closed.
\end{rmk}

The difference Wronskian lemma implies that a $\s$-field $k$ is linearly $\s$-closed
if and only if any linear difference equation of order $s$, with coefficients $k$:
$$
a_0y+a_1\s(y)+\dots+a_s\s^s(y)=0,
\hbox{~with $a_0,\dots,\a_s\in k$, $a_0a_s\neq 0$,}
$$
has $s$ solutions in $k$, linearly independent over $k^\s$.
It follows that:

\begin{cor}\label{cor:linsclosedsubgroupsGa}
Let $k$ be a linearly $\s$-closed $\s$-field of zero characteristic and let $G$ be a proper $\s$-closed subgroup  of $\mathbf{G}_a$, such that
$\mathcal{L} = \sum_{i=0}^s \lambda_i \s^i(x)$, $\lambda_0\lambda_s\neq 0$, generates $\I(G)$ as a $\s$-ideal.
Then, $G(k)$ is a $k^\s$-vector space of dimension $s$, generated by a basis of solutions in $k$ of the
linear difference equation $\mathcal L(y)=0$.
Moreover, for any $k$-$\s$-algebra $S$, we have $G(S)=G(k)\otimes_{k^\s} S^\s$.
\end{cor}

\subsection{The $\s$-closed subgroups of $\Gm^n$}

By a multiplicative function on $\Gm^n$, we mean a $\s$-polynomial in $k\{\Gm^n\}=k\{x_1,\dots,x_n,x_1^{-1},\dots,x_n^{-1}\}_\s$ which is a  product of monomials $\s^i(x_j)$. The structure of the $\s$-closed subgroups of $\Gm^n$ was already noted in Lemma A.40 in \cite{articleone}.

\begin{thm}\label{thm:classgm}
Let $k$ be a $\s$-field and let $G$ be a $\s$-closed subgroup of $\Gm^n$.
Then there exists a set $\Psi$ of multiplicative functions such that
$$
G(S)=\{g\in \Gm^n(S)|\ \psi(g)=1 \text{ for } \psi\in\Psi\},
$$
for any $k$-$\s$-algebra $S$.
\end{thm}

In the following proposition, we show that, for certain $\s$-closed subgroup of $\Gm$, we can give
a more precise description.

\begin{prop}\label{prop:mup}
Let $k$ be a $\s$-field, $p$ a prime number,
$\mu_p \leq \Gm$ the group scheme of $p$-th roots of unity
and $G$ a proper $\s$-closed subgroup of $\mu_p$. Then, there exist
two integers $m \geq 0$ and $d \geq 1$ such that
$$
G(S) \leq \{g\in \mu_p(S)|\;\s^{m+d}(g)=\s^m(g) \},
$$
for any $k$-$\s$-algebra $S$.
\end{prop}

\begin{proof}
Note that $\s^l(g)^p=1$ for any integer $l\geq 0$, $g\in G(S)\subset S^\times$ and any $k$-$\s$-algebra $S$. Since $G$ is properly contained in $\mu_p$ it follows from Theorem \ref{thm:classgm} that there exist $\alpha_0,\ldots,\alpha_l\in\{0,\ldots,p-1\}$, not all equal to zero, such that
$$g^{\a_0}\s(g)^{\a_1}\cdots\s^l(g)^{\a_l}=1 \text{ for all } g\in G(S).$$
We may assume that $a_l\neq 0$, so that there exist $a,b\in\Z$ such that $a\alpha_l+bp=1$. Then
$$\s^l(g)=\s^l(g)^{1-bp}=\s^l(g)^{a\alpha_l}=g^{-a\a_0}\s(g)^{-a\a_1}\cdots\s^{l-1}(g)^{-a\a_{l-1}}.$$
Thus there exist $\beta_0,\ldots,\beta_{l-1}\in\{0,\ldots,p-1\}$ such that
\begin{equation}\label{eq:prime}
\s^l(g)=g^{\beta_0}\s(g)^{\beta_1}\cdots\s^{l-1}(g)^{\beta_{l-1}}.
\end{equation}
Applying $\s$ to both sides of this identity and recursively substituting equation (\ref{eq:prime}), we find that for every $m\geq l$ there exist
$\beta_{0,m},\ldots,\beta_{l-1,m}\in\{0,\ldots,p-1\}$ such that
$$\s^m(g)=g^{\beta_{0,m}}\s(g)^{\beta_{1,m}}\cdots\s^{l-1}(g)^{\beta_{l-1,m}}.$$
Since there are only finitely many possibilities for the exponents $\beta_{i,m}$, this shows that there must exist distinct integers $m_1,m_2\geq l$ with $\s^{m_1}(g)=\s^{m_2}(g)$ for all $g\in G(S)$ and any $k$-$\s$-algebra $S$.
\end{proof}

\subsection{The $\s$-closed subgroups of $\Ga\rtimes\Gm$}
\label{subsubsec:classg2}

Let $k$ be a $\s$-field of characteristic zero and let $\mathbb G$ denote the algebraic subgroup of $\Gl_{2,k}$ given by
$$
\mathbb{ G} (S) = \left\{ \begin{pmatrix}\alpha & \beta \\ 0 &1 \end{pmatrix} \Big|\ \alpha \in S^\times,\ \beta \in S \right\}\leq \Gl_2(S),
$$
for any $k$-algebra $S$. Notice that $\mathbb{G}\simeq \Ga\rtimes\Gm$.
Let $\mathbb{G}_u$ denote the algebraic subgroup of $\mathbb{G}$ corresponding to $\Ga$, i.e.,
$$
\mathbb{G}_u(S)=\left\{ \begin{pmatrix}1 & \beta \\ 0 &1 \end{pmatrix} \Big|\ \beta \in S \right\}\leq\Gl_2(S),
$$
for any $k$-algebra $S$.

\begin{thm}\label{thm:classg2}
Let $G$ be a $\s$-closed subgroup of $\mathbb{G}$ such that the scheme theoretical intersection
$G_u:=G \cap \mathbb{G}_u$  is properly contained in $\mathbb{G}_u$. Then
$G$ has one of the following properties:
\begin{enumerate}
\item there exists an integer $n\geq 0$ such that $\s^n(\beta)=0$ for all $\begin{pmatrix}1 & \beta \\ 0 &1 \end{pmatrix}\in G_u(S)$ and all $k$-$\s$-algebras $S$;
\item there exist integers $n>m\geq 0$ such that $\s^n(\alpha)=\s^m(\alpha)$ for all $\begin{pmatrix}\alpha & \beta \\ 0 &1 \end{pmatrix}\in G(S)$ and all $k$-$\s$-algebras $S$.
\end{enumerate}
\end{thm}
% We propose here too proofs of the Theorem. One uses the vanishing ideals and the other is more or less a "points "version.
% \begin{proof}
% First of all let us remark that $\mathbb{G}_u$ is a normal subgroup $k$-$\s$-scheme of $\mathbb{G}$ which is isomorphic to the additive group. Moreover,
% the quotient $\mathbb{G}/\mathbb{G}_u$ is isomorphic to the multiplicative group.
% \begin{enumerate}
% \item The group $H_u$ is a proper $\s$-closed subgroup of  the additive group and thus by Theorem \ref{thm:clasga}, its vanishing ideal is  of the form described.{subsec:zarclos}
% \item Now, let $L(y) \in k\{y\}_\s$ be a non-zeron linear homogeneous $\s$-polynomial  such that $\I(H_u)=[L(y)]$.  We may  choose $\{L(y)\}$
% such that the order $k$ of $L$ w.r.t. $\s$ is minimal. Since $H_u$ is a normal sub-group
% of $H$, we have a  morphism of group $k$-$\s$-schemes $H \times H_u \rightarrow H_u, (g,h) \mapsto ghg^{-1}$. Since $H_u$ is an abelian group, this morphism
% factors through $H/H_u \times H_u \rightarrow H_u$. This means that $L(xy)  \in \I(H_u)$ modulo $\I(H/H_u)$. Now , $L(xy)-\s^k(x)L(y)$ belongs
% to $\I(H_u)$ modulo $\I(H/H_u)$. Then, either $L(y)= \s^k(y)$ or there exists $l \in \N$ such that   $\s^k(x)-\s^l(x) \in \I(H/H_u)$.
% \end{enumerate}
% \end{proof}
\begin{proof}
Since $G_u$ is properly contained in $\mathbb{G}_u$, it follows from Corollary \ref{cor:ssubgroupofGa} that there exists a non-zero, linear, homogeneous difference polynomial $p \in k\{x\}_\s$ such that
$$G_u(S)=\left\{\begin{pmatrix}1 & \beta \\ 0 &1 \end{pmatrix}\in \mathbb{G}_u(S)\Big|\ p(\beta)=0\right\}$$  for any $k$-$\s$-algebra $S$. If $p$ is a monomial we are in case (i). Otherwise, we may assume that $p$ is of the form
$$p(x) =\s^n(x) +a_{n-1}\s^{n-1}(x)+\cdots+a_m\s^m(x)$$
where $n>m\geq 0,\ a_i \in k$ for $i=m,\hdots,n-1$ and $a_m \neq 0$.
Let $S$ be a $k$-$\s$-algebra and $g =\begin{pmatrix}\alpha & \beta \\ 0 &1 \end{pmatrix} \in G(S)$.
We have to show that $\s^n(\alpha)=\s^m(\alpha)$. Note that for $h=\begin{pmatrix}1 & \beta_1 \\ 0 &1 \end{pmatrix} \in G_u(S)$, we have $ghg^{-1} =\begin{pmatrix}1& \alpha \beta_1 \\ 0 &1 \end{pmatrix} \in G_u(S)$.
Therefore $p(\alpha\beta_1)=0$ for all $\beta_1\in S$ with $p(\beta_1)=0$.
Let $$A =\begin{pmatrix} 0&1&0 &\cdots  &0 \\
                    			0& 0&1& \cdots & 0 \\
			                \vdots & \vdots & \vdots & \ddots & \vdots\\
			                0 & 0 & \cdots& \cdots & 1& \\
			                -a_m&-a_{m+1}& \cdots& \cdots & -a_{n-1}  \end{pmatrix}\in\Gl_{n-m}(k).$$
Let $y_1,\hdots,y_{n-m},\s(y_1),\hdots,\s(y_{n-m}),\hdots,\s^{n-1}(y_1),\hdots,\s^{n-1}(y_{n-m})$ be indeterminates over $S$ and abbreviate
$y = (y_1,\hdots,y_{n-m})$. We extend the action of $\s$ on $S$ to an action on $S[y,\ldots,\s^{n-1}(y)]$ as suggested by the names of the variables and by setting $\s(\s^{n-1}(y_i))=-a_{n-1}\s^{n-1}(y_i)-\cdots-a_m\s^m(y_i)$ for $i=1,\ldots,n-m$.
%   $S$-algebra $S[\overline{Y},\s(\overline{Y}),\hdots,\s^{n-1}(\overline{Y}) ]$ extend the action of $\s$ over $S$ by setting
%   $$\begin{array}{cccc}
%    \s : & \overline{Y} &\mapsto & \s(\overline{Y}) \\
%      & \s(\overline{Y}) &\mapsto& \s^2(\overline{Y}) \\
%      & &\hdots & \\
%     & \s^{n-1}(\overline{Y}) &\mapsto & -a_m\s^m(\overline{Y})-\hdots-a_{n-1}\s^{n-1}(\overline{Y}) \end{array}.$$
Then the matrix $$Z=\begin{pmatrix}
 \s^m(y_1) & \cdots & \s^m(y_{n-m}) \\
\vdots & & \vdots \\
\s^{n-1}(y_1) & \cdots & \s^{n-1}(y_{n-m})
\end{pmatrix}$$
satisfies $\s(Z)=AZ$, which implies $\s(\det(Z))=\det(A)\det(Z)$. So we can extend the action of $\s$
   to $S^\p:=S[y,\ldots,\s^{n-1}(y),\frac{1}{\det(Z)}]$.
Since $p(y_i)=0$ we have $p(\alpha y_i)=0$ for $i=1,\ldots,n-m$.

This implies
   $$ \s\left( \begin{pmatrix}
   \s^m(\alpha) & 0 & \hdots & 0\\
   0 & \ddots & \ddots & \vdots \\
   \vdots & \ddots & \s^{n-2}(\alpha)& 0 \\
   0 & \hdots  & 0 & \s^{n-1}(\alpha) \end{pmatrix} Z \right)= A \begin{pmatrix}
   \s^m(\alpha) & 0 & \hdots & 0\\
   0 & \ddots & \ddots & \vdots \\
   \vdots & \ddots & \s^{n-2}(\alpha)& 0 \\
   0 & \hdots  & 0 & \s^{n-1}(\alpha) \end{pmatrix} Z .$$
Since two fundamental solution matrices for $\s(Y)=AY$ only differ by a constant matrix, there exists $C \in\Gl_{n-m}((S^\p)^\s)$ such that $\begin{pmatrix}
   \s^m(\alpha) & 0 & \hdots & 0\\
   0 & \ddots & \ddots & \vdots \\
   \vdots & \ddots & \s^{n-2}(\alpha)& 0 \\
   0 & \hdots  & 0 & \s^{n-1}(\alpha) \end{pmatrix} Z  =ZC$.
By applying the determinant to the previous equation, we find $\s^m(\alpha)\cdots\s^{n-1}(\alpha)\in  (S^\p)^\s$ and thus
$$\s^{m+1}(\alpha)\cdots\s^n(\alpha)=\s^{m}(\alpha)\cdots\s^{n-1}(\alpha).$$ Since $\alpha \in S^\times$ we obtain $\s^m(\alpha)=\s^n(\alpha)$ as desired.
\end{proof}

\subsection{The Zariski dense $\s$-closed subgroups of almost simple algebraic groups}

This section is devoted to the classification of the Zariski dense $\s$-closed subgroups of almost simple
algebraic groups. In the framework of ``difference varieties'', this classification
has been achieved  in Proposition 7.10, p. 309, in \cite{ChatHrusPet}. Since we have to deal
with not necessarily perfectly $\s$-reduced difference group schemes, we need to
obtain a schematic version of the result of Chatzidakis, Hrushovski and Peterzil.
In full generality, this schematic version fails to be true  (see Remark \ref{rem:theoremfailsforalmostsimple}). But we succeed to obtain a satisfactory statement for Zariski dense
$\s$-reduced subgroups of simple algebraic groups (see Theorem \ref{thm:classsimple}) and
Zariski dense $\s$-integral subgroups of almost simple algebraic groups (see Theorem \ref{thm:classalmostsimple}).

We start this section with a collection of technical lemmas on compatibility of quotients, Zariski closure
and base extension in the category of group $k$-$\s$-schemes. This lemmas  are only used in the almost simple case. Then, we study the case of simple groups and conclude with
the almost simple case.

\subsubsection{Preliminary lemmas on  quotients of group $k$-$\s$-schemes}

\par Let  $G$ be a group $k$-$\s$-scheme, $N \trianglelefteq G$ be a normal $\s$-closed subgroup. A quotient of $G$ by $N$ is a morphism $\pi : G \rightarrow G/N$ of group $k$-$\s$-scheme which satisfies the universal property of quotients in the category of group $k$-$\s$-schemes. As shown in \S A.9 of  \cite{articleone}, quotients exist and are unique up to unique isomorphism. Moreover,  we have $\ker(\pi)=N$ and $\pi^*:k\{G/N\} \rightarrow k\{G\}$ is injective.

\begin{lemma}\label{lemma:kernelprop}
Let $k$ be a $\s$-field. Let $\phi: G \rightarrow H$ be a morphism of group $k$-$\s$-schemes.
Then, the induced morphism $G/\ker(\phi) \rightarrow H$ is a $\s$-closed embedding.
\end{lemma}
\begin{proof}
We will use the language of $k$-$\s$-Hopf algebras (Definition A.35 in \cite{articleone}). If $R$ is a Hopf-algebra over $k$, we denote by $R^+$ the kernel of the counit $R\to k$.

Let $N=\ker(\phi)\unlhd G$ and $\phi^*\colon k\{H\}\to k\{G\}$ the dual map. Note that $\I(N)=k\{G\}\phi^*(k\{H\}^+)\subset k\{G\}$.
From Proposition A.42 and Theorem A.43 in \cite{articleone} we know that $k\{G/N\}$ is the unique sub-Hopf algebra of $k\{G\}$ with the property that $k\{G\}k\{G/N\}^+=\I(N)$. But $\phi^*(k\{H\})$ is also a sub-Hopf algebra of $k\{G\}$ and $k\{G\}\phi^*(k\{H\})^+=k\{G\}\phi^*(k\{H\}^+)=\I(N)$.
Therefore $k\{G/N\}=\phi^*(k\{H\})$. Thus $k\{H\}\to k\{G/N\}$ is surjective and $G/N\to H$ is a $\s$-closed embedding.
\end{proof}

% \begin{lemma}\label{lemma:trivkerclosedemb}
% Let $k$ be a $\s$-field. Let $\phi : G \rightarrow H$ be a morphism of group $k$-$\s$-scheme. If the kernel of $\phi$ is trivial (see Remark A.38 in \cite{articleone}), then $\phi$ is a $\s$-closed embedding.
% \end{lemma}
%
% \begin{proof}
%  Let $\phi^* : k\{H\}
% \rightarrow k\{G\}$ be the dual morphism. Since the kernel of $\phi^*$ is a $\s$-Hopf ideal of $k\{H\}$,
% we can replace $H$ by a $\s$-closed subgroup and assume that $\phi^*$ is injective.
% Hence we have to prove that $\phi$ is an isomorphism.
% Then, Corollary A.44 in \cite{articleone} shows that $\phi$ is the quotient of $G$ by $\ker(\phi)=1_G$, i.e., by the trivial subgroup of $G$.
% Since $G/1_G=G$, the morphism $\phi$ is an isomorphism the statement is a $\s$-closed embedding.
% \end{proof}
%
% \begin{lemma}\label{lemma:kernelprop}
% Let $k$ be a $\s$-field. Let $\phi: H \rightarrow G$ be a morphism of group $k$-$\s$-schemes.
% Then, the induced morphism $\overline{\phi} : H/\ker(\phi) \rightarrow G$ is a $\s$-closed embedding.
% \end{lemma}
%
% \begin{proof}
% By Lemma \ref{lemma:trivkerclosedemb}, it is sufficient to remark  that $\overline{\phi} : H/\ker(\phi) \rightarrow G$ has trivial kernel.\end{proof}

Next we want to show that
the functorial construction $[\s]_k$ (see \S \ref{sec:recallarticleone} and Section A.4 in the appendix in \cite{articleone}, for more details)
is compatible with quotients.
First, we prove the following lemma:

\begin{lemma}\label{lemma:injective}
Let $k$ be a $\s$-field and let $\phi : R \rightarrow S$ be a morphism of $k$-algebras.
Let $[\s]_k \phi : [\s]_k R \rightarrow [\s]_k S$ be the morphism of
$k$-$\s$-algebras deduced from $\phi$ by functoriality of
$[\s]_k$. If $\phi$ is injective then $[\s]_k \phi$ is also injective.
\end{lemma}

\begin{proof}
By construction, $[\s]_k \phi$ is the limit of $\phi_d : R_d \rightarrow S_d$ for $d \geq 0$. It is thus sufficient to prove that the $\phi_d$'s are injective for all $d \geq 0$. We prove this by induction on $d$.
For $d =0$, it is true by assumption on  $\phi$. Assume that  $\phi_d : R_d \rightarrow S_d$ is injective.
We prove that $\phi_{d+1}  : R_{d+1} \rightarrow S_{d+1}$
is injective, proving that it is the composition of the following injective morphisms:
$$
\xymatrix{
R_{d+1}=R_d \otimes_k  {}^{\s^{d+1}}R \ar[rrr]^{\phi_d \otimes \id _{{}^{\s^{d+1}}R}} &&&S_d \otimes_k {}^{\s^{d+1}}R
\ar[rrr]^{\id_{S_d} \otimes_k {}^{\s^{d+1}}\phi}&&&S _d \otimes_k{}^{\s^{d+1}}S=S_{d+1}}.
$$
Indeed the flateness of the tensor product over a field immediately implies that
$\phi_d \otimes \id _{{}^{\s^{d+1}}R}$ and $\id_{S_d} \otimes_k {}^{\s^{d+1}}\phi$ are injective. \end{proof}

\begin{lemma}\label{lemma:compquotientandsigma}
Let $k$ be a $\s$-field. Let $H$ be a group $k$-scheme and let $N \trianglelefteq H$ be a normal closed subgroup. Then,
the group $k$-$\s$-scheme $[\s]_k(H/N)$  is the quotient of $[\s]_k  H$ by $[\s]_k N$, i.e.,
$[\s]_k(H/N)$ and $[\s]_k H/ [\s]_k N$ are canonically isomorphic.
\end{lemma}

\begin{proof}
Since $p\colon H \rightarrow H/N$ is a quotient map, the dual morphism $p^*\colon k[H/N] \rightarrow k[H]$ is
injective and $\ker(p)=N$. Because of \eqref{eq:skG(S)}, for any $k$-$\s$-algebra $S$ we have
$$
\ker\Big([\s]_k p\colon [\s]_k H(S) \rightarrow [\s]_k (H/N)(S)\Big)=
\ker\Big(p\colon H(S^\sharp) \rightarrow (H/N)(S^\sharp)\Big)
=N(S^\sharp)=[\s]_k N(S),
$$
where $S^\sharp$ is the underlying $k$-algebra of $S$. I.e., $\ker([\s]_k p)=[\s]_k N$. Moreover,
by Lemma \ref{lemma:injective}, $[\s]_k p^*\colon  [\s]_k k[H/N] \rightarrow [\s]_k k[H]$ is injective.
By Corollary A.44 in \cite{articleone}, $[\s]_kp\colon [\s]_k H \rightarrow [\s]_k(H/N)$ is the quotient of $[\s]_kH$ modulo $[\s]_kN$.\end{proof}

\begin{rem} If there is no possible confusion, we write $H/N$ instead of $[\s]_k (H/N)$.\end{rem}

In the following lemma, we prove the compatibility  between
the $d$-th order Zariski closure and quotients.

\begin{lemma}\label{lemma:zarclosquotient}
Let $H$ be an algebraic group over $k$ and $N \trianglelefteq H$ a normal closed subgroup of $H$.
Let $i\colon G \rightarrow H$ be the inclusion of a $\s$-closed subgroup $G$ of $H$ into $H$,
$G \cap N$ the schematic intersection of $G$ with $N$ and
$p\colon H \rightarrow H/N$ (resp. $\pi: G \rightarrow G/(G\cap N)$) the quotient map. Then,
\begin{itemize}
\item the morphism $p \circ i \colon G \rightarrow H/N$ factors through $\pi$ into a $\s$-closed embedding
$\iota \colon G/(G\cap N) \rightarrow H/N$;
\item for $d \geq 0$, the $d$-th order Zariski closure $G/(G\cap N)[d]$ of $G/(G\cap N)$ inside $H/N$
is a quotient of  the $d$-th order Zariski closure $G[d]$ of $G$ by  $G[d] \cap N_d$. That is, the algebraic groups
 $G/(G\cap N)[d]$ and $G[d] /G[d] \cap N_d$ are canonically isomorphic over $k$.
\end{itemize}
\end{lemma}

\begin{proof}
Since $\ker(p \circ i)= G \cap N$, Lemma \ref{lemma:kernelprop} implies that $p \circ i$ factors through $\pi$ into a $\s$-closed embedding $\iota \colon  G/(G\cap N) \rightarrow H/N$.

Let $d \geq 0$. First of all, we prove that $(H/N)_d$ is a quotient of $H_d$ by $N_d$. Since $p\colon H \rightarrow H/N$
is a quotient map, its dual $p^*\colon  k[H/N] \rightarrow k[H]$  is injective. Following the lines
of the proof of   Lemma \ref{lemma:injective}, we see that the dual map $p_d^*$ of $p_d \colon  H_d \rightarrow (H/N)_d$
is also injective. Moreover, $\ker(p_d)$ is exactly $N_d$. To see this, we remark that
$H_d =H\times {}^\s H \times \dots \times {}^{\s^d} H$,
where ${}^{\s^ i} H = H \times_{\spec(k)} \spec(k)$ and the fiber
product is taken with respect to the morphism
induced from $\s^i \colon k \rightarrow k.$
 Moreover
$p_d$ acts on the $i$-th
component via the morphism  ${}^{\s ^i} p$, obtained from $p$ by extension of scalars  $\s^i \colon k \rightarrow k$. This shows that $\ker (p_d) = N\times {}^\s N \times \dots \times {}^{\s^d} N$.
It follows that $p_d \colon   H_d \rightarrow (H/N)_d$ is the quotient of $H_d$ by $N_d$. (Cf. \cite[Theorem 7.8]{Milne:BasicTheoryOfAffineGroupSchemes}.)
By construction of the $d$-th order Zariski closure, we have closed embeddings $G[d] \hookrightarrow H_d$ and $ G/(G \cap N)[d] \hookrightarrow (H/N)_d$. Since $\ker(p_d)=N_d$, we find that the kernel of $G[d] \rightarrow G/(G\cap N)[d]$ is $G[d]\cap N_d$.

By construction of the $d$-th order Zariski closure, we have
$k\left[G/(G \cap N)[d] \right]=\iota^*(k[H/N]_d)$, $k\left[G[d]\right] =i^*(k[H]_d)$ and
$k\left[G/(G \cap N)[d]\right] \rightarrow k\left[G[d]\right]$ is the restriction of $\pi^*$ to
 $k\left[G/(G \cap N)[d]\right]$.
Since $\pi^*$ is injective, the same holds for $k\left[G/(G \cap N)[d]\right] \rightarrow k\left[G[d]\right]$. We can then conclude that
$G[d] \rightarrow G/(G\cap N)[d]$ is a quotient of $G[d]$ by $G[d]\cap N_d$.\end{proof}

%\textcolor{red}{some remark}{I don't know in fact if $G\cap N[d] $ is equal to $G[d] \cap N_d$. It seems
%to me that in general intersection and Zariski closure don't commute}

\begin{lemma}\label{lemma:quotientbaseextension}
Let  $G$ be a group $k$-$\s$-scheme, $N \trianglelefteq G$ be a normal $\s$-closed subgroup and $k'|k$ be a $\s$-field extension. Let $\pi\colon  G \rightarrow G/N$ be a quotient of $G$ by $N$. Denote by
$G_{k'},N_{k'}$ and $(G/N)_{k'}$ the base  extension to $k'$ of $G,N$ and $G/N$. Then, the group $k'$-$\s$-schemes $(G/N)_{k'}$ and $G_{k'}/N_{k'}$ are isomorphic.

\end{lemma}
\begin{proof}
Since $\ker(\pi)=N$, for every $k'$-$\s$-algebra $S$, viewed as a $k$-$\s$-algebra $S$, the sequence
$N(S) \rightarrow G(S) \rightarrow (G/N)(S)$ of abstract groups is exact. It follows that $\ker(\pi_{k'}) =N_{k'}$ where $\pi_{k'}\colon  G_{k'} \rightarrow (G/N)_{k'}$ is the base  extension of $\pi$. Moreover since $k'|k$ is a faithfully flat extension and $\pi^*\colon k\{G/N\} \rightarrow k\{G\}$ is injective, we get that $\pi_{k'}^*\colon k'\{G/N\} \rightarrow k'\{G\}$ is injective. By Corollary A.44 in \cite{articleone}, we deduce that
$(G/N)_{k'}$ is a quotient of $G_{k'}$ by $N_{k'}$.\end{proof}

\subsubsection{The Zariski dense $\s$-closed subgroups of a simple algebraic group} \label{subsubsec:simple}

In this section, we show that every $\s$-reduced $\s$-closed subgroup of $\Gl_{n,k}$ with simple Zariski closure is conjugated to a $\s^d$-constant subgroup of $\Gl_{n,k}$ for some $d\geq 1$. This is the crucial group theoretic input for the proof of Proposition \ref{prop:simpleintegra}. The result is originally due to Z. Chatzidakis, E. Hrushovski and Y. Peterzil.
A slightly simplified version of their result reads as follows:

\begin{prop}[{Prop. 7.10, p. 309 in \cite{ChatHrusPet}}]
Let $k$ be a $\s$-closed $\s$-field of characteristic zero, $H$ an almost simple algebraic group over $k$, and let $G$ be a Zariski dense definable\footnote{In the language of difference rings.} subgroup of $H(k)$. Then, either $G=H(k)$, or, there exist an isomorphism $f\colon H\to H'$ of algebraic groups and an integer $d\geq1$ such that some subgroup of $f(G)$ of finite index is conjugate to a subgroup of $H'(k^\s)$. If $H$ is defined over the algebraic closure of $k^\s$, then we may take $H=H'$ and $f$ to be conjugation by an element of $H(k)$.
\end{prop}

Unfortunately it seems impossible to apply the above proposition directly in the proof of Proposition \ref{prop:simpleintegra} and Theorem \ref{thm:almostintegra}.
Our versions are Theorem \ref{thm:classsimple} and \ref{thm:classalmostsimple} below. They differ in several aspects:
\begin{enumerate}
\item We prefer to work over an inversive algebraically closed difference field of characteristic zero, rather than over a $\s$-closed $\s$-field.
\item We have to avoid the isomorphism $f$ and the ``finite index''.
\item We are interested in the $\s$-closed subgroups of $H$, rather than in the definable subgroups of $H(k)$.
\item If $H$ is simple, it suffices to assume that $G$ is $\s$-reduced. If $H$ is almost simple, we have to assume that $G$ is $\s$-integral. See Remark \ref{rem:theoremfailsforalmostsimple}.
\end{enumerate}
Because of these differences, we have chosen to include a full proof of our reformulation, even though our proof follows \cite{ChatHrusPet} very closely. The divergence in the formulations is mainly due to (iii). Since ACFA does not (fully) eliminate quantifiers, the definable subgroups, say of $\Gl_n(k)$, $k$ a $\s$-closed $\s$-field, need not be defined by difference polynomials. For example, the group
$$\{g\in k^\times|\ \exists\ h\in k^\times:\ h^2=g,\ \s(h)=h\}\leq\Gl_1(k)$$ does not correspond to a $\s$-closed subgroup of $\Gl_{1,k}$.
On the other hand, the quantifier free definable subgroups of $\Gl_n(k)$, i.e., the subgroups of $\Gl_n(k)$ defined by difference polynomials in the matrix entries, only correspond to perfectly $\s$-reduced $\s$-closed subgroups of $\Gl_{n,k}$.

\vspace{5mm}

In this section we assume that the base field has characteristic zero. Note however, that \cite{ChatHrusPet} also provides a version for positive characteristic. So let $k$ be a field of characteristic zero. By an algebraic group over $k$ we mean an affine group scheme of finite type over $k$. If $H$ is an algebraic group over $k$ and $\wtilde{k}$ a field extension of $k$, we denote by $H_{\wtilde{k}}=H\otimes_k\wtilde{k}$ the algebraic group over $\wtilde{k}$ obtained from $H$ via base extension from $k$ to $\wtilde{k}$. By a representation $V$ of $H$ we mean a finite dimensional $k$-vector space $V$ together with a morphism $H\to \Gl(V)$ of algebraic groups over $k$. (Here, from a formal point of view, $\Gl(V)$ is the functor such that $\Gl(V)(S)=\Gl(V\otimes_k S)$ for every $k$-algebra $S$.)

It is well-known that every simple algebraic group descends to $\mathbb{Q}$. To avoid the isomorphism $f$ in the above proposition we require a slightly more precise statement. We thank Michael Singer for the proof of Lemma \ref{lemma:descentreductive} and Corollary \ref{cor:descentafterconjugation}.
% Let $k$ be an algebraically closed field of characteristic zero. It follows from the classification of the semi-simple algebraic groups ?? that every semi-simple algebraic group over $k$ descends to $\mathbb{Q}$, i.e., is isomorphic (over $k$) to an algebraic group defined over $\mathbb{Q}$.
%
% We first recall some terminology from the theory of algebraic groups (\cite{??}): Let $k$ be a field. By an algebraic group over $k$ we mean a group scheme of finite type over $k$. (Recall that all our schemes are affine.) ?? semi-simple simple

\begin{lemma}\label{lemma:descentreductive}
Let $\mathcal{H}$ be a split and connected reductive algebraic group over $\mathbb{Q}$ and let $k$ be an algebraically closed field of characteristic zero. If $V$ is a representation of $H:=\mathcal{H}_k$, then there exists a representation $\mathcal{V}$ of $\mathcal{H}$ such
that $V\simeq\mathcal{V}\otimes_\Q k$, as representations of $H$.
\end{lemma}
\begin{proof}
Since we are in characteristic zero, every representation of $H$ decomposes as a direct sum of irreducible representations. Therefore, we can assume without loss of generality, that $V$ is irreducible. The irreducible representation of $H$ are classified by the so-called dominant weights. See Part II, Chapter 2 of  \cite{Jantzen:RepresentationsOfAlgebraicGroups}. The claim now follows from Corollary  2.9, Part II, p. 203 in \cite{Jantzen:RepresentationsOfAlgebraicGroups}.
\end{proof}

\begin{cor} \label{cor:descentafterconjugation}
Let $k$ be an algebraically closed field of characteristic zero and $H\leq\Gl_{n,k}$ a connected reductive algebraic group. Then there exists $g\in\Gl_n(k)$ such that $gHg^{-1}\leq\Gl_{n,k}$ is defined over $\mathbb{Q}$.
\end{cor}
\begin{proof}
Since $k$ is algebraically closed, $H$ is split and it follows from the classification (See e.g. Corollary  1.3, p. 411 in \cite{DemazExiste}.) that $H=\mathcal{H}_k$ for some split and connected reductive algebraic group $\mathcal{H}$ over $\mathbb{Q}$. By Lemma \ref{lemma:descentreductive} the given representation $V:=k^n$ of $H$ is of the form $V=\mathcal{V}\otimes_{\Q}k$ for some representation $\mathcal{V}$ of $\mathcal{H}$. Since $V$ is faithful, also $\mathcal{V}$ is faithful and it follows that $H$ is conjugated to $\mathcal{H}_k$.
\end{proof}

%\paragraph{The $\s$-subgroup-scheme of simple algebraic groups}
% \par In \cite[Proposition 7.10]{ChatHrusPet}, the authors prove the difference counterpart of a result obtained by Cassidy (\cite{Cassimpl}) in the differential case.
% They show that if $H$ is an almost simple algebraic group defined over an algebraically closed field $k$ and
% $G$ is a Zariski dense perfectly reduced $\s$-subgroup scheme of $H$ then either $H =G$ or  there exists an isomorphism $f : H \rightarrow H^\p$
% of algebraic groups and integers $m>0 $ and $n$,
% such that some subgroup of $f(G)$ of finite index is conjugate to a subgroup of $H^\p( k^{\s^m Frob^n})$. The framework of Chatzidakis, Hrushovski and Peterzil
% is the theory of ACFA and they work with perfectly $\s$-reduced $\s$-closed subgroup of some $Gl_{n,k}$. We would like to show in this paragraph that
% the arguments of Chatzidakis, Hrushovski and Peterzil easily extend to a schematic approach and that we may weaken the hypothesis of
% perfect $\s$-reduceness. Even if  our arguments follow essentially  the  one used by Chatzidakis, Hrushovski and Peterzil, we would like,  for clarity of exposition and completeness,  still give here our schematic counterpart.  First, we  state some definitions
%
% \begin{defn}[Definition 3.1 in \cite{KowPillsigma}]  A  $\s$-field $(k,\s)$ is linearly $\s$-closed if   any linear difference system with coefficients in $k$
% possess a fundamental system of solutions in $k$. \end{defn}

Let $k$ be an algebraically closed field of characteristic zero. Recall that a connected non-commutative algebraic group $H$ over $k$ is called simple if every proper normal algebraic subgroup of $H$ is trivial.

\begin{prop} \label{prop: classification for simple}
Let $k$ be an algebraically closed, inversive $\s$-field of characteristic zero. Let $H$ be a simple algebraic group over $k$ and let $G$ be a $\s$-reduced, Zariski dense, $\s$-closed subgroup of $H$, properly contained in $H$. Then there exist an integer $m\geq 1$ and an isomorphism $\varphi\colon H\to{^{\s^m}H}$ of algebraic groups such that
$$G(S)=\{h\in H(S)|\ \s^m(h)=\varphi(h)\}$$ for every $k$-$\s$-algebra $S$.
\end{prop}
\begin{proof}
Note that $^{\s^d}H$ is simple for every $d\geq 0$. (Indeed, since $H$ descends to $\mathbb{Q}\subset k^\s$, $H$ and $^{\s^d}H$ are (non-canonically) isomorphic as algebraic groups over $k$.) So $H_d=H\times{^{\s}H}\times\cdots\times{^{\s^d}H}$ is a product of simple algebraic groups. As $G$ is Zariski dense in $H$, we have $G[0]=H_0=H$. Since $G$ is properly contained in $H$, we have $G[d]\subsetneq H_d$ for some $d\geq 1$. Let $m\geq 1$ denote the smallest integer such that $G[m]\subsetneq H_m$. Then $G[m-1]=H_{m-1}$. We will show that the projection $G[m]\to G[m-1]=H_{m-1}$ is an isomorphism (of algebraic groups).

Let $\pi\colon G[m]\to{^{\s^m}H}$ denote the projection onto the last factor of $H_m\supset G[m]$. Let $\I(G)\subset k\{H\}$ denote the defining ideal of $G$. Then $$\I(G)\cap k[H_m]=\I(G[m])\subset k[H_m]=k[H]\otimes_k\cdots\otimes_k k[^{\s^m}H]$$ defines $G[m]\leq H_m$. A non-zero element in the kernel of $\pi^*\colon k[^{\s^m}H]\to k[G[m]]=k[H_m]/\I(G[m])$ corresponds to an element of the form
$1\otimes\cdots\otimes 1\otimes a\in \I(G[m])\subset k[H]\otimes_k\cdots\otimes_k k[^{\s^m}H]$, with $a\in k[^{\s^m}H]\smallsetminus\{0\}$. As $G$ is $\s$-reduced (i.e., $\I(G)$ is reflexive) and $k$ is inversive, such an element would give rise to a non-zero element of $\I(G)\cap k[H_0]=\{0\}$. Thus $\pi^*$ is injective, i.e., $\pi$ is dominant.
%Because a dominant morphism of algebraic groups is surjective () it follows that $\pi$ is surjective.

Let $H'\leq{^{\s^m}H}$ denote the algebraic subgroup of ${^{\s^m}H}$ determined by $G[m]\cap(\{1\}^m\times {^{\s^m}H})=\{1\}^m\times H'\leq H_m$. We will show that $H'$ is normal in $H$. Because $k$ is algebraically closed, it suffices to work with the $k$-rational points. Let $h'\in H'(k)$ and $h\in{^{\s^m}H}(k)$. Because $\pi$ is dominant, $\pi(G[m](k))={^{\s^m}H}(k)$ (Proposition 2.2.5, (ii), p. 26, in \cite{Springer:LinearAlgebraicGroups}). Thus, there is a $g\in H_{m-1}(k)$ such that $(g,h)\in G[m](k)$. It follows that
\[(g,h)(1,h')(g,h)^{-1}=(1,hh'h^{-1})\in G[m](k),\]
and so $H'$ is normal in $^{\s^m}H$. Since $^{\s^m}H$ is simple, we have $H'=1$ or $H'={^{\s^m}H}$. Suppose that $H'={^{\s^m}H}$. As $G[m]\to G[m-1]=H_{m-1}$ is dominant, this implies $G[m]=H_m$; a contradiction.
So $H'=1$. This means that $G[m](k)\to G[m-1](k)$ is injective. Thus, $G[m](k)\to G[m-1](k)$ is actually bijective. Since we are in characteristic zero, this is enough to deduce that $G[m]\to G[m-1]$ is an isomorphism of algebraic groups ( Exercise 5.3.5, p. 87 in \cite{Springer:LinearAlgebraicGroups}). Let
\[\varphi\colon H_{m-1}=G[m-1]\rightarrow G[m]\xrightarrow{\pi}{^{\s^m}}H\]
denote the composition of the inverse of this isomorphism with $\pi$. Then $G[m]\leq H_m$ is the graph of $\varphi$.
I.e.,
\[G[m](k)=\left\{(h,\varphi(h))|\ h\in H_{m-1}(k)\right\}\leq H_m(k).\]
Suppose $H\times \{1\}^{m-1}\subset\ker(\varphi)$. Then
$$\varphi(1,h_1,\ldots,h_{m-1})=\varphi(h_0,h_1,\ldots,h_{m-1})=h_m\in{^{\s^m}H}(k)$$ for all $(h_0,\ldots,h_m)\in G[m](k)\leq H_m(k)$.
For every $f\in k[{^{\s^m}H}]$, the map $$(h_0,\ldots,h_m)\mapsto f(\varphi(1,h_1,\ldots,h_{m-1}))-f(h_m)$$ defines
a regular function $\wtilde{f}\in k[H_m]$. Moreover $\wtilde{f}\in\I(G[m])\subset\I(G)$. We can choose $f\in k[{^{\s^m}H}]$ such that $\wtilde{f}$ is non-zero. Because $G$ is $\s$-reduced and $k$ inversive, we can obtain from $\wtilde{f}$, by considering the inverse image of $\wtilde{f}$ under $\s$, a non-zero element of $\I(G[m-1])=\I(G)\cap k[H_{m-1}]$. But this contradicts the minimality $m$. Therefore $H\times \{1\}^{m-1}\nsubseteq\ker(\varphi)$.
Since $H$ is simple, we have $(H\times \{1\}^{m-1})\cap\ker(\varphi)=\{1\}^{m}$. Thus, $\varphi$ restricts to an isomorphism $\varphi\colon H=H\times\{1\}^{m-1}\to{^{\s^m}}H$ of algebraic groups.

Since the $^{\s^i}H$ are simple, a normal algebraic subgroup $N$ of $H_{m-1}=H\times{^{\s}H}\times\cdots\times{^{\s^{m-1}}H}$ must be a product of some of the factors $^{\s^i}H$.
(This is easy to see: For example, if $(h_0,\ldots,h_{m-1})\in N(k)$ and $h_0\neq 1$, then there exists $h_0'\in H(k)$ with $h_0h_0'\neq h_0'h_0$. Then $(h_0',1,\ldots,1)(h_0,\ldots,h_{m-1})({h_0'}^{-1},1,\ldots,1)(h_0,\ldots,h_{m-1})^{-1}=(h_0'h_0{h_0'}^{-1}{h_0}^{-1},1,\ldots,1)$ is a non-trivial element of $N(k)\cap (H(k)\times\{1\}^{m-1})$. It follows that $H\times\{1\}^{m-1}\subset N$.)

Because $\ker(\varphi)$ is a normal algebraic subgroup of $H_{m-1}$, we must have $\ker(\varphi)=\{1\}\times{^{\s}H}\times\cdots\times{^{\s^{m-1}}H}$.
In summary, we find that
\[G[m](k)=\left\{(h_0,\ldots,h_{m-1},\varphi(h_0))\big| \ h_i\in{^{\s^i}H}(k)\right\}\leq H_m(k).\]
Because $k[G[m-1]]=k[G[m]]=k\{G\}$, the defining ideal $\I(G)$ of $G$ is determined by $\I(G[m])$.
Therefore
$$G(S)=\{h\in H(S)|\ \s^m(h)=\varphi(h)\}$$ for every $k$-$\s$-algebra $S$.
\end{proof}

Recall (Definition \ref{defn:sconstantgroup}) that a $\s$-closed subgroup $G$ of $\Gl_{n,k}$ is called $\s^d$-constant if $\s^d(g)=g$ for every $g\in G(S)\leq\Gl_n(S)$ and all $k$-$\s$-algebras $S$. We say that
$G$ is conjugated to a $\s^d$-constant subgroup of $\Gl_{n,k}$ if there exists
$u \in \Gl_n(k)$ such that $u Gu^{-1}$ is $\s^d$-constant.

\begin{thm} \label{thm:classsimple}
Let $k$ be an algebraically closed, inversive $\s$-field of characteristic zero and let $G$ be a $\s$-reduced, $\s$-closed subgroup of $\Gl_{n,k}$. Assume that the Zariski closure of $G$ in $\Gl_{n,k}$ is a simple algebraic group, properly containing $G$. Then there exist a $\s$-field extension $\wtilde{k}$ of $k$ and an integer $d\geq 1$ such that $G_{\wtilde{k}}$ is conjugated to a $\s^d$-constant subgroup of $\Gl_{n,\wtilde{k}}$. If $k$ is linearly $\s$-closed we can choose $\wtilde{k}=k$.
\end{thm}
\begin{proof}
Let $H\leq\Gl_{n,k}$ denote the Zariski closure of $G$. By Corollary \ref{cor:descentafterconjugation}, there exist an algebraic group $\mathcal{H}\leq\Gl_{n,\mathbb{Q}}$ over $\mathbb{Q}$ and $g\in\Gl_n(k)$ such that $gHg^{-1}=\mathcal{H}_k\leq\Gl_{n,k}$. Therefore, we can assume without loss of generality that $H=\mathcal{H}_k$.
Then, for every $d\geq 1$ the algebraic groups $H$ and $^{\s^d}H$ are equal as subgroups of $\Gl_{n,k}$. So the $\varphi$ in Proposition
\ref{prop: classification for simple} is an automorphism of $H$ (and not merely an isomorphism).

The group $\operatorname{Inn}(H)$ of inner automorphisms of $H$, is of finite index in the group $\operatorname{Aut}(H)$ of all automorphisms of $H$. (See Theorem 27.4, p. 166,
in \cite{Humphline}.) Therefore $\operatorname{Out}(H):=\operatorname{Aut}(H)/\operatorname{Inn}(H)$ is a finite group. If $\psi\colon H \to H$ is an automorphism of $H$, and $d\geq 0$ an integer, we denote by $^{\s^d}\psi\colon{^{\s^d}H}\to{^{\s^d}H}$ the automorphism of ${^{\s^d}H}$ obtained from $\psi$ by extension of scalars via $\s^d\colon k\to k$. Since ${^{\s^d}H}=H$ (as subgroups of $\Gl_{n,k}$) we see that $^{\s^d}\psi$ is actually an automorphism of $H$. We obtain an action $\psi\mapsto{^{\s}\psi}$ of $\s$ on $\operatorname{Aut}(H)$ and an induced action on $\operatorname{Out}(H)$. Since $\operatorname{Out}(H)$ is finite, there exists for every integer $m\geq 1$ and every element $\psi$ of $\operatorname{Out}(H)$ an integer $e\geq 1$ such that ${^{\s^{m(e-1)}}\psi}\cdots{^{\s^m}\psi}\cdot\psi=1\in \operatorname{Out}(H)$. (For example, we can choose $e$ as the product of the length of the orbit of $\psi$ under $\psi\mapsto{^{\s^m}\psi}$ and the size of $\operatorname{Out}(H)$.) With $m$ and $\varphi$ as in Proposition \ref{prop: classification for simple}, we see that there exists an integer $e\geq 1$ such that
${^{\s^{m(e-1)}}\varphi}\cdots{^{\s^m}\varphi}\cdot\varphi$ is the conjugation with an element $h\in H(k)$.

Let $S$ be a $k$-$\s$-algebra and $g\in H(S)$. Because $\s^m(g)=\varphi(g)$ we find that
$$\s^{me}(g)={^{\s^{m(e-1)}}\varphi}(\cdots{^{\s^m}\varphi}(\varphi(g))\cdots)=hgh^{-1}.$$
Set $d=me$ and let $\wtilde{k}$ be a $\s$-field extension of $k$ such that there exists $u\in\Gl_n(\wtilde{k})$ with $\s^{d}(u)=hu$. (For example we can take $\wtilde{k}=k(X)$ with $\s(X)=hX$.) If $k$ is linearly $\s$-closed we can choose $\wtilde{k}=k$ by Lemma \ref{lemma:linearlysclosedstableunderpower}.

Then, for every $\wtilde{k}$-$\s$-algebra $S$ and $g\in G_{\wtilde{k}}(S)=G(S)$ we have
$$\s^d(u^{-1}gu)=\s^d(u)^{-1}\s^{me}(g)\s^d(u)=u^{-1}h^{-1}hgh^{-1}hu=u^{-1}gu.$$
Thus $u^{-1}G_{\wtilde{k}}u\leq\Gl_{n,\wtilde{k}}$ is $\s^d$-constant.
\end{proof}

\begin{rem}
The assumption that $G$ is $\s$-reduced can not be omitted. For example, let $H\leq\Gl_{n,k}$ be a simple algebraic group and $H'$ a connected, non-trivial algebraic subgroup of $H$. Define a $\s$-closed subgroup $G$ of $H$ by
$$G(S)=\{h\in H(S)|\ \s(h)\in{^{\s}H'(S)}\}\leq H(S)$$ for every $k$-$\s$-algebra $S$. Then $G$ is Zariski dense in $H$ but not conjugated to a $\s^d$-constant group.
\end{rem}

\begin{rem} \label{rem:theoremfailsforalmostsimple}
The theorem does not extend to almost simple algebraic groups. For example, let $G$ be a $\s$-closed subgroup of $\Sl_{n,k}$ containing the
center $Z$ of $\Sl_{n,k}$, given by $Z(S)=\{\lambda I|\ \lambda\in S^\times,\ \lambda^n=1\}$ for every $k$-$\s$-algebra $S$.
Suppose there exists an integer $d\geq 1$ and a $\s$-field extension $\wtilde{k}$ of $k$, such that $G$ is conjugated over $\wtilde{k}$ to a $\s^d$-constant subgroup of $\Gl_{n,\wtilde{k}}$. Since $Z_{\wtilde{k}}$ is in the center of $\Gl_{n,\wtilde{k}}$, this would imply that $Z_{\wtilde{k}}\le\Gl_{n,\wtilde{k}}$ is $\s^d$-constant. Obviously this is not the case: Despite the fact that $\s^n(\lambda)=\lambda$ for every $\lambda$ in a \emph{$\s$-field} extension of $k$ with $\lambda^n=1$, this property is not retained if we choose $\lambda$ from an arbitrary $k$-$\s$-algebra.
\end{rem}

\begin{rem}\label{rem:placeofu}
In the notations of the proof of Theorem \ref{thm:classsimple}, there exists a $\s$-field
extension $\wtilde{k}$ of $k$ and an element $u \in \Gl_n(\wtilde{k})$ such that
$u^{-1}G_{\wtilde{k}}u\leq\Gl_{n,\wtilde{k}}$ is $\s^d$-constant. Assume that $H$ is defined over $\Q$.
One can show that a convenient choice of the $\s$-field
extension $\wtilde{k}$ allows us to choose $u$ in $H(\wtilde{k})$. In other words, there exists
a $\s$-field extension $\wtilde{k}$ of $k$ such that $G_{\wtilde{k}}$ is conjugated inside
$H_{\wtilde{k}}$ to a $\s^d$-constant group.
\end{rem}

\begin{proof}
 The proof is as follows. Going back to the proof
of Theorem \ref{thm:classsimple}, we need to show that, for a given $h \in H(k)$, one can
find a $\s$-field extension $\wtilde{k}$ of $k$ and $u \in H(\wtilde{k})$ such that $\s^d(u)=hu$.
% First of all, as in Corollary \ref{cor:descentafterconjugation}, one can assume that $H$ is defined over $\Q \subset k^\sigma$.
Let $\wtilde{k}$ be the function field
of $H^d=H\times \dots \times H$, i.e., $d$ copies of $H$. One can write $\wtilde{k}=k(X_1,\dots,X_d)$ where $X_i$ is
an $n\times n$-matrix of coordinates on the $i$-th copy of $H$. Since the vanishing ideal $\I(H) \subset
k[X,\frac{1}{\det(X)}]=k[\Gl_n]$ is defined over $\Q$, the action of $\s$ on the coefficients of the
 polynomial ring  $k[X,\frac{1}{\det(X)}]$ stabilizes $\I(H)$. This allows us to define a structure of
$k$-$\s$-field extension on $\wtilde{k}$ via
$$
\s(X_1)=X_2,\s(X_2)=X_3,\dots,\s(X_d)=hX_1.
$$
Then, $X_1$ is an element of $H(\wtilde{k})$ and satisfies $\s^d(X_1)=hX_1$ by construction.\end{proof}

\subsubsection{The case of almost simple groups}

%
%
%
% given by
% $$G(S)=\left\{g\in\Sl_n(S)|\ \exists\ \lambda\in S^\times:\ g\s(g)^{-1}=\lambda I,\ \lambda^n=1\right\}\leq\Sl_n(S).$$
% Here $I$ denotes the identity matrix of size $n$.
% Then $G$ contains $\Sl_{n,k}^\s$ and the center $Z$ of $\Sl_{n,k}$, given by $Z(S)=\{\lambda I|\ \lambda\in S^\times,\ \lambda^n=1\}$ for every $k$-$\s$-algebra $S$.

Let $k$ be an algebraically closed field of characteristic zero. Recall that a connected non-commutative algebraic group $H$ over $k$ is called almost simple if every proper connected normal closed subgroup is trivial.
In remark \ref{rem:theoremfailsforalmostsimple}, we have seen that a proper Zariski dense $\s$-closed subgroup of $H$ is not necessarily $\s^d$-constant.
However, the aim of this section is to show that Theorem \ref{thm:classsimple} still holds for a proper, \textit{$\s$-integral}, Zariski dense $\s$-closed subgroup $G$ of $H$.

We start with some remarks on almost simple algebraic groups, their center and their adjoint group.
Let $H$ be an almost simple algebraic group. By Corollary 1.3, p. 411 in \cite{DemazExiste}, the algebraic group $H$ is defined over $\Q$, i.e., is equal to $\mathcal{H}_k$ for some split almost simple algebraic group $\mathcal{H}$
over $\Q$. Since we are in characteristic zero, the schematic center $\mathcal{N}$ of $\mathcal{H}$
is representable by a closed subscheme of $\mathcal{H}$ (see Expos\'{e} XI, 6.11, in \cite{SGA3II}). Thus, we
get that, in our situation, the formation of schematic center commutes with base extension.
Precisely, $\mathcal{N}_k=N$. Now, by Expos\'{e} XXII, 4.3 in  \cite{SGA3III}, the quotient $\mathcal{H}/\mathcal{N}$
is representable by an affine group scheme over $\Q$. Since formation of quotients commute with flat base
extension (see Expos\'{e} VIIa, 4.6 of \cite{SGA3I}), the $k$-scheme $(\mathcal{H}/\mathcal{N})_k$
is a quotient of $H$ by $N$. In conclusion, given an almost simple algebraic group
$H$ over $k$, we see that its center $N$ is defined over $\Q$ as well as the quotient map
$p: H \rightarrow H/N$ and the adjoint group $H/N$.

Before proving our main result on the $\s$-closed subgroups of almost simple algebraic groups, we need one more lemma.

\begin{lemma}\label{lemma:propsubgroup}
Let $k$ be an algebraically closed field of characteristic zero,
$H$ an almost simple algebraic group over $k$ and $N \trianglelefteq H$ the center  of $H$.
Let $i : G \rightarrow H$ be the inclusion  of  a proper
$\s$-closed subgroup $G$ of $H$ into $H$ and $p: H \rightarrow H/N$ (resp. $\pi: G \rightarrow G/(G\cap N)$) the quotient map. Then,
\begin{itemize}
\item the morphism $p \circ i : G \rightarrow H/N$ factors through $\pi$ into a $\s$-closed embedding
$\iota : G/(G\cap N) \rightarrow H/N$;
\item the $\s$-closed subgroup $G/(G\cap N)$ is a proper $\s$-closed subgroup of $H/N$.
\end{itemize}
\end{lemma}

\begin{proof}
The first assertion  comes from Lemma \ref{lemma:zarclosquotient}. Suppose that $G/(G\cap N)=H/N$. Since $G$ is a proper $\s$-closed subgroup of $H$, there
exists $m \geq 1$ such that $G[m] \subsetneq H_m$. Now, by Lemma \ref{lemma:zarclosquotient},
we have $G/(G\cap N)[m]=G[m]/G[m]\cap N_m$ and $(H/N)_m =H_m/N_m$. Then %by \S 16, exercice 6 in \cite{Waterhouse:IntrotoAffineGroupSchemes},
$$
\dim (G[m]/G[m]\cap N_m)= \dim (G[m])- \dim (G[m]\cap N_m) \leq \dim (G[m])
$$
and $\dim (H_m/N_m)=\dim (H_m)- \dim (N_m)=\dim (H_m)$ (since $N_m$ is a finite algebraic group).
Therefore
$$
\dim (G[m]/G[m]\cap N_m) \leq \dim (G[m]) < \dim (H_m) = \dim (H_m/N_m),
$$
which contradicts the assumption $G/(G\cap N)= H/N$.\end{proof}

Recall that a $\s$-algebraic group $G$ is called $\s$-integral if $k\{G\}$ is a $\s$-domain, i.e., $k\{G\}$ is an integral domain and $\s\colon k\{G\}\to k\{G\}$ is injective.

\begin{thm} \label{thm:classalmostsimple}
Let $k$ be an algebraically closed, inversive $\s$-field of characteristic zero and let $G$ be a $\s$-integral, $\s$-closed subgroup of $\Gl_{n,k}$. Assume that the Zariski closure of $G$ in $\Gl_{n,k}$ is an almost simple algebraic group, properly containing $G$. Then there exist a $\s$-field extension $\wtilde{k}$ of $k$ and an integer $d\geq 1$ such that $G_{\wtilde{k}}$ is conjugated to a $\s^d$-constant subgroup of $\Gl_{n,\wtilde{k}}$.
\end{thm}
\begin{proof}
Let $H\leq\Gl_{n,k}$ denote the Zariski closure of $G$, $i\colon G \rightarrow H$ be  the inclusion of $G$ into  $H$,
$N$ the center of $H$ and $p\colon H \rightarrow H/N$ (resp. $\pi\colon G \rightarrow G/(G\cap N)$) the quotient map.
By Corollary \ref{cor:descentafterconjugation},
we can assume that $H$ is  over $\Q$.  By the discussion above, we get that the finite center $N$  of $H$ as well as the quotient map $p\colon H \rightarrow H/N$ are  over $\Q$.
\par
By Lemma \ref{lemma:zarclosquotient}, the morphism $G \rightarrow H/N $ factors into a $\s$-closed embedding $\iota \colon  G/(G\cap N) \rightarrow H/N$. Moreover, since $G$ is Zariski dense in $H$, we get that $G/(G\cap N)$ is Zariski dense in  $H/N$. By Lemma \ref{lemma:propsubgroup}, $G/G\cap N$ is a proper $\s$-closed subgroup of $H/N$. Finally, by Theorem A.43 in  \cite{articleone}, one can identify the $k$-$\s$-Hopf algebra  of $G/(G\cap N)$ with a sub-Hopf algebra of $k\{G\}$. Since $k\{G\}$ is $\s$-reduced, the same holds for $k\{G/(G\cap N)\}$.
%Moreover the fact that $k$ is inversive implies
%that $G/(G\cap N)$ is $\s$-reduced (see Corollary A.19 in \cite{articleone}).
%\par

Embed $H/N$ into some $\Gl_{m,k}$ such that $H/N\leq\Gl_{m,k}$ is defined over $\Q$. We can then apply Theorem \ref{thm:classsimple} and Remark \ref{rem:placeofu} to the proper Zariski
dense $\s$-closed subgroup $G/(G\cap N)$ of the simple algebraic group $H/N$. Thereby, there exist
a $\s$-field extension $\tilde{k}$ of $k$, an integer $d \geq 1$ and $u \in H/N(\tilde{k})$ such
that $u^{-1} (G/G\cap N)_{\tilde{k}} u$ is a $\s^d$-constant subgroup of $\Gl_{m,\tilde{k}}$.
By Lemma \ref{lemma:quotientbaseextension}, we can exchange base field extension and quotients,
which implies that  $u^{-1} G_{\tilde{k}}/(G\cap N)_{\tilde{k}} u$ is a $\s^d$-constant subgroup of some $\Gl_{m,\tilde{k}}$.
\par
Now, up to enlarging $\wtilde{k}$ to its algebraic closure, we can find an element $v$ of $ H(\tilde{k})$
 such that $p(v)=u$. Since $p\colon  H \rightarrow H/N$ is defined by polynomials over $\Q$, we have $p(\s^d(v))=\s^d(u)$.
Then, let    $\phi \colon  G_{\tilde{k}}  \rightarrow \Gl_{n,\tilde{k}}$ be  the $\s$-morphism defined by
$\phi(g) =\s^d(v^{-1}gv)(v^{-1}gv)^{-1}$ for all $ g \in  G_{\tilde{k}} (S)$ and for every $\tilde{k}$-$\s$-algebra $S$. By the above, $\phi$ maps $G_{\tilde{k}} $ into $N_{\tilde{k}}$.  Since $k$ is algebraically closed and inversive, Corollary A.14 in \cite{articleone} implies that  that $\tilde{k}\{G\}=k\{G\}\otimes_k\tilde{k}$ is also a $\s$-domain.
Since $N$ is a finite algebraic group, there exists a set of orthogonal idempotents $e_1,\dots,e_r \in k[N]$, such that
$k[N] = ke_1\oplus \dots \oplus ke_r$. Then, $\tilde{k}\{N\}$ is generated as $\tilde{k}$-$\s$-algebra by $e_1,\dots,e_r$. Since $\tilde{k}\{ G\}$ is a domain, $e_1+\dots + e_r =1$ and $e_ie_j =0$ for $i\neq j$,  there exists  $1\leq i_0 \leq r$ such that $\phi^* (e_{i_0})=1$ and $\phi^*(e_j)=0$ for all $j \neq i_0$. This implies that $\phi^*(\tilde{k}\{N\}) \subset \tilde{k}$ and $\phi$ is a constant morphism. Considering  the image  of the identity element of $G_{\tilde{k}}$ under $\phi$, we see that $v^{-1}G_{\tilde{k}} v$ is a $\s^d$-constant subgroup of $\Gl_{n,\tilde{k}}$.
\end{proof}

%%%%%%%%%%%%%%%%%%%%%%%%%%%%%%%%%%%%%%%%%%%%%%%%%%%%%%%%%%%%%%%%%%%%%%%%%%%%%%%%%%%%%%%%%%%%%%%%
%%%%%%%%%%%%%%%%%%%%%%%%%%%%%%%%%%%%%%%%%%%%%%%%%%%%%%%%%%%%%%%%%%%%%%%%%%%%%%%%%%%%%%%%%%%%%%%%
\def\cprime{$'$}

%
%\part{Dépotoir}
%\input{depotoir}

\end{document}